\documentclass[a4paper,11pt]{article}
\usepackage{graphicx}        
\usepackage{multicol}        
\usepackage[bottom]{footmisc}
\usepackage{epsfig}
\usepackage{authblk}
\usepackage{graphicx}
\usepackage{float}
\usepackage{diagbox}
\usepackage{nicefrac}
\usepackage[applemac]{inputenc}
\usepackage[T1]{fontenc}
\usepackage{blkarray}
\usepackage{mathtools,amsthm,amsfonts}

\usepackage{newtxtext}
\usepackage[upint,varg]{newtxmath}

\usepackage{color}

\usepackage{bm}
\usepackage{nicefrac}
\usepackage[hmargin={15mm,15mm},vmargin={17mm,22mm}]{geometry}
\usepackage{array}
\usepackage{pdfpages}
\usepackage{subfig}
\usepackage{multirow}
\usepackage{enumitem}
\usepackage{bbold}
\usepackage{xparse}								
\usepackage{stmaryrd}							
\usepackage{sidecap}								
\usepackage{upgreek}
\usepackage{accents}
\allowdisplaybreaks
\usepackage{tcolorbox}
\usepackage{algorithmic}
\usepackage[ocgcolorlinks,allcolors={blue}]{hyperref}

\usepackage[bbgreekl]{mathbbol}

\DeclareSymbolFontAlphabet{\mathbb}{AMSb}
\DeclareSymbolFontAlphabet{\mathbbl}{bbold}

\def\N{\mathbb N}
\def\R{\mathbb R}

\def\P{\mathbb P}

\def\D{\mathcal{D}}
\def\S{\mathcal{S}}

\def\<{\langle}
\def\>{\rangle}

\def\wt{\widetilde}
\def\Chi{\raise .3ex \hbox{\large $\chi$}}

\def\ul{\underline}

\newcommand{\HGbracket}[2]{\langle #1,#2\rangle_{\Gamma}}

\def\[{\Bigl [}
\def\]{\Bigr ]}
\def\({\Bigl (}
\def\){\Bigr )}
\def\l{\iota}

\def\dsp{\displaystyle}
\def\x{{\bf x}}

\def\n{{\bf n}}

\def\U{{\bf U}}

\def\q{{\bf q}}

\def\g{{\mathbf{g}}}
\def\aa{{\mathfrak a}}

\def\faces{\mathcal{F}}

\def\x{{\bf x}}
\def\dsp{{\displaystyle x}}
\def\d{{\rm d}}
\def\div{{\rm div}}
\def\a{\alpha}
\def\q{\mathbf{q}}
\def\K{\mathbb{K}}
\def\n{\mathbf{n}}
\def\dsp{\displaystyle}
\def\bu{\mathbf{u}}
\def\bv{\mathbf{v}}

\def\O{\Omega}
\def\G{\Gamma}

\def\bbsig{\bbsigma}
\def\bbeps{\bbespilon}
\newcommand{\bb}{\mathbb}
\newcommand{\cc}{\mathcal}
\newcommand{\jump}[1]{\llbracket #1 \rrbracket}

\newcommand{\dtk}{\delta t^{k+\frac{1}{2}}}
\newcommand{\dtkmun}{\delta t^{k-\frac{1}{2}}}

\newcommand{\dtzero}{\delta t^{\frac{1}{2}}}

\newcommand{\nf}{\nicefrac}

\newcommand{\email}[1]{\href{mailto:#1}{#1}}

\begingroup
    \makeatletter
    \@for\theoremstyle:=definition,remark,plain\do{%
        \expandafter\g@addto@macro\csname th@\theoremstyle\endcsname{%
            \addtolength\thm@preskip\parskip
            }%
        }
\endgroup

\newtheorem{theorem}{Theorem}
\numberwithin{theorem}{section}

\newtheorem{lemma}[theorem]{Lemma}

\theoremstyle{remark}
\newtheorem{remark}[theorem]{Remark}
\theoremstyle{definition}

\newenvironment{acknowledgements}
      {\bigskip\bigskip~\newline\textbf{Acknowledgements}}

\def\g{{\rm nw}}
\def\l{{\rm w}}
\def\rt{{\rm rt}}

\def\trace{\gamma}
\def\grad{\nabla}
\def\del{\partial}
\def\div{{\rm div}}

\definecolor{labelkey}{rgb}{0.6,0,1}

\usepackage[normalem]{ulem}
 \newcounter{corr}
 \definecolor{violet}{rgb}{0.580,0.,0.827}
 \newcommand{\corr}[3]{\typeout{Warning : a correction remains in page
 \thepage}
 				\stepcounter{corr}        
 				{\color{blue}\ifmmode\text{\,\sout{\ensuremath{#1}}\,}\else\sout{#1}\fi}
         {\color{red}#2}
         {\color{violet} \ifmmode\text{#3}\else #3\fi }}

\usepackage{parskip}
\usepackage{algorithmic}

\usepackage{diagbox}
\usepackage{booktabs}

\makeatletter
\def\thm@space@setup{%
  \thm@preskip=\parskip \thm@postskip=0pt
}
\makeatother

\newcommand{\doi}[1]{\href{https://dx.doi.org/#1}{DOI:#1}}

\begin{document}
\title{Energy-stable discretization of two-phase flows in deformable porous media with frictional contact at matrix--fracture interfaces 
}
\author[1,2]{{Francesco Bonaldi}\footnote{Corresponding author, \email{francesco.bonaldi@umontpellier.fr}}}
\author[3]{{J\'er\^ome Droniou}\footnote{\email{jerome.droniou@monash.edu}}}
\author[1]{{Roland Masson}\footnote{\email{roland.masson@univ-cotedazur.fr}}}
\author[4]{{Antoine Pasteau}\footnote{\email{antoine.pasteau@andra.fr}}}
\affil[1]{Universit\'e C\^ote d'Azur, Inria, CNRS, Laboratoire J.A. Dieudonn\'e, team Coffee, Nice, France}%
\affil[2]{IMAG, Univ Montpellier, CNRS, Montpellier, France}%
\affil[3]{School of Mathematics, Monash University, Victoria 3800, Australia}%
\affil[4]{Andra, Chatenay-Malabry, France}%
\date{}
\maketitle
\vspace{-.5cm}
{\footnotesize{
\centering
\textbf{Highlights}
\vspace{-.2cm}
\begin{itemize}
  %
\item Energy-stable model and scheme for coupled two-phase flow and contact mechanics in discrete fracture networks  
\item Mechanics conforming discretization coupled with $\bb P_0$ Lagrange multipliers to circumvent singularities, local contact equations
\item Flow discretization in the abstract gradient discretization framework accounting for a large family of schemes 
\item Investigation of nonlinear algorithms both for the contact mechanics and for the fully coupled problem 
\item Validation on benchmark 2D examples and application to a realistic axisymmetric case study
\end{itemize}
}}

\begin{abstract}
\noindent
We address the discretization of two-phase Darcy flows in a fractured and deformable porous medium, including frictional contact between the matrix--fracture interfaces. Fractures are described as a network of planar surfaces leading to the so-called mixed- or hybrid-dimensional models. Small displacements and a linear elastic behavior are considered for the matrix. Phase pressures are supposed to be discontinuous at matrix--fracture interfaces, as they provide a better accuracy than continuous pressure models even for high fracture permeabilities.

The general gradient discretization framework~\cite{gdm} is employed for the numerical analysis, allowing for a generic stability analysis and including several conforming and nonconforming discretizations. We establish energy estimates for the discretization, and prove existence of a solution. To simulate the coupled model, we employ a Two-Point Flux Approximation (TPFA) finite volume scheme for the flow and second-order ($\P_2$) finite elements for the mechanical displacement coupled with face-wise constant ($\bb P_0$) Lagrange multipliers on fractures, representing normal and tangential stresses, to discretize the frictional contact conditions. This choice allows to circumvent possible singularities at tips, corners, and intersections between fractures, and provides a local expression of the contact conditions.
We present numerical simulations of two benchmark examples and one realistic test case based on a drying model in a radioactive waste geological storage structure.
\bigskip \\
\textbf{MSC2010:} 65M12, 76S05, 74B10, 74M15\medskip\\
\textbf{Keywords:} poromechanics, discrete fracture matrix models, contact mechanics, two-phase Darcy flows, discontinuous pressure model, Gradient Discretization Method, non-smooth Newton.
\end{abstract}
%
\section{Introduction}
This work deals with the discretization and simulation of processes coupling two-phase Darcy flows in a fractured and deformable porous medium, the mechanical deformation of the matrix domain surrounding the fractures, and the mechanical behavior of the fractures. Such coupled models are of paramount importance in a broad range of subsurface processes and engineering contexts, whereof we provide hereinafter a non-exhaustive list. In the so-called Enhanced Geothermal Systems, rock permeability is increased by reactivating fractures thanks to hydraulic injection. In the context of CO$_2$ sequestration, it is necessary to verify that injection of CO$_2$ in the subsoil does not trigger the reactivation of sealed faults, in order to guarantee the storage integrity. An important aspect of water management, as well as oil and gas recovery, is the evaluation of the potential impact of fluid depletion, which can trigger fault slip and induce seismicity of human origin. In all of these processes, depending on the kind of exploitation considered, the flow can be characterized by a single phase or by two phases. As a last example, in radioactive waste storage facilities, which constitute the main application of this work, the excavation of operating tunnels generates fracture networks in the so-called Excavation Damage Zone (EDZ). These fractures have a strong impact on the desaturation of the EDZ in the exploitation time scale of the facility of, say, 200 years. High capillary pressures induce a contraction of the pores as well as an extension of the fracture apertures which retroactively impact the desaturation, leading to highly coupled poromechanical processes.

This work focuses on pre-existing fractures or faults, i.e., fracture generation and propagation are not addressed. It considers large-scale fractures represented as a network of codimension-one planar surfaces including intersecting, immersed, and non-immersed fractures coupled with the surrounding matrix. For computational complexity reasons related to such physical models, small-scale fractures are rather included using homogenization techniques. 
Modeling and numerical simulation of poromechanical processes on such so-called mixed- or hybrid-dimensional geometries have been drawing a remarkable attention over the last years, and many approaches have been developed in the literature. We recall here briefly some of the key recent contributions, mostly in the context of single-phase flows. Franceschini \emph{et al.}~\cite{tchelepi-castelletto-2020} used a low-order displacement--Lagrange multiplier--pressure formulation for quasi-static contact mechanics coupled with fracture fluid flow; since this formulation is not uniformly inf-sup stable, an algebraic macroelement-based stabilization technique is implemented. Berge \emph{et al.}~\cite{contact-norvegiens} employed a finite volume method both for the flow and the mechanical deformation, combined with a variationally-consistent mixed discretization of contact mechanics based on face-wise constant Lagrange multipliers accounting for the contact tractions. Stefansson \emph{et al.}~\cite{thm-bergen} extended the coupling by also including thermal effects, giving rise to a Thermo-Hydro-Mechanical coupled system. Garipov \emph{et al.}~\cite{GKT16} for the isothermal case and Garipov and Hui~\cite{GH19} for the non isothermal case use as well a finite-volume approximation for the mass (and energy balance in~\cite{GH19}) equations, combined with a Galerkin finite element approximation for the rock mechanics; a penalty method is employed to enforce contact conditions on fractures in~\cite{GKT16} while a Nitsche consistent penalization is used in~\cite{GH19}.

Concerning the contact mechanics problem alone, it has been the subject of an extensive literature; we cite here in particular the key monographs by Kikuchi and Oden~\cite{oden-contact} and Wriggers~\cite{wriggers2006} as well as the review by Wohlmuth~\cite{Wohlmuth11} on variationally consistent mixed formulations based on Lagrange multipliers. In particular, Ben Belgacem and Renard~\cite{BR2003} studied three mixed linear finite element methods for a frictionless contact problem, i.e., the so-called Signorini problem. Hild \emph{et al.}~\cite{HR2007} provided an error estimate for the {C}oulomb frictional contact problem using finite elements in mixed formulation. Chouly \emph{et al.}~\cite{Chouly2017} provided an extensive overview of Nitsche's method for contact problems and, more recently, proved existence results for the static and dynamic finite element formulations of the problem with Coulomb friction~\cite{CHLR2020}.

Regarding the Darcy flow alone, hybrid-dimensional, also termed Discrete Fracture Matrix (DFM) models have been the object of a considerable amount of works since the last twenty years. To mention a few, let us refer to~\cite{MAE02,FNFM03,KDA04,MJE05,ABH09,TFGCH12,SBN12,BGGLM16,BHMS2016,BNY18,FLEMISCH2018239,CDF2018,NBFK2019,BERRE2021103759,ANTONIETTI2021} for single-phase Darcy flows and to~\cite{BMTA03,RJBH06,MF07,Jaffre11,BGGM2015,DHM16,AHMED201749,BHMS2018,gem.aghili,AGHILI2021110452} for two-phase Darcy flows.  

Despite the abundance of contributions concerning single-phase flows in literature, not much work has been specifically devoted to the mathematical modeling of two-phase fractured poromechanics. Let us cite here~\cite{JHA2014} for compressible two-phase flow with contact mechanics for CO$_2$ sequestration applications, and~\cite{MRAFH11} for the case of open fractures (i.e., whose interfaces are not in contact). In our contribution, the extension of fractured poromechanical models to unsaturated flows is guided by the stability of the coupled model in the energy norm, which requires a careful definition of the coupling terms both on the matrix and fracture sides. The notion of equivalent pressure as a convex combination of the phase pressures is the key ingredient for this extension, and its choice both in the matrix and in the fracture network must be consistent with the definition of the fluid mass content to guarantee its stability through energy estimates. Two-phase flow models also involve additional nonlinearities which require advanced nonlinear algorithms to couple the flow and mechanical problems.

In this work, the two-phase flow hybrid-dimensional model accounts for discontinuous pressures and saturations at matrix--fracture interfaces~\cite{BHMS2018,DHM16,gem.aghili}. It is coupled with the deformation of the matrix using a poroelastic model based on the concept of equivalent pressure introduced by Coussy in~\cite{coussy} and taking into account the capillary energy contribution  (see also~\cite{KTJ2013,JHA2014}). The fracture mechanical model accounts for the contact conditions at both sides of the fractures using Coulomb's frictional model and again the concept of equivalent pressure in the fractures. It involves a strong poromechanical coupling through matrix porosity, fracture conductivity and aperture, and matrix and fracture equivalent pressures. The definition of the coupling terms, namely the matrix porosity and fracture aperture on the mechanical side, and the matrix and fracture equivalent pressures on the flow side, ensures the stability of the coupled model in a suitable energy norm. A variant of this model, based on a partial linearization of the fluid mass content using the initial porosity, is also briefly discussed. This slightly different model requires to modify the definition of the matrix equivalent pressure to preserve the energy estimate of the coupled model. This illustrates the key modeling feature of a consistent definition of both coupling terms for two-phase poromechanical models.

 The discretization of the two-phase flow model is based  on the abstract Gradient Discretization (GD) framework~\cite{gdm}, which encompasses a large family of both conforming and nonconforming schemes. The frictional contact mechanics is discretized with a conforming scheme of the displacement field combined with a variationally consistent mixed formulation based on $\mathbb{P}_0$ Lagrange multipliers accounting for the contact tractions. Assuming the coercivity of the flow GD and the inf-sup condition for the contact tractions--displacement spaces, we infer energy estimates for the fully coupled discrete problem, thereby proving the stability of its formulation, as well as the existence of a solution for the discretization of the modified two-phase poromechanical model.
 
 The numerical simulations are based on a typical example fitting this framework, using a cell-centered finite volume scheme for the flow, the Two-Point Flux Approximation (TPFA), and a $\bb P_2$--$\bb P_0$ mixed formulation in terms of displacement and contact tractions for the mechanics. This discretization provides several advantages:~(i) the method is variationally consistent for the frictional contact problem; (ii)~no further parameters like in penalty or Nitsche's method are required; (iii)~it is inf-sup stable in the displacement--contact tractions variables; (iv)~it readily deals with fracture networks including tips, corners or intersections; (v)~it provides a local expression of the contact conditions, which gives access to efficient nonlinear solvers based on a non-smooth Newton formulation; (vi)~it yields the possibility to locally eliminate the contact--traction unknowns, provided that the displacement space contains a bubble function for each fracture face; (vii)~it features pressure--displacement stability in the limit of fluid incompressibility and small time steps.
 
A Newton--Raphson algorithm is used to solve the two-phase flow problem combined with a local nonlinear interface solver to eliminate the matrix--fracture interface pressure unknowns as in~\cite{gem.aghili}. We take account of the frictional contact conditions, as in~\cite{contact-norvegiens}, by resorting to suitable complementarity functions leading to a non-smooth Newton nonlinear solver for the contact mechanics. This non-smooth Newton algorithm is compared with an active set counterpart based on simplified equations for each contact state (open, stick or slip). The fully coupled nonlinear problem of two-phase flow and contact mechanics is iteratively solved using a fixed-point method formulated on the equivalent pressures and combined with a Newton--Krylov acceleration algorithm.  The numerical section also investigates the performance of this Newton--Krylov algorithm and of the non-smooth Newton and active set methods.

The  paper is structured as follows. In Section~\ref{sec:notation} we introduce the notation and geometry assumptions employed throughout this work. In Section~\ref{sec:modeleCont} we give a presentation of the continuous problem in its strong formulation, and focus on the Lagrange-multiplier formulation of the contact conditions. Section~\ref{sec:gradientscheme} presents the general gradient discretization framework we adopt to introduce the discrete counterpart of the coupled problem, focusing in particular on the local formulation of the Coulomb frictional conditions, and contains the main theoretical results, concerning energy estimates and existence of a solution for the discrete problem. In Section~\ref{num.experiments}, we present three numerical experiments, the first two to validate the contact mechanics discretization itself, and the last one to simulate the coupling with a two-phase flow occurring in a drying model of a low-permeability medium by suction at the interface with a ventilation tunnel. The data set of this last test case is based on the Callovo--Oxfordian argilite rock properties of the radioactive waste storage prototype facility of Andra. Finally, in Section~\ref{sec:conclusions} we draw some conclusions and outline potential perspectives for future work.

\section{General notation and assumptions}\label{sec:notation}
In what follows, scalar fields are represented by lightface letters, vector fields by boldface letters. We use the overline notation $\bar v$ to distinguish an exact (scalar or vector) field from its discrete counterpart $v$.
We let $\Omega\subset\R^d$, $d\in\{2,3\}$, denote a bounded polytopal domain, partitioned
into a fracture domain $\Gamma$ and a matrix domain $\O\backslash\overline\G$.
The network of fractures is defined by 
$$
\overline \Gamma = \bigcup_{i\in I} \overline \Gamma_i,
$$  
where each fracture $\Gamma_i\subset \Omega$, $i\in I$ is a planar polygonal simply connected open domain. Without restriction of generality, we will assume that the fractures may only intersect at their boundaries (Figure \ref{fig_network}), that is, for any $i,j \in I, i\neq j$ it holds $\Gamma_i\cap \Gamma_j = \emptyset$, but not necessarily $\overline{\Gamma}_i\cap \overline{\Gamma}_j = \emptyset$.

\begin{figure}[h!]
\begin{center}
\includegraphics[scale=.55]{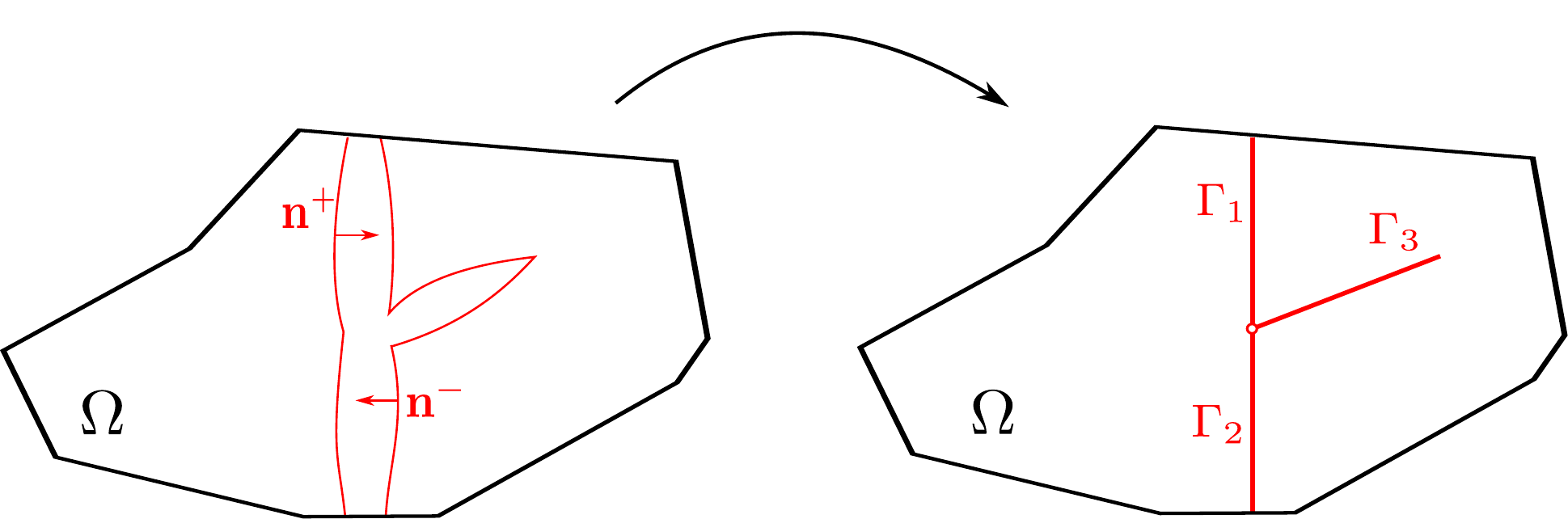}   
\caption{Illustration of the dimension reduction in the fracture width for a 2D domain $\Omega$ with three intersecting fractures $\Gamma_i$, $i\in\{1,2,3\}$, with the equi-dimensional geometry on the left and the mixed-dimensional geometry on the right.}
\label{fig_network}
\end{center}
\end{figure}

The two sides of a given fracture of $\Gamma$ are denoted by $\pm$ in the matrix domain, with unit normal vectors $\n^\pm$ oriented outward from the sides $\pm$. We denote by $\gamma_\aa$ the trace operators on the side $\aa \in \{+,-\}$ of $\Gamma$ for functions in $H^1(\Omega{\setminus}\overline\Gamma)$ and by $\gamma_{\del\Omega}$ the trace operator for the same functions on $\del\O$. The jump operator on $\Gamma$ for functions $\bar\bu$ in  $(H^1(\O\backslash\overline\Gamma))^d$ is defined by
$$
\jump{\bar\bu} = {\gamma_+\bar\bu - \gamma_-\bar\bu}, 
$$
and we denote by
$$
\jump{\bar\bu}_n = \jump{\bar\bu}\cdot\n^+ \quad \mbox{ and } \quad \jump{\bar\bu}_\tau = \jump{\bar\bu} -\jump{\bar\bu}_n \n^+
$$
its normal and tangential components. 
The tangential gradient and divergence along the fractures are respectively denoted by $\nabla_\tau$ and $\div_\tau$. The symmetric gradient operator $\bbeps$ is defined such that $\bbeps(\bar\bv) = {1\over 2} (\nabla \bar\bv +^t\!(\nabla \bar\bv))$ for a given vector field $\bar\bv\in H^1(\O\backslash\overline\Gamma)^d$.

\begin{figure}
  \begin{center}
  \includegraphics[scale=.65]{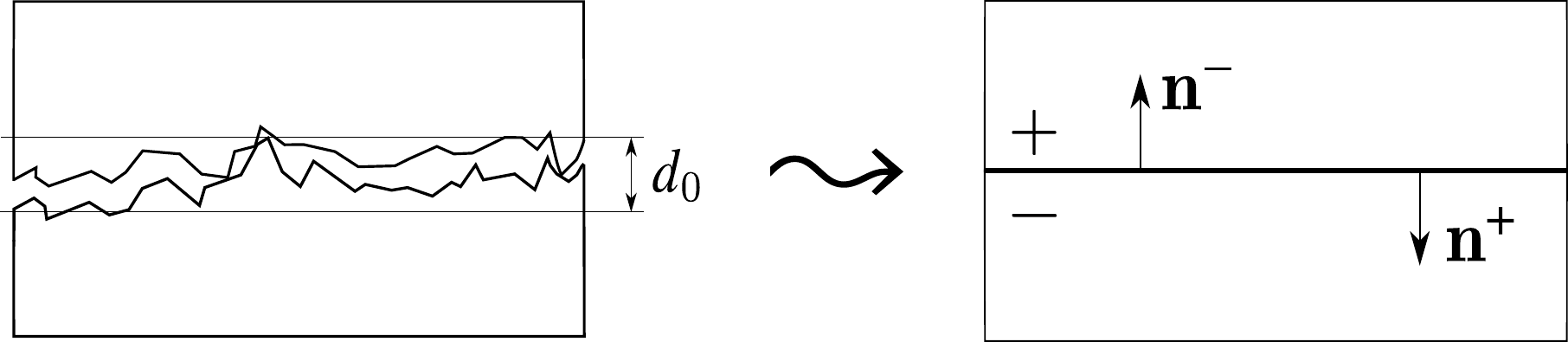}
  \caption{Conceptual fracture model with contact at asperities, $d_0$ representing the fracture aperture at contact state.}
  \label{fig_d0}
\end{center}  
\end{figure}

Let us denote by $d_0: \Gamma \to (0,+\infty)$ the fracture aperture in the contact state (see Figure \ref{fig_d0}). The function $d_0$ is assumed to be continuous with zero limits at
$\partial \Gamma \setminus (\partial\Gamma\cap \partial\Omega)$ (i.e.~the tips of $\Gamma$) and strictly positive limits at $\partial\Gamma\cap \partial\Omega$.

Let us introduce some relevant function spaces. $H_{d_0}^1(\Gamma)$ is the space of functions $v_\Gamma\in L^2(\Gamma)$ such that $d_0^{\nf 3 2} \nabla_\tau v_\Gamma$ belongs to  $L^2(\Gamma)^{d-1}$, 
and whose traces are continuous at fracture intersections $\del\G_i\cap\del\G_j$ (for $(i,j)\in I\times I$, $i\neq j$) and vanish on the boundary $\partial \Gamma\cap \partial\Omega$.
We then introduce the space
\begin{equation*}
\U_0 =\{ \bar \bv\in (H^1(\O\backslash\overline\Gamma))^d : \trace_{\del\O} \bar \bv = 0\}
\end{equation*}
for the displacement vector.
The spaces for each couple of matrix/fracture phase pressures is
\begin{equation*}
V^0 = V^0_m \times V^0_f\quad \mbox{ with }\quad
V^0_m = \{ \bar v\in H^1(\Omega{\setminus}\overline\Gamma) \,:\, \trace_{\del\O} \bar v = 0\}\  \mbox{ and } \
V^0_f = H_{d_0}^1(\Gamma).
\end{equation*}
For $\bar v =(\bar v_m,\bar v_f) \in V^0$, let us denote the jump operator on the side $\aa \in \{+,-\}$ of the fracture by
$$
\jump{\bar v}_\aa = \gamma_\aa \bar v_m - \bar v_f.
$$
The matrix, fracture, and damaged rock types are denoted by the indices ${\rm rt} = m$, ${\rm rt} = f$, and ${\rm rt} = \aa$, {respectively}, and the non-wetting and wetting phases by the superscripts $\alpha =\g$ and $\alpha=\l$, {respectively}. Finally, for any $x\in\mathbb R$, we set $x^+ = \max\{0,x\}$ and $x^- = (-x)^+$.

\section{Problem statement}\label{sec:modeleCont}
%
The primary unknowns of the coupled model are:
{
\begin{itemize}
\item the matrix and fracture phase pressures $\bar p^\alpha_\omega$ for $\omega\in\{m,f\}$ (matrix and fracture) and $\alpha\in\{\g,\l\}$ (non-wetting and wetting phases),
\item the displacement vector field $\bar \bu$.
\end{itemize}
}
The coupled problem  is formulated in terms of flow model, contact mechanics model together with coupling conditions. The flow model is a two-phase hybrid-dimensional model assuming immiscible and incompressible fluids and accounting for the volume conservation equations for each phase $\alpha\in\{\g,\l\}$ and for two-phase Darcy laws 
\begin{equation}
\label{eq_edp_hydro} 
\left\{\!\!\!\!
\begin{array}{lll}
&\partial_t\left(  \bar\phi_m S^{\a}_m(\bar p_{c,m}) \right) + \div \left(  \q^\a_m \right) = h_m^\a & \mbox{ on } (0,T)\times \Omega{\setminus}\overline\Gamma,\\[1ex]
& \q^\a_m = \dsp - \eta^\a_m ( S^{\a}_m(\bar p_{c,m})) \K_m  \nabla \bar p^\a_m & \mbox{ on } (0,T)\times \Omega{\setminus}\overline\Gamma,\\[1ex]
&\partial_t\left(  \bar d_f S^{\a}_f(\bar p_{c,f}) \right)
+ \div_\tau (  \q^\a_f ) -   Q^\a_{f,+} -   Q^\a_{f,-}  = h_f^\a & \mbox{ on } (0,T)\times \Gamma,\\[1ex]
& \q^\a_f = \dsp -  \eta^\a_f (S^{\a}_f(\bar p_{c,f})) ({1\over 12} \bar d_f^{\;3}) \nabla_\tau \bar p_f^\a  & \mbox{ on } (0,T)\times \Gamma. 
\end{array}
\right.
\end{equation}
In \eqref{eq_edp_hydro},  for $\omega\in \{m,f\}$, $\bar p_{c,\omega} = \bar p_\omega^\g - \bar p_\omega^\l$ denotes the capillary pressure, $\eta_\omega^\alpha$ is the phase mobility function, and $S^{\a}_\omega$ the phase saturation function such that $S^{\g}_\omega + S^{\l}_\omega=1$. The matrix porosity is denoted by $\bar \phi_m$ and the matrix permeability tensor by $\K_m$. The fracture aperture, denoted by $\bar d_f$, yields the fracture conductivity ${1\over 12} \bar d_f^{\;3}$ \emph{via} the Poiseuille law. The matrix--fracture interface fluxes $\bar Q^\a_{f ,\aa}$ are defined by the coupling conditions below on each side $\aa$ of the fractures.  

The contact mechanics model accounts for the poromechanical equilibrium equation with a Biot linear elastic constitutive law and Coulomb frictional contact model at matrix--fracture interfaces
\begin{equation}
\label{eq_edp_meca} 
\left\{\!\!\!\!
\begin{array}{lll}
& -\div \(\bbsig(\bar\bu) - b ~ \bar p^E_m{\mathbb I}\)= \mathbf{f} & \mbox{ on } (0,T)\times \Omega{\setminus}\overline\Gamma,\\[1ex]
  & \bbsig(\bar\bu) = \frac{E}{1+\nu}\(\bbeps(\bar\bu) + \frac{\nu}{1-2\nu} (\div\,\bar\bu)\mathbb{I}\) & \mbox{ on } (0,T)\times \Omega{\setminus}\overline\Gamma,  \\[1ex]
  & {\bf T}^{+} + {\bf T}^{-} = {\bf 0}  & \mbox{ on } (0,T)\times \Gamma,\\[1ex]
  & T_n \leq 0, \,\,   \jump{\bar\bu}_n \leq 0, \,\, \jump{\bar\bu}_n ~  T_n = 0   & \mbox{ on } (0,T)\times \Gamma,\\[1ex]
  & |{\bf T}_\tau| \leq - F ~T_n   & \mbox{ on } (0,T)\times \Gamma, \\[1ex]
  &  (\partial_t \jump{\bar\bu}_\tau)\cdot {\bf T}_\tau - F ~T_n |\partial_t \jump{\bar\bu}_\tau|=0   & \mbox{ on } (0,T)\times \Gamma.   
\end{array}
\right.
\end{equation}
In \eqref{eq_edp_meca}, $b$ is the Biot coefficient, $E$ and $\nu$ are the effective Young modulus and Poisson ratio, $F \ge 0$ is the friction coefficient, and the contact tractions are defined by 
\begin{equation*}
\left\{\!\!\!\!
\begin{array}{lll}
  & {\bf T}^{\aa} = {(\bbsig(\bar\bu) - b ~ \bar p^E_m \mathbb I)\n^\aa + \bar p^E_f \n^\aa}  & \mbox{ on } (0,T)\times \Gamma,  \aa \in \{+,-\},\\[1ex]
  & T_n = {\bf T}^{+}\cdot \n^+, & \mbox{ on } (0,T)\times \Gamma, \\[1ex]
  & {\bf T}_\tau = {\bf T}^{+} - ({\bf T}^{+}\cdot \n^+)\n^+ & \mbox{ on } (0,T)\times \Gamma. 
\end{array}
\right.
\end{equation*}
The complete system of equations~\eqref{eq_edp_hydro}--\eqref{eq_edp_meca} is closed by means of coupling conditions. The first equation in \eqref{closure_laws} accounts for the linear poroelastic state law for the variations of the matrix porosity $\bar\phi_m$ extended to two-phase flows using the key concept of equivalent pressure $\bar p^E_m$~\cite{coussy}. The second and third ones are the matrix--fracture transmission conditions for the two-phase flow model. Following~\cite{DHM16} they account for the volume conservation equations at each side $\aa$ of the fractures. A layer of damaged rock of thickness $\bar d_\aa$ (possibly vanishing) is included in the model, characterized by its own porosity $\bar \phi_\aa$,  mobility functions $\eta_\aa^\alpha$, and saturation functions $S^\a_\aa$. It is illustrated in Figure \ref{fig_mfflux} which also exhibits the matrix--fracture interface fluxes $\bar Q^\a_{f ,\aa}$, which can be viewed as two-point monotone approximations in the fracture width (see~\cite{DHM16} for more details). The normal fracture transmissibility $\Lambda_f$ is assumed here to be independent of $\bar d_f$ for simplicity, but the subsequent analysis can accommodate a fracture-width-dependent transmissibility, provided that the function $\bar d_f\mapsto\Lambda_f(\bar d_f)$ is continuous and bounded from below and above by strictly positive constants. 
The fourth equation in \eqref{closure_laws} is the definition of the fracture aperture $\bar d_f$.  
\begin{equation}
\label{closure_laws}
\left\{\!\!\!\!
\begin{array}{lll}
  & \partial_t \bar\phi_m = \dsp b~\div \partial_t \bar\bu + \frac{1}{M} \partial_t \bar p^E_m & \mbox{ on } (0,T)\times \Omega{\setminus}\overline\Gamma,\\  [2ex]
  & \q_m^\a\cdot\n^\aa - Q_{f,\aa}^\a = \bar d_\aa \bar \phi_\aa \partial_t S^\a_\aa(\gamma_\aa \bar p_{c,m}) &   \mbox{ on } (0,T)\times \Gamma, \aa \in \{+,-\},\\[1ex]
  & \bar Q_{f,\aa}^\a =  \eta^\a_\aa(S^\a_\aa(\gamma_\aa \bar p_{c,m})) \Lambda_f \jump{\bar p^\a}_\aa^+  -
  \eta^\a_f(S^\a_f(\bar p_{c,f})) \Lambda_f \jump{\bar p^\a}_\aa^-  &   \mbox{ on } (0,T)\times \Gamma, \aa \in \{+,-\}, \\[1ex]
& \bar d_f = d_0 -\jump{\bar\bu}_n   & \mbox{ on } (0,T)\times \Gamma.  
\end{array}
\right.
\end{equation}
The following initial conditions are imposed on the phase pressures and matrix porosity 
\begin{equation}
  \label{initial_conditions}
\bar p^\a_\omega|_{t=0} = \bar p^{\a}_{0,\omega},  \quad \bar \phi_m|_{t=0}= \bar\phi_m^0, 
\end{equation}
and normal flux conservation for $\q^\a_f$ is prescribed at fracture intersections not located on the boundary $\partial\Omega$ (see Figure~\ref{fluxes}). \\

As exhibited in Figure \ref{fig_d0}, due to surface roughness, the fracture aperture $\bar d_f \geq d_0$ does not vanish except at the tips. The open space is always occupied by the fluids, which act on each side $\aa$ of the fracture by means of the fracture equivalent pressure $\bar p^E_f$ defined below, appearing in the definition of the contact traction ${\bf T}^{\aa}$.
In the above equations, the equivalent pressure $\bar p^E_\omega$, $\omega\in \{m,f\}$, is defined following~\cite{coussy} as
$$
\bar p^E_{\omega} = \dsp \sum_{\alpha \in \{\g,\l\} } \bar p^\alpha_\omega~S^\alpha_\omega(\bar p_{c,\omega})  - U_{\omega}(\bar p_{c,\omega}), 
$$
where
\begin{equation*}
U_{\rm rt}(\bar p_c) =  \int_0^{\bar p_c} q  ( S^{\g}_{{\rm rt}})'(q) \d q
\end{equation*}
is the capillary energy density function for each rock type ${\rm rt}\in \{m,f,\aa\}$. As already noticed in~\cite{KTJ2013,JHA2014,GDM-poromeca-cont}, the use of these equivalent pressures is instrumental to obtaining energy estimates.
It formally relies on the following identity, in the matrix:
$$
\sum_{\alpha\in \{\g,\l\}} \bar p^\alpha_m \partial_t \( \bar\phi_m S^{\a}_m(\bar p_{c,m}) \) - b \bar p^E_m \div(\partial_t \bu) = \partial_t \( \bar \phi_m U_m(\bar p_{c,m}) + {1\over 2 M} (\bar p^E_m)^2 \). 
$$
This idea is extended here to the fracture two-phase poromechanical coupling and is based on the formal identity:
$$
\sum_{\alpha\in \{\g,\l\}} \bar p^\alpha_f \partial_t \(\bar d_f S^{\a}_f(\bar p_{c,f})\) + \bar p^E_f \jump{\partial_t \bar\bu}_n =
\partial_t \( \bar d_f U_f(\bar p_{c,f})\). 
$$
Note that, unlike~\cite{JHA2014}, the definition of the contact traction on each side $\aa$ is based on the fracture pressures $\bar p^\alpha_f$ and not on the matrix--fracture interface pressures $\gamma_\aa \bar p^\alpha_m$, which does not seem to lead to an energy estimate for the coupled model. \\
%
\begin{SCfigure}
\includegraphics[width=6cm]{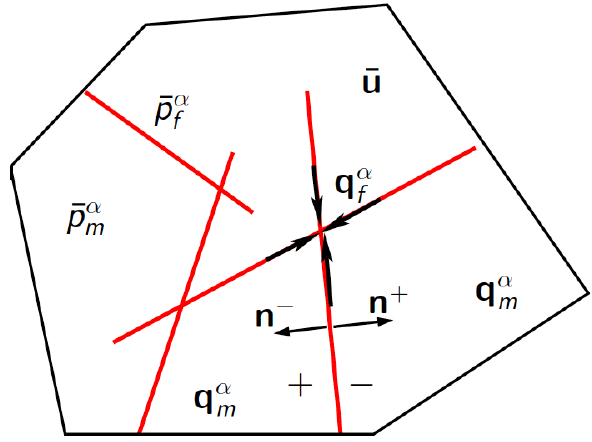}
\caption{Example of a 2D domain $\Omega$ with its fracture network $\Gamma$, unit normal vectors $\n^\pm$ to $\Gamma$, phase pressures $\bar p^\alpha_m$ in the matrix and $\bar p^\alpha_f$ in the fracture network, displacement vector field $\bar \bu$, matrix Darcy velocities $\q^\alpha_m$ and fracture tangential Darcy velocities $\q^\alpha_f$ integrated along the fracture width. }
\label{fluxes}
\end{SCfigure}
\ \
\begin{SCfigure}
\raisebox{.5cm}{
\includegraphics[scale=.65]{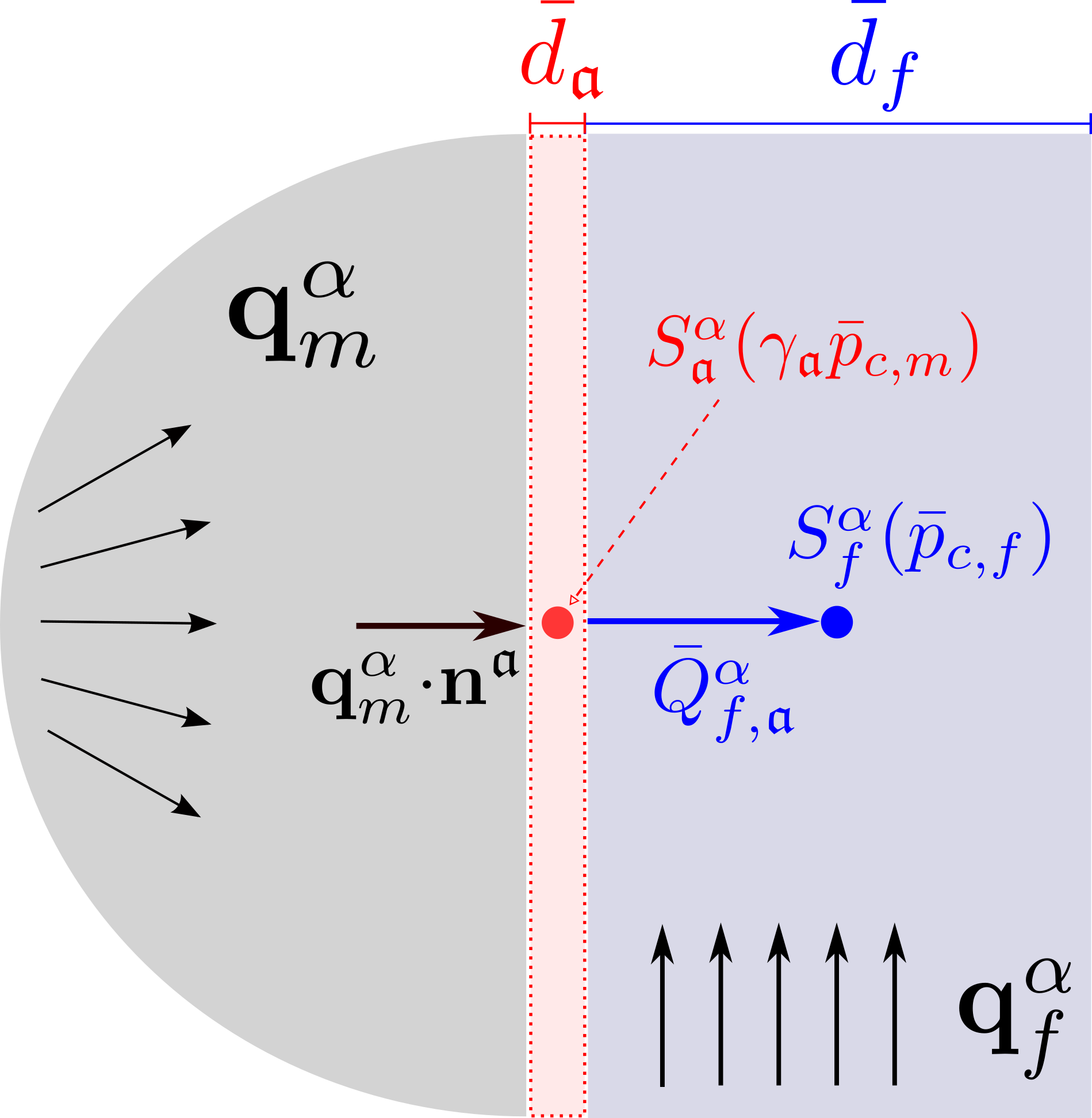}}
\caption{Illustration of the flux transmission condition between matrix and fracture, including a layer of damaged rock of thickness $\bar d_\aa$, $\aa \in \{+,-\}$. It can be seen as an upwind two-point-like approximation of $\bar Q^\a_{f ,\aa}$. The arrows show the positive orientation of the normal fluxes $\q_m^\a\cdot\n^\aa$ (inward to the damaged layer) and $\bar Q^\a_{f,\aa}$ (outward from the damaged layer).}
 \label{fig_mfflux}
\end{SCfigure}
We make the following main assumptions on the {parameters of the model}:
\begin{enumerate}[label=(H\arabic*),leftmargin=*]
\item For each phase $\a \in\{\g,\l\}$ and rock type ${\rm rt}\in\{m,f,\aa\}$, the mobility function $\eta^\a_{\rm rt}$ is continuous, non-decreasing, and there exist
  $0< \eta_{\rm rt,{\rm min}}^\a \leq \eta^\a_{\rm rt,{\rm max}}< +\infty$  such that 
  $\eta^\a_{\rm rt,{\rm min}} \leq \eta^\a_{\rm rt}(s) \leq \eta^\a_{\rm rt,{\rm max}}$ for all $s\in [0,1]$. Moreover, $\eta_{\mathfrak a}^{\rm w}(1) = \eta_f^{\rm w}(1)$ and $\eta_{\mathfrak a}^{\rm nw}(0) = \eta_f^{\rm nw}(0)$.
  \label{first.hyp}
\item For each rock type ${\rm rt}\in\{m,f,\aa\}$, the non-wetting phase saturation function $S^\g_{\rm rt}$ is a non-decreasing {Lipschitz}
  continuous function with values in $[0,1]$, and   $S^\l_{\rm rt} = 1 - S^\g_{\rm rt}$.
  \label{second.hyp}
  \item For $\aa \in \{+,-\}$, the width $\bar d_\aa$ and porosity $\bar\phi_\aa$ of the damaged rock are positive constants.
  {%
  \item The fracture aperture at contact state satisfies $d_0 > 0$ and is continuous over the fracture network $\Gamma$, with zero limits at $\partial\Gamma\setminus (\partial\Gamma \cap \partial\Omega)$ and strictly positive limits at $\partial\Gamma\cap\partial\Omega$.
  }
\item $b\in [0,1]$ is the Biot coefficient, $M\in (0,+\infty]$ is the Biot modulus, and $E >0$, $-1<\nu<1/2$, {and $F\ge 0$} are Young's modulus, Poisson's ratio, {and friction coefficient, respectively}. These coefficients are assumed constant {to alleviate technicalities in the analysis,} and $1/M$ is interpreted as $0$ when $M=+\infty$ (incompressible rock).
\item The initial matrix porosity satisfies $\bar\phi_m^0 \in L^\infty(\O)$ and $\bar\phi_m^0 \geq 0$. 
\item The initial pressures are such that $p^\alpha_{0,m}\in L^\infty(\Omega)$ and $\bar p^\alpha_{0,f}\in L^\infty(\Gamma)$, $\alpha\in \{\g,\l\}$. 
\item The source terms satisfy $\mathbf f \in L^2(\Omega)^d$, $h_m^\alpha\in L^2((0,T)\times \Omega),$ and $h_f^\alpha \in L^2((0,T)\times \Gamma)$.
  \item The normal fracture transmissibility $\Lambda_f\in L^\infty(\Gamma)$ is uniformly bounded from below by a strictly positive constant. 
 \item The matrix permeability tensor $\K_{m}\in L^\infty(\Omega)^{d\times d}$ is symmetric and uniformly elliptic.
 \label{last.hyp}
\end{enumerate}

\subsection{{Variational formulation}}

Following~\cite{Wohlmuth11}, the poromechanical model  with Coulomb frictional contact is formulated in mixed form using a vector Lagrange multiplier $\bar{\boldsymbol{\lambda}}: \Gamma \to \R^d$ at matrix--fracture interfaces.
Denoting for $r\in\{1,d\}$ the duality pairing of $H^{-1/2}(\Gamma)^r$ and $H^{1/2}(\Gamma)^r$ by $\HGbracket{\cdot}{\cdot}$, we define the dual cone 
\begin{align*}
\bm{C}_f(\bar\lambda_n) = \Big\{\bar{\boldsymbol{\mu}}{} \in  (H^{-1/2}(\Gamma))^d \,:\,
 \HGbracket{\bar{\boldsymbol{\mu}}}{\bar{\bv}} \leq \HGbracket{F \bar\lambda_n}{|\bar{\bv}_\tau|} \mbox{ for all } \bar{\bv} \in (H^{1/2}(\Gamma))^d \mbox{ with } \bar v_n \leq 0\Big\}. 
\end{align*}

The Lagrange multiplier formulation of \eqref{eq_edp_meca} then formally reads, dropping any consideration of regularity in time: find $\bar\bu:[0,T]\to \U_0$ and $\bar{\boldsymbol{\lambda}}=(\bar{\lambda}_n,\bar{\boldsymbol{\lambda}}_\tau) :[0,T]\to \bm{C}_f(\bar{\lambda}_n)$ such that for all $\bar\bv:[0,T]\to \U_0$ and $\bar{\boldsymbol{\mu}} = (\bar{\mu}_n,\bar{\boldsymbol{\mu}}_\tau) :[0,T]\to \bm{C}_f(\bar{\lambda}_n)$, one has
\begin{equation}
  \label{Lagrange_meca_contactfriction}
\begin{aligned}
 & \dsp \int_\O \( \bbsig(\bar \bu): \bbeps(\bar \bv) - b ~\bar p_m^E \div(\bar \bv)\) \d\x
+  \HGbracket{\bar{\boldsymbol{\lambda}}}{\jump{\bar \bv}} 
+ \int_\G \bar p_f^E ~\jump{\bar \bv}_n  ~\d\sigma 
\dsp =  \int_\Omega \mathbf{f}\cdot \bar\bv ~\d\x,\\[.8em]
& \dsp \HGbracket{\bar{\mu}_n-\bar{\lambda}_n}{\jump{\bar \bu}_n} + \HGbracket{\bar{\boldsymbol{\mu}}_\tau-\bar{\boldsymbol{\lambda}}_\tau}{ \jump{\partial_t \bar \bu}_\tau}\leq 0.   
\end{aligned}
\end{equation}
Note that, based on the variational formulation, the Lagrange multiplier satisfies $\bar{\boldsymbol{\lambda}} = -{\bf T}^+ = {\bf T}^-$.

{
Dropping again any consideration of time regularity, the variational formulation of the two-phase Darcy flow model can be formulated as follows: find $(\bar p^\alpha_m, \bar p^\a_f) : (0,T) \to V^0$, $\alpha\in \{\g,\l\}$ such that, for any $\alpha\in \{\g,\l\}$, $\bar{d}_f^{\;\nf 3 2}\nabla_\tau\bar p_f^\alpha:  (0,T)\to L^2(\G)^{d-1}$ and, for all $(\bar \varphi^\alpha_m,\bar\varphi^a_f): (0,T) \to V^0$, 
\begin{equation}
\label{eq_var_hydro}      
\left.
\begin{array}{ll}
&\dsp \int_\O \(\partial_t \(\bar \phi_m S^{\a}_m(\bar p_{c,m})\) \bar \varphi_m^\a  + \eta^\a_m ( S^{\a}_m(\bar p_{c,m})) \K_m  \nabla \bar p^\a_m \cdot \nabla \bar \varphi^\a_m\) \d\x  \\[2ex]
  &   + \dsp \int_\G \(\partial_t \( \bar d_f S^{\a}_f(\bar p_{c,f})\) \bar \varphi^\a_f  + \eta^\a_f ( S^{\a}_f(\bar p_{c,f})) {\bar d_f^{\;3}\over 12}  \nabla_\tau \bar p^\a_f \cdot \nabla_\tau \bar \varphi^\a_f \) \d\sigma \\[2ex]
  &   + \dsp \dsp \sum_{\aa=\pm} \int_\G  \( \bar Q^\a_{f,\aa}  \jump{\bar \varphi^\a}_\aa
    + \bar d_\aa \bar \phi_\aa \partial_t S^{\a}_\aa( \gamma_\aa \bar p_{c,m}) \gamma_\aa \bar \varphi^\a_m\) \d\sigma = \dsp  \int_\O h_m^\a \bar \varphi^\a_m \d\x  +  \int_\G h_f^\a \bar \varphi^\a_f  ~\d\sigma.
\end{array}
\right.
\end{equation}
The coupled model amounts to finding $(\bar p^\alpha_m, \bar p^\a_f)$, $\alpha\in \{\g,\l\}$, $\bar\bu$ and $\bar{\boldsymbol{\lambda}}$ satisfying the variational formulations \eqref{Lagrange_meca_contactfriction} and \eqref{eq_var_hydro} as well
as the first, third and fourth closure laws in \eqref{closure_laws} and the initial conditions \eqref{initial_conditions}. Note that the initial displacement $\bar\bu^0$ and Lagrange multiplier $\bar{\boldsymbol{\lambda}}^0$ are  the solution of
\eqref{Lagrange_meca_contactfriction} without the time variable and with the equivalent pressures obtained from the initial pressures $\bar p^\alpha_{0,\omega}$, with ${\alpha\in\{\g,\l\}}$ and ${\omega\in\{m,f\}}$.
}

\begin{remark}[Modified two-phase flow model with partial linearization of the matrix accumulation term]\label{rem:modified2phase} 
  A classical modification of this two-phase flow model consists in applying a partial linearization of the matrix accumulation term in \eqref{eq_edp_hydro}, in the spirit of~\cite{both2019global} for the case of an unsaturated poroelastic model based on the Richards equation. Motivated by the small porosity variations assumption~\cite{coussy} for linear poroelasticity, the term $\partial_t ( \bar \phi_m S^\alpha_m(\bar p_{c,m}))$ is replaced by
  \begin{equation}
    \label{ACC_PL}
  \bar \phi^0_m \partial_t  S^\alpha_m(\bar p_{c,m}) + S^\alpha_m(\bar p_{c,m}) \partial_t \bar \phi_m. 
   \end{equation}
  To establish an energy estimate in the subsequent analysis, this modified model must be combined with a new definition of the matrix equivalent pressure, in which the capillary energy density function is removed: 
  \begin{equation}
    \label{PE_WA}
  \bar p^E_m = \sum_{\alpha\in \{\g,\l\}} \bar p^\alpha_m S^\alpha_m(\bar p_{c,m}),  
  \end{equation}
  leading formally to
  $$
\sum_{\alpha\in \{\g,\l\}} \bar p^\alpha_m \partial_t \( \bar\phi_m S^{\a}_m(\bar p_{c,m}) \) - b \bar p^E_m \div(\partial_t \bu) = \partial_t \( \bar \phi^0_m U_m(\bar p_{c,m}) + {1\over 2 M} (\bar p^E_m)^2 \). 
$$
Note that the fracture equivalent pressure is unchanged and that both two-phase models degenerate to the same \mbox{single-phase} model. 
  The main advantage of this modified model is to guarantee the positivity of the porosity term in front of the time derivative of the saturation term (see Remark \ref{rem:EEmodified2phase}). As will be seen in the stability analysis, this positivity needs to be assumed in the analysis of the original model \eqref{eq_edp_hydro}--\eqref{eq_edp_meca}. We however notice that the choice of the equivalent pressure \eqref{PE_WA} is not as physically relevant as the original one for strong capillary effects~\cite{coussy,KTJ2013}, which are a feature of our main application.
\end{remark}

\section{The gradient discretization method}\label{sec:gradientscheme}
\subsection{Gradient discretizations}
The gradient discretization (GD) for the Darcy discontinuous pressure model, introduced in~\cite{DHM16}, is defined by a finite-dimensional vector space of discrete unknowns $X^0_{\D_p} = X^0_{\D_p^m} \times X^0_{\D_p^f}$
and 
\begin{itemize}
\item  two discrete gradient linear operators on the matrix and fracture domains
$$
 \nabla_{\D_p}^m : X^0_{\D_p^m} \rightarrow L^\infty(\Omega)^d, \quad \quad 
   \nabla_{\D_p}^f : X^0_{\D_p^f} \rightarrow L^\infty(\Gamma)^{d-1},
$$ 
\item two function reconstruction linear operators on the matrix and fracture domains 
$$
\Pi_{\D_p}^m : X^0_{\D_p^m} \rightarrow L^\infty(\Omega),\quad\quad \Pi_{\D_p}^f : X^0_{\D_p^f} \rightarrow L^\infty(\Gamma),
$$
\item for $\aa \in \{+,-\}$, jump reconstruction linear operators $\jump{\cdot}^\aa_{\D_p}$: $X^0_{\D_p} \rightarrow L^\infty(\Gamma)$,  and trace reconstruction linear operators $\mathbb{T}^\aa_{\D_p}$: $X^0_{\D_p^m} \rightarrow L^\infty(\Gamma)$.
\end{itemize}  
The operators $\Pi_{\D_p}^m$, $\Pi_{\D_p}^f$, $\mathbb{T}^\aa_{\D_p}$ are assumed to be \emph{piecewise constant}~\cite[Definition 2.12]{gdm}. 
The vector space $X^0_{\D_p}$ is endowed with the following quantity, assumed to define a norm:
\begin{equation}\label{def:XDp.norm}
\displaystyle  \|\nabla_{\D_p}^m v\|_{L^2(\Omega)^d} 
+ \|d_{0}^{\nf 3 2}\nabla_{\D_p}^f v\|_{ L^2(\Gamma)^{d-1} } + \sum_{\aa \in \{+,-\}} \|\jump{v}^\aa_{\D_p}\|_{L^2(\Gamma)}.
\end{equation}

As usual in the GDM framework~\cite[Part III]{gdm}, various choices of spaces and operators above lead to various numerical methods for the flow component of the model. It covers the case of cell-centered finite volume schemes with Two-Point Flux Approximation on strongly admissible meshes \cite{KDA04,ABH09,gem.aghili}, or some symmetric Multi-Point Flux Approximations \cite{TFGCH12,SBN12,AELHP153D} on tetrahedral or hexahedral meshes. It also accounts for the families of Mixed Hybrid Mimetic and Mixed or Mixed Hybrid Finite Element discretizations such as in \cite{MJE05,BHMS2016,AFSVV16,Girault2019}, and for vertex-based discretizations such as  the Vertex Approximate Gradient scheme \cite{BHMS2016,DHM16,BHMS2018}. For the discretization of the mechanical component of the system, with Coulomb frictional contact, we however restrict ourselves to conforming methods for the displacement (usual methods for elasticity models), and piecewise constant spaces for the Lagrange multipliers. We therefore take a finite-dimensional space
$$
X_{\D_\bu} = X^0_{\D_\bu^m}\times  X_{\D_\bu^f} \ \mbox{ with } \ X^0_{\D_\bu^m}\subset\U_0
\quad \mbox{ and }\quad X_{\D_\bu^f} = \{\boldsymbol{\mu}=(\boldsymbol{\mu}_\sigma)_{\sigma\in\faces_{\D_\bu}}\,:\,\boldsymbol{\mu}_\sigma\in \R^d\},
$$
where $\faces_{\D_\bu}$ is a partition of $\Gamma$ assumed to be conforming with the partition $\{\Gamma_i, i\in I\}$ of $\Gamma$ in planar fractures. We will also identify $\boldsymbol{\mu}\in X_{\D_\bu^f}$ with the piecewise constant function $\boldsymbol{\mu}:\Gamma\to\R^d$ defined by $\boldsymbol{\mu}_{|\sigma}=\boldsymbol{\mu}_\sigma$ for all $\sigma\in \faces_{\D_\bu}$. As previously, it will be useful to separate the normal and tangential components of elements in $X_{\D_\bu^f}$ and we thus set, for $\boldsymbol{\mu}$ such an element,
$$
\mu_{n,\sigma} = \boldsymbol{\mu}\cdot \n_\sigma, \quad \boldsymbol{\mu}_{\tau,\sigma} = \boldsymbol{\mu}_\sigma - \mu_{n,\sigma} \n_\sigma, \quad \mu_n = (\mu_{n,\sigma})_{\sigma  \in\faces_{\D_\bu}},
\quad \boldsymbol{\mu}_\tau = (\boldsymbol{\mu}_{\tau,\sigma})_{\sigma  \in\faces_{\D_\bu}},
$$
where $\n_\sigma = \n^+$ is the constant unit normal vector on $\sigma$ oriented outward from the side $+$. We also identify $\mu_n$ and $\boldsymbol{\mu}_\tau$ with the corresponding piecewise constant functions on $\Gamma$.

Let us define the discrete dual cone
$$
\bm{C}_{\D_\bu^f}(\lambda_n) = \{\boldsymbol{\mu} = (\mu_n,\boldsymbol{\mu}_\tau)\in X_{\D_\bu^f}  \,:\, \mu_{n,\sigma}\geq 0, \, |\boldsymbol{\mu}_{\tau,\sigma}| \leq F \lambda_{n,\sigma} \quad\forall \sigma \in\faces_{\D_\bu}\}.  
$$

A spatial GD can be extended into a space-time GD by complementing it with
\begin{itemize}
\item 
a discretization $ 0 = t_0 < t_1 < \dots < t_N = T $ of the time interval $[0,T]$,
\item
interpolators $I^m_{\D_p} \colon L^\infty(\Omega)\rightarrow X^0_{\D_p^m}$, $I^f_{\D_p} \colon L^\infty(\Gamma)\rightarrow X^0_{\D_p^f}$, and $J^m_{\D_p} \colon L^2(\Omega)\rightarrow X^0_{\D_p^m}$ of initial conditions.
\end{itemize}
For $k\in\{0,\ldots,N\}$, we denote by $\dtk = t_{k+1}-t_k$ the time steps, and by $\Delta t = \max_{k=0,\ldots,N} \dtk$ the maximum time step. 

Spatial operators are extended into space-time operators as follows. Let $\Psi_\D$ be a spatial GDM operator defined in $X_\D^0$ with $\D=\D_p^m$ or $\D_p^f$,  and let $w=(w_k)_{k=0}^{N}\in (X^0_{\D})^{N+1}$. Then, its space-time extension is defined by 
$$
\Psi_{\D}w(0,\cdot) = \Psi_{\D}w_0 \mbox{ and, } \forall k\in\{0,\dots,N-1\}\,,\;\forall t\in (t_k,t_{k+1}],\,\;
\Psi_{\D}w(t,\cdot) = \Psi_{\D}w_{k+1}.
$$
For convenience, the same notation is kept for the spatial and space-time operators. Similarly, we identify $(X_{\D_\bu^m}^0)^{N+1}$ and $(X_{\D_\bu^f})^{N+1}$, respectively, with the spaces of piecewise constant functions $[0,T]\to X_{\D_\bu^m}^0$ and $[0,T]\to X_{\D_\bu^f}$; so, for example, if $\bu=(\bu^k)_{k=0,\ldots,N}\in (X_{\D_\bu^m}^0)^{N+1}$, we set $\bu(0)=\bu^0$ and $\bu(t)=\bu^{k+1}$ for all $t\in (t_k,t_{k+1}]$ and $k=0,\ldots,N-1$.
Moreover, we define the discrete time derivative as follows: for $f:[0,T]\to E$, with $E$ a vector space, piecewise constant on the time discretization, with $f_k=f_{|(t_{k-1},t_k]}$ and $f_0=f(0)$, we set $\delta_t f (t) = \frac{f_{k+1} - f_k}{\dtk}$ for all  $t\in (t_k,t_{k+1}]$, $k\in\{0,\ldots,N-1\}$. Note that it will correspond to the Euler implicit time integration in the following gradient scheme formulation. 
  
\subsection{Gradient scheme}
First, let us define, for all $\sigma \in \faces_{\D_\bu}$ the displacement average on the side $\aa \in \{+,-\}$ of $\sigma$ by 
\begin{equation}
\label{average_jump}
{\bf u}_\sigma^\aa = {1\over |\sigma|}\int_\sigma \gamma_\aa{\bf u}(\x)  \d\sigma.
\end{equation}
The displacement jump average and its normal and tangential components are
$$
\jump{\bu}_{\sigma}  =  {\bf u}_\sigma^+ - {\bf u}_\sigma^-, \quad
\jump{\bu}_{n,\sigma} = \jump{\bu}_{\sigma} \cdot {\bf n}^+\quad
\mbox{and}\quad
\jump{\bu}_{\tau,\sigma}  = \jump{\bu}_{\sigma} - \jump{\bu}_{n,\sigma}  ~{\bf n}^+.   
$$
We then define the global displacement normal jump reconstruction $\jump{\bu}_{n,\faces}$ such that, for any $\sigma\in\faces_{\D_\bu}$,
$${(\jump{\bu}_{n,\faces})}_{|_\sigma} = \jump{\bu}_{n,\sigma}.$$

The gradient scheme for \eqref{eq_edp_hydro}--\eqref{eq_edp_meca}  consists in writing a discrete weak formulation obtained after a formal integration by parts in space and by replacing the continuous operators by their discrete counterparts: find $p^\alpha = (p^\a_m, p^\a_f) \in (X^0_{\D_p})^{N+1}$ for $\alpha\in \{\g,\l\}$, $\bu \in (X^0_{\D_\bu^m})^{N+1}$, and $(\boldsymbol{\lambda}^k)_{k=0,\ldots,N}$ with $\boldsymbol{\lambda}^k\in \bm{C}_{\D_\bu^f}(\lambda^k_n)$ for all $k=0,\dots,N$, such that for all $\varphi^\alpha = (\varphi^\a_m, \varphi^\a_f) \in (X_{\D_p}^0)^{N+1}$ ($\alpha\in \{\g,\l\}$), $\bv \in (X^0_{\D_\bu^m})^{N+1}$, and $\boldsymbol{\mu}=(\boldsymbol{\mu}^k)_{k=1,\ldots,N} \in \bigtimes_{k=1}^N\bm{C}_{\D_\bu^f}(\lambda^k_n)$,
\begin{subequations}\label{eq:GS}
\begin{equation}
\begin{aligned}
  {}&\int_0^T \int_\Omega \( \delta_t \(\phi_\D \Pi_{\D_p}^m s^\alpha_m \)\Pi_{\D_p}^m \varphi_m^\alpha 
  +  \eta_m^\alpha(\Pi_{\D_p}^m s_m^\alpha) \K_m \nabla_{\D_p}^m p^\alpha_m \cdot   \nabla_{\D_p}^m \varphi^\alpha_m \) \d\x \d t\\
   &+ \int_0^T \int_\Gamma \delta_t \(d_{f,\D_\bu} \Pi_{\D_p}^f s^\alpha_f \)\Pi_{\D_p}^f \varphi^\alpha_f \d\sigma\d t
   + \int_0^T \int_\Gamma  \eta_f^\alpha(\Pi_{\D_p}^f s_f^\alpha) {d_{f,\D_\bu}^3 \over 12} \,\, \nabla_{\D_p}^f p^\alpha_f \cdot   \nabla_{\D_p}^f \varphi^\alpha_f  \d\sigma \d t\\
   &+ \sum_{\aa \in \{+,-\}}
  \int_0^T \int_\G  \( Q^\a_{f,\aa}  \jump{\varphi^\a}^\aa_{\D_p}   
  +   \bar d_\aa \bar \phi_\aa \delta_t \(\mathbb{T}^\aa_{\D_p} s^{\a}_\aa \) \mathbb{T}^\aa_{\D_p}\varphi^\a_m\) \d\sigma \d t \\
  &\qquad =  \int_0^T \int_\Omega h_m^\alpha \Pi_{\D_p}^m \varphi^\alpha_m \d\x \d t + \int_0^T \int_\Gamma h_f^\alpha \Pi_{\D_p}^f \varphi^\alpha_f \d\sigma\d t,
  \label{GD_hydro}
\end{aligned}
\end{equation}
\medskip
\begin{equation}
\begin{aligned}
   \int_0^T  \int_\Omega {}&\( \bbsig(\bu) : \bbeps(\bv)  
    - b ~\Pi_{\D_p}^m p_m^E~  \div\,\bv\)  \d\x \d t
   +  \int_0^T \int_\Gamma  \! \boldsymbol{\lambda} \cdot  \jump{\bv}~\d\sigma\d t\\
     &
     +  \int_0^T  \int_\Gamma \!\! \Pi_{\D_p}^f p_f^E~  \jump{\bv}_{n,\faces} ~\d\sigma\d t 
  = \int_0^T  \int_\Omega \!\! \mathbf{f} \cdot  \bv ~\d\x \d t,
\end{aligned}
\label{GD_meca}
\end{equation}
\begin{equation}
    \int_0^T \int_\Gamma (\mu_n - \lambda_n) \jump{\bu}_{n}~ \d\sigma\d t + \int_0^T \int_\Gamma(\boldsymbol{\mu}_\tau - \boldsymbol{\lambda}_\tau) \cdot \delta_t \jump{\bu}_\tau \d\sigma \d t\leq 0,  
  \label{GD_meca_var}
\end{equation}
with the closure equations, for $\omega\in \{m,f\}$ and $\aa\in\{+,-\}$,
\begin{equation}
  \left\{\!\!\!\!\begin{array}{ll}
  & Q^\a_{f,\aa} = \Lambda_f \[ \eta^\a_\aa(\mathbb{T}^\aa_{\D_p} s^\a_\aa)  (\jump{p^\a}^\aa_{\D_p})^+  -
  \eta^\a_f(\Pi^f_{\D_p}s^\a_f)  (\jump{p^\a}_{\D_p}^\aa)^- \],\\[2ex]
  & p_{c,\omega} = p^\g_\omega - p^\l_\omega,  \quad s^\alpha_\omega = S^\alpha_\omega(p_{c,\omega}), \quad s^\a_\aa = S^\a_\aa(p_{c,m}),\\[2ex]
  & \dsp p_\omega^E = \sum_{\alpha\in\{\g,\l\}}  p^\alpha_\omega s^\alpha_\omega - U_\omega(p_{c,\omega}),\\[4ex]
  &  \phi_{\D} - \Pi_{\D_p}^m  \phi_m^0 = b ~\div(\bu-\bu^0) + {1\over M} \Pi_{\D_p}^m (p_m^E-p_m^{E,0}),\\[2ex]
  & d_{f,\D_\bu} = d_0 - \jump{\bu}_{n,\faces}.
  \end{array}\right.
  \label{GD_closures}
\end{equation}
\end{subequations}
The initial conditions are given by $p^\a_{0,\omega} = I^\omega_{\D_p} \bar p^\a_{0,\omega}$ ($\a\in\{\g,\l\}$, $\omega\in \{m,f\}$), $\phi_m^0 = J_{\D_p}^m \bar \phi^0$, and the initial displacement $\bu^0$ and Lagrange multiplier $\boldsymbol{\lambda}^0$ are  the solution in $X_{\D_\bu^m}\times \bm{C}_{\D_\bu^f}(\lambda^0_n)$ of
\eqref{GD_meca} without the time variable and with the equivalent pressures obtained from the initial pressures $(p^\alpha_0)_{\alpha\in\{\g,\l\}}$.

\subsubsection{Formulation in local Coulomb frictional contact conditions}

We provide here a local reformulation of the variational condition \eqref{GD_meca_var}.

\begin{lemma}[Local Coulomb frictional contact conditions]\label{lem:local.coulomb}
Let $\boldsymbol{\lambda} \in \bm{C}_{\D_\bu^f}(\lambda_n)$ and $\bu\in (X_{\D^m_\bu}^0)^{N+1}$. Then $(\bu,\boldsymbol{\lambda})$ satisfy the variational inequality \eqref{GD_meca_var} if and only if the following local frictional contact conditions hold on $[0,T]$ and for any $\sigma\in \faces_{\D_\bu}$:
\begin{subequations}
\label{fracture_frictionconditions}
\begin{align}
\label{fracture_frictionconditions.1}
  & \lambda_{n,\sigma} \geq 0, \,\,   \jump{\bu}_{n,\sigma} \leq 0, \,\, \jump{\bu}_{n,\sigma}  \lambda_{n,\sigma} = 0, \\
\label{fracture_frictionconditions.2}
  & | \boldsymbol{\lambda}_{\tau,\sigma}| \leq  F \lambda_{n,\sigma}, \\
\label{fracture_frictionconditions.3}
  &  \jump{\delta_t\bu}_{\tau,\sigma} \cdot \boldsymbol{\lambda}_{\tau,\sigma}  - F \lambda_{n,\sigma} |\jump{\delta_t\bu}_{\tau,\sigma}|=0. 
\end{align}
\end{subequations}
\end{lemma}

\begin{proof}
We first notice that, by selecting $\boldsymbol{\mu}\in \bm{C}_{\D_\bu^f}(\lambda_n)$ equal to $\boldsymbol{\lambda}$ for all time steps and all faces in $\faces_{\D_\bu}$ except one time step and one face (which is possible since $\boldsymbol{\lambda}\in \bm{C}_{\D_\bu^f}(\lambda_n)$), \eqref{GD_meca_var} implies: for all $t\in[0,T]$, all $\sigma\in\faces_{\D_\bu}$ and all $\boldsymbol{\mu}\in\R^d$ such that $\mu_n\ge 0$ and $|\boldsymbol{\mu}_\tau|\le F\lambda_{n,\sigma}(t)$,
\begin{equation}\label{eq:coulomb.local}
(\mu_n-\lambda_{n,\sigma}(t))\jump{\bu}_{n,\sigma}(t) + (\boldsymbol{\mu}_\tau-\boldsymbol{\lambda}_{\tau,\sigma}(t))\cdot\delta_t\jump{\bu}_{\tau,\sigma}(t)\le 0.
\end{equation}
Conversely, the local relations \eqref{eq:coulomb.local} clearly imply the integrated variational inequality \eqref{GD_meca_var}. 

From hereon, we drop the explicit mention of the time $t$ for simplicity. Choosing $\boldsymbol{\mu}_\tau=\boldsymbol{\lambda}_{\tau,\sigma}$ we see that \eqref{eq:coulomb.local} implies $(\mu_n-\lambda_{n,\sigma})\jump{\bu}_{n,\sigma}\le 0$ for all $\mu_n\ge 0$. This shows that the linear map $\mu_n \mapsto \mu_n\jump{\bu}_{n,\sigma}-\lambda_{n,\sigma}\jump{\bu}_{n,\sigma}$ is negative on $[0,+\infty)$, which forces both its slope $\jump{\bu}_{n,\sigma}$ and its intercept $-\lambda_{n,\sigma}\jump{\bu}_{n,\sigma}$ to be negative; hence $\jump{\bu}_{n,\sigma}\le 0$ and, since $\lambda_{n,\sigma}\ge 0$, the condition $-\lambda_{n,\sigma}\jump{\bu}_{n,\sigma}\le 0$ implies $\lambda_{n,\sigma}\jump{\bu}_{n,\sigma}=0$. This concludes the proof of \eqref{fracture_frictionconditions.1}. 

The condition \eqref{fracture_frictionconditions.2} comes from $\boldsymbol{\lambda} \in \bm{C}_{\D_\bu^f}(\lambda_n)$. To prove \eqref{fracture_frictionconditions.3}, we select $\mu_n=\lambda_{n,\sigma}$ and $\boldsymbol{\mu}_\tau= F\lambda_{n,\sigma}\boldsymbol{e}$ where $\boldsymbol{e}$ is a unit vector such that $\boldsymbol{e}\cdot \delta_t\jump{\bu}_{\tau,\sigma}=|\delta_t\jump{\bu}_{\tau,\sigma}|$
(this choice of $\boldsymbol{\mu}$ is valid since $\boldsymbol{\lambda}\in \bm{C}_{\D_\bu^f}(\lambda_n)$). Then \eqref{eq:coulomb.local} implies $F\lambda_{n,\sigma}|\delta_t\jump{\bu}_{\tau,\sigma}|-\boldsymbol{\lambda}_{\tau,\sigma}\cdot\delta_t\jump{\bu}_{\tau,\sigma}\le 0$. But the converse inequality is trivially true since $|\boldsymbol{\lambda}_{\tau,\sigma}|\le F\lambda_{n,\sigma}$, which proves \eqref{fracture_frictionconditions.3}.

This concludes the proof that \eqref{GD_meca_var} implies \eqref{fracture_frictionconditions}. The converse implication is obtained by direct estimates: under \eqref{fracture_frictionconditions}, for $\boldsymbol{\mu}\in \bigtimes_{k=1}^N \bm{C}_{\D_\bu^f}(\lambda_n)$ we have
$\mu_n\ge 0$ and $|\boldsymbol{\mu}_\tau|\le F\lambda_n$ on $[0,T]\times\Gamma$ and thus
\begin{align*}
\int_0^T \int_\Gamma{}& (\mu_n - \lambda_n) \jump{\bu}_{n} \d\sigma\d t + \int_0^T \int_\Gamma(\boldsymbol{\mu}_\tau - \boldsymbol{\lambda}_\tau) \cdot \delta_t \jump{\bu}_\tau \d\sigma \d t\\
=\int_0^T &{}\int_\Gamma \underbrace{\mu_n \jump{\bu}_{n}}_{\le 0} \d\sigma\d t + \int_0^T \sum_{\sigma\in\faces_{\D_\bu}} \(\underbrace{\boldsymbol{\mu}_{\tau,\sigma}\cdot \delta_t \jump{\bu}_{\tau,\sigma} - F\lambda_{n,\sigma}|\delta_t \jump{\bu}_{\tau,\sigma}|}_{\le 
|\boldsymbol{\mu}_{\tau,\sigma}|\,|\delta_t \jump{\bu}_{\tau,\sigma}| - F\lambda_{n,\sigma}|\delta_t \jump{\bu}_{\tau,\sigma}|\le 0} \) |\sigma|\d t.
\qquad\qedhere
\end{align*}
\end{proof}

\subsection{Energy estimates for the two-phase model, and existence of a solution}

We assume here that the gradient discretizations we consider for the flow are \emph{coercive}~\cite{GDM-poromeca-disc}, that is: there exists $c^*>0$ independent of $\D_p$ such that, for all $v=(v_m,v_f)\in X_{\D_p}^0$,
\begin{multline}\label{eq:GD.coercif}
\|\Pi^m_{\D_p} v_m\|_{L^2(\Omega)} + \|\Pi^f_{\D_p} v_f \|_{L^2(\Gamma)} + \sum_{\aa \in \{+,-\}} \|\mathbb{T}^\aa_{\D_p} v_m\|_{L^2(\Gamma)}\\
\le
c^\star \Big(\|\nabla_{\D_p}^m v\|_{L^2(\Omega)^d} 
+ \|d_{0}^{\nf 3 2}\nabla_{\D_p}^f v\|_{ L^2(\Gamma)^{d-1} } + \sum_{\aa \in \{+,-\}} \|\jump{v}^\aa_{\D_p}\|_{L^2(\Gamma)}
\Big).
\end{multline}

The fracture network is assumed to be such that the Korn inequality holds on $\U_0$; in particular,
this ensures that the following expression defines a norm on $\U_0$, which is equivalent to the $H^1$-norm:
$$
\|\bv\|_{\U_0} = \|\bbeps(\bv)\|_{L^2(\Omega,\mathcal S_d(\mathbb R))}.
$$
The Korn inequality is known to hold if the boundary of each connected component of $\Omega{\setminus}\Gamma$ has a nonzero measure intersection with $\partial\Omega$ (see e.g.~\cite[Section~1.1]{ciarlet}). 

We also assume that $X_{\D_\bu^m}^0$ satisfies the following discrete inf-sup condition, in which the infimum and supremum are taken over nonzero elements of $X_{\D_\bu^f}$ and $X_{\D_\bu^m}^0$, respectively, and $c_\star$ does not depend on the mesh:
\begin{equation}\label{infsup}
\inf_{\boldsymbol{\mu}} \sup_{\bv} {\int_\Gamma \boldsymbol{\mu}\cdot\jump{\bv} \over \|{\bv}\|_{\U_0} \|\boldsymbol{\mu}\|_{H^{-1/2}(\Gamma)^d}} \geq c_{\star}  > 0. 
\end{equation}
We note that this inf-sup condition holds if $\D_\bu^m$ corresponds to the conforming $\P_1$ bubble or $\P_2$ finite elements on a regular triangulation of $\Omega\backslash\Gamma$, and $\faces_{\D_\bu}$ is made of the traces on $\Gamma$ of that triangulation. {Indeed, the proof of~\eqref{infsup} can be readily adapted from that of~\cite[Lemma~6.3]{BR2003}, obtained for a mixed $\P_1$ bubble--$\P_0$ formulation, given that the space of $\P_1$ bubble functions is a subspace of $\bb P_2$. Since a
vector Lagrange multiplier is considered here, the same arguments as in~\cite{BR2003} can be followed upon splitting the space of scalar discrete Lagrange multipliers into subspaces, each one associated with a fracture $\Gamma_i$ in the network (vanishing outside $\Gamma_i$), and then working component-wise in the local reference frame on $\Gamma_i$ for any $i\in I$.
}
\subsubsection{Energy estimates for the gradient scheme}

\begin{theorem}[Energy estimates for \eqref{eq:GS}]
If $(p^\alpha)_{\alpha\in\{\g,\l\}}$, $\bu$ and $(\boldsymbol{\lambda}^k)_{k=0,\ldots,N}$ solve the gradient scheme \eqref{eq:GS}, and if $\phi_\D\ge\phi_{\rm min}\ge 0$, then there exists $C\ge 0$ depending only on the data in Assumptions~\ref{first.hyp}--\ref{last.hyp} (except the Biot coefficient $b$, the Biot modulus $M$, and the damage rock coefficients $\bar d_\aa$, $\bar\phi_\aa$), and on $c^\star,c_\star$, such that
\begin{equation}\label{apriori.est}
\begin{aligned}
\|\grad_{\D_p}^m p_m^\alpha\|_{L^2((0,T)\times\Omega)} \le C, 
&\quad& \| d_{f,\D_\bu}^{\nf 3 2} \grad_{\D_p}^f p_f^\alpha\|_{L^2((0,T)\times\Gamma)} \le C, \\
\| \jump{p^\a}^\aa_{\D_p} \|_{L^2((0,T)\times\Gamma)} \le C,
&\quad& \bar d_\aa\bar\phi_\aa \|U_\aa(\mathbb{T}^\aa_{\D_p}p_{c,m})\|_{L^\infty(0,T;L^1(\Gamma))} \le C,\\
\phi_{\rm min}\|U_m(\Pi^m_{\D_p}p_{c,m})\|_{L^\infty(0,T;L^1(\Omega))} \le C,
&\quad&
\|d_0 U_f(\Pi^f_{\D_p}p_{c,f})\|_{L^\infty(0,T;L^1(\Gamma))} \le C,\\
\frac{1}{\sqrt{M}}\|\Pi^m_{\D_p} p^E_m\|_{L^\infty(0,T;L^2(\Omega))} \le C, 
&\quad&
\| \bbeps(\bu)\|_{L^\infty(0,T;L^2(\Omega,\S_d(\R)))} \le C,\\
\|d_{f,\D_\bu}\|_{L^\infty(0,T;L^4(\Gamma))}\le C &\quad& \|\boldsymbol{\lambda}\|_{L^2(0,T;H^{-1/2}(\Gamma))}\le C.
\end{aligned}
\end{equation}
\label{th:energy_estimates}
\end{theorem}

\begin{proof}

We first prove that, for all $\sigma\in\faces_{\D_\bu}$,
\begin{equation}\label{eq:energy:coulomb.1}
\jump{\delta_t\bu}_{\tau,\sigma} \cdot \boldsymbol{\lambda}_{\tau,\sigma} \geq 0 \quad \mbox{ and }\quad \jump{\delta_t\bu}_{n,\sigma}   \lambda_{n,\sigma} \geq 0\qquad\mbox{ on $(0,T]$}.
\end{equation}
The first relation follows from \eqref{fracture_frictionconditions.3} in which $F \lambda_{n,\sigma} |\jump{\delta_t\bu}_{\tau,\sigma}|\ge 0$ due
to the sign condition on $\lambda_{n,\sigma}$ in \eqref{fracture_frictionconditions.1}. To prove the second relation, we note that, for a time $t\in (t_k,t_{k+1}]$ (for some $k=0,\ldots,N-1$) and owing to our interpretation of $\bu$ and $\lambda_n$ as piecewise constant functions in time,
$$
\jump{\delta_t\bu}_{n,\sigma}(t)  \lambda_{n,\sigma}(t)=\frac{1}{\dtk}(\jump{\bu^{k+1}}_{n,\sigma}-\jump{\bu^k}_\sigma)  (\lambda^{k+1})_{n,\sigma}.
$$
The last relation in \eqref{fracture_frictionconditions.1} at time $t$ imposes $\jump{\bu^{k+1}}_{n,\sigma}  (\lambda^{k+1})_{n,\sigma}=0$, while the first one gives $(\lambda^{k+1})_{n,\sigma}\ge 0$ and the second one at time $t_k$ yields $\jump{\bu^k}_{n,\sigma}\le 0$. This concludes the proof that $\jump{\delta_t\bu}_{n,\sigma}(t) \lambda_{n,\sigma}(t)\ge 0$.

\medskip

Setting $\bv=\delta_t\bu$ in \eqref{GD_meca}, the relations \eqref{eq:energy:coulomb.1} show that $\int_\Gamma \boldsymbol{\lambda}\cdot\jump{\delta_t\bu}~\d\sigma \ge 0$, which yields
\begin{equation}\label{eq:energy:coulomb.2}
   \int_0^T \int_\Omega \( \bbsig(\bu) : \bbeps(\delta_t\bu)  
    - b ~\Pi_{\D_p}^m p_m^E~  \div(\delta_t\bu)\)  \d\x \d t
    +  \int_0^T \int_\Gamma \Pi_{\D_p}^f p_f^E~  \jump{\delta_t\bu}_{n,\faces} \d\sigma\d t 
    \le \int_0^T \int_\Omega \mathbf{f} \cdot  \delta_t\bu ~\d\x \d t.
\end{equation}
We then follow the arguments in~\cite[Lemma 4.3]{GDM-poromeca-disc}. Taking $\varphi^\alpha=p^\alpha$ in \eqref{GD_hydro}, summing over the phases and adding \eqref{eq:energy:coulomb.2}, and accounting for the fact that \eqref{GD_closures} and \eqref{fracture_frictionconditions.1} ensure that $d_{f,\D_\bu}\ge d_0$, we obtain~\cite[Eq.~(25)]{GDM-poromeca-disc} (using~\cite[Eq.~(21)]{GDM-poromeca-disc} to keep track of the damaged rock coefficients $\bar d_\aa$, $\bar\phi_\aa$ which could be zero here), namely
\begin{align}
&\int_0^T \int_\Omega \delta_t(\phi_\D U_m(\Pi^m_{\D_p} p_{c,m})) \, \d\mathbf x \d t
  + \int_0^T \int_\Gamma \delta_t (d_{f,\D_\bu} U_f(\Pi^f_{\D_p} p_{c,f})) \, \d\sigma \d t \nonumber\\
  & + \sum_{\aa \in \{+,-\}} \int_0^T \int_\Gamma \bar d_\aa\bar\phi_\aa\delta_t U_\aa(\mathbb{T}^\aa_{\D_p} p_{c,m}) \, \d\sigma \d t
 + \int_0^T \int_\Omega \delta_t\left(\frac{1}{2}\bbsigma(\bu):\bbeps(\bu) 
+ \frac{1}{2M}(\Pi^m_{\D_p}p^E_m)^2\right)\, \d\x \d t  \nonumber\\
&+\sum_{\alpha} \int_0^T\int_\Omega |\grad^m_{\D_p}p^\alpha_m|^2 \, \d\mathbf x \d t + 
\sum_{\alpha} \int_0^T\int_\Gamma | d_{f,\D_\bu}^{\nf 3 2} \grad^f_{\D_p}p^\alpha_f |^2 \, \d\sigma \d t 
 +\sum_\alpha \sum_{\aa \in \{+,-\}}\int_0^T \int_\Gamma  |\jump{p^\a}^\aa_{\D_p}|^2 \d\sigma \d t\nonumber\\
& \le C\left( \int_0^T \int_\Omega \mathbf f \cdot \delta_t \bu \, \d\mathbf x \d t +
\sum_{\alpha} \int_0^T \int_\Omega h_m^\alpha \Pi^m_{\D_p}p^\alpha_m\,\d\mathbf x \d t + \sum_{\alpha} \int_0^T \int_\Gamma h_f^\alpha \Pi^f_{\D_p}p^\alpha_f\,\d\sigma \d t \right)
\label{GD_energy_estimate}
\end{align}
where $C$ depends only on the data in Assumptions (H), $(b,M,\bar d_\aa,\bar \phi_\aa)$ excepted. From here, the same arguments as in~\cite{GDM-poromeca-disc} provide the estimates on $(p,\bu)$ stated in the theorem. Using these estimates, the definition \eqref{GD_closures} of $(p^E_\omega)_{\omega=m,f}$, the coercivity bound \eqref{eq:GD.coercif} and the fact that $0\le U_\omega(p)\le 2|p|$, we also obtain a bound on $\|\Pi_{\D_p}^f p^E_f\|_{L^2(0,T;L^2(\Gamma))}$ and $\|\Pi_{\D_p}^m p^E_m\|_{L^2(0,T;L^2(\Omega))}$.

To estimate $\boldsymbol{\lambda}$, we use the inf-sup condition \eqref{infsup} to find, for each $k=1,\ldots,N$, $\bv^k\in X_{\D_\bu^m}^0$ such that
$\|\bv^k\|_{\U_0}=\|\boldsymbol{\lambda}^k\|_{H^{-1/2}(\Gamma)}$ and
$$
\int_\Gamma \boldsymbol{\lambda}^k\cdot\jump{\bv^k}\ge c_\star \|\boldsymbol{\lambda}^k\|_{H^{-1/2}(\Gamma)}^2.
$$
Using $\bv=(\bv^k)_{k=1,\ldots,N}$ as a test function in \eqref{GD_meca} (note that $\bv^0$ actually does not play any role in this relation) and invoking the estimates established above on $\bu$, $\Pi_{\D_p}^mp_m^E$ and $\Pi_{\D_p}^fp^E_f$, together with the trace inequality in $\U_0$ and an $L^2$-projection bound to write $\|\jump{\bv^k}_{n,\faces}\|_{L^2(\Gamma)}\le C\|\bv^k\|_{\U_0}$, we infer the estimate on the Lagrange multiplier.
\end{proof}
\begin{remark}[Energy estimate for the modified two-phase flow model]\label{rem:EEmodified2phase}
  The gradient scheme \eqref{eq:GS} can readily be adapted to the modified two-phase flow model of Remark \ref{rem:modified2phase}: the term $\delta_t (\phi_\D \Pi_{\D_p}^m s^\alpha_m )$ in \eqref{GD_hydro} should simply be replaced by
  $$
   (\Pi_{\D_p}^m \phi^0_m)~ \delta_t \Pi_{\D_p}^m s^\alpha_m + (\Pi_{\D_p}^m s^\alpha_m)~ \delta_t \phi_\D
  $$
  and, of course, the new definition \eqref{PE_WA} of the equivalent pressure should be used.
  An inspection of the arguments in~\cite[Lemma 4.3]{GDM-poromeca-disc} shows that the estimate \eqref{GD_energy_estimate} still holds replacing $\phi_\D U_m(\Pi^m_{\D_p} p_{c,m})$ with $(\Pi_{\D_p}^m \phi^0_m)~ U_m(\Pi^m_{\D_p} p_{c,m})$,
  which yields the energy estimates \eqref{apriori.est} under the assumption
  \begin{equation}\label{eq:assum.initial.poro}
  \Pi_{\D_p}^m \phi^0_m \ge\phi_{\rm min}\ge 0
  \end{equation}
  on the initial porosity (a datum of the model) rather than on the current porosity $\phi_\D$ (an unknown of the model). 
\end{remark}

\subsubsection{Existence result for the gradient scheme}

As shown in Theorem \ref{th:energy_estimates}, obtaining estimates on the solution to the gradient scheme (which is the first step to showing the existence of said solution) requires a non-negativity assumption on $\phi_{\D}$. The model itself does not ensure such a property due to the small porosity variations assumption on which this linear poroelastic model is based~\cite{coussy}. Going back to large deformations is clearly outside the scope of this work. Alternatively, the non-negativity of the porosity could be imposed as an additional inequality constraint, similarly to the condition $\jump{\bar\bu}_n\le 0$. However, as seen in \eqref{closure_laws} the porosity $\bar \phi_m$ depends on both the displacement field $\bar\bu$ and on the matrix equivalent pressure $\bar p^E_m$, which itself has a nonlinear dependency on the phase pressures $\bar p_m^\alpha$. These dependencies challenge the translation of the non-negativity constraint on $\bar \phi_m$ using a Lagrange multiplier, all the while ensuring that the resulting weak formulation yields (through a suitable inf-sup condition) estimates on this multiplier.
On the other hand, as stated in Remark \ref{rem:EEmodified2phase}, the modified two-phase flow model of Remark \ref{rem:modified2phase} circumvents this issue by a partial linearization of the matrix accumulation term using the initial porosity. It results that the existence of a discrete solution can be proved for this modified model.%
\begin{theorem}[Existence of a discrete solution]\label{th:existence}
Under Assumptions (H), \eqref{eq:GD.coercif}, \eqref{infsup} and \eqref{eq:assum.initial.poro}, there exists at least one solution of the gradient scheme described in Remark \ref{rem:EEmodified2phase} for the modified two-phase flow model of Remark \ref{rem:modified2phase}.
\end{theorem}

\begin{proof}
Let us denote by (GSm) the gradient scheme for the modified two-phase flow model.
The proof uses a topological degree argument~\cite{deimling}. The equations (GSm) written in variational formulation, through the usage of test functions, can be equivalently rewritten as $\mathcal{H}(p^\g,p^\l,\bu,\bm\lambda)=0$ where $\mathcal{H}:\mathcal X\to\mathcal X$ is a continuous function on the space $\mathcal X=(X^0_{\D_p})^{N+1} \times (X^0_{\D_p})^{N+1}\times (X^0_{\D_\bu^m})^{N+1}\times (X_{\D_\bu^f})^{N+1}$ (note that we also include the condition $\bm\lambda\in \bm{C}_{\D_\bu^f}(\lambda_n)$ in the definition of $\mathcal{H}$). This recasting of the scheme's equations is obvious for \eqref{GD_hydro}--\eqref{GD_meca}, but slightly less for the variational inequality \eqref{GD_meca_var}; the arguments given in Step 3 of this proof (see also Remark \ref{rem:meca.zero}) however show how this inequality can be recast into a nonlinear equation consisting in finding the zero of a function (on which the homotopy arguments to follow can be properly applied). 

To establish the existence of a solution to $\mathcal{H}(p^\g,p^\l,\bu,\bm\lambda)=0$, we will transform through continuous homotopies $\mathcal{H}_{\theta}$ the function $\mathcal{H}$ into a function that has a nonzero degree at $0$, ensuring throughout the transformations that any solution to $\mathcal{H}_{\theta}(p^\g,p^\l,\bu,\bm\lambda)=0$ remains uniformly bounded. The topological degree theory then ensures the existence of a solution to (GSm).

\medskip

\emph{Step 1: decoupling the flow and mechanical equations}.

We consider (GSm) with the following substitutions, for $\theta\in [0,1]$:
\begin{equation}\label{eq:step1.theta}
d_{f,\D_\bu}\leadsto{} d_0-\theta \jump{\bu}_{n,\faces},\quad
b\leadsto{} \theta b,\quad
\Pi_{\D_p}^f p^E_f\mbox{ in \eqref{GD_meca}}\leadsto{} \theta \Pi_{\D_p}^f p_f^E.
\end{equation}
We note that \eqref{GD_meca_var} is unchanged, so Lemma \ref{lem:local.coulomb} remains valid for all $\theta\in [0,1]$.
As a consequence, and since they do not depend on $b$, one can easily check that the estimates in Theorem \ref{th:energy_estimates} (adapted to the modified model, see Remark \ref{rem:EEmodified2phase}) are uniformly valid with respect to $\theta$. The only slightly non-trivial element to analyze is how the term involving $\Pi_{\D_p}^f p^E_f$ in \eqref{GD_meca} compensates with a similar term coming in the estimates from the flow equations. Following the arguments and notations in~\cite[Lemma 4.3]{GDM-poromeca-disc}, we see that the substitution $\delta_t d_{f,\D_\bu}\leadsto \delta_t d_0-\theta\delta_t \jump{\bu}_{n,\faces}=-\theta\delta_t \jump{\bu}_{n,\faces}$ comes down to the substitution 
$$
-\int_0^T\int_\Gamma \Pi_{\D_p}^f p^E_f \jump{\delta_t\bu}_{n,\faces}\leadsto -\theta \int_0^T\int_\Gamma \Pi_{\D_p}^f p^E_f\jump{\delta_t\bu}_{n,\faces}
$$
in the last term of~\cite[Eq.~(4.5)]{GDM-poromeca-disc}. When adding the mechanical equations with $\bv=\delta_t\bu$, this term precisely compensates with the one involving $\theta \Pi_{\D_p}^f p_f^E$ in the substituted version of \eqref{GD_meca}, thus ensuring that the introduction of $\theta$ does not impact the estimates. The solutions to (GSm) with \eqref{eq:step1.theta} thus remain uniformly bounded for all $\theta\in [0,1]$.

For $\theta=1$, we recover the original (GSm). For $\theta=0$, the flow \eqref{GD_hydro} and mechanical equations \eqref{GD_meca}--\eqref{GD_meca_var} 
are fully decoupled since, in the closure equations, $d_{f,\D_\bu}=d_0$ and $\phi_\D-\Pi_{\D_p}^m\phi_m^0=\frac{1}{M}\Pi_{\D_p}^m(p^E_m-p_m^{E,0})$ no longer depend on $(\bu,\bm\lambda)$, and $b\Pi_{\D_p}^m p^E_m$ and $\Pi_{\D_p}^f p^E_f$ have disappeared from \eqref{GD_meca}.
Hence, the topological degree of the underlying function will be nonzero (on a ball determined by the uniform \emph{a priori} estimates mentioned above) if the topological degrees of each function corresponding to the decoupled equations is nonzero.

\medskip

\emph{Step 2: topological degree of the transformed flow equations}.

We consider here \eqref{GD_hydro} (including the adaptations of Remark \ref{rem:EEmodified2phase}) with \eqref{eq:step1.theta} and $\theta=0$. We perform the homotopy $S^\g_\rt \leadsto \rho S^\g_\rt$,  $S^\l_\rt \leadsto \rho S^\l_\rt + 1-\rho$ for $\rt \in \{m,f,\aa\}$, $\eta^\a_\aa \leadsto \rho \eta^\a_\aa + (1-\rho) \eta^\a_f$ with $\rho \in [0,1]$. These new saturation and mobility functions satisfy respectively Assump\-tions \ref{second.hyp} and \ref{first.hyp} and the estimates \eqref{apriori.est} therefore remain valid for any  $\rho$; in particular, the bounds therein on $\nabla_{\D_p}^mp_m^\alpha$, $d_{f,\D_\bu}^{\nicefrac32}\nabla_{\D_p}^f p_f^\alpha$ and $\jump{p^\alpha}_{\D_p}^\aa$ are uniform with respect to $\rho$, and yield a uniform bound on the phase pressure unknowns since \eqref{def:XDp.norm} is a norm on $X_{\D_p}^0$.

 For $\rho = 0$, we obtain $s^\g_\rt = 0$, $s^\l_\rt = 1$,  $U_m = U_f = 0$ leading to
$p^E_m = p^\l_m$ and $p^E_f = p^\l_f$. 
Using the (linear) relation between $\phi_\D$ and $p^E_m$, that $d_{f,\D_\bu}=d_0$ is fixed and that (since $\eta^\l_\aa(1)=\eta^\l_f(1)$ and $\eta^\g_\aa(0)=\eta^\g_f(0)$),
$$
Q^\l_{f,\aa}=\Lambda_f\[ \eta^\l_\aa(1) (\jump{p^\l}_{\D_p})^+ - \eta^\l_f(1)(\jump{p^\l}_{\D_p})^-\] = \Lambda_f \eta^\l_f(1) \jump{p^\l}_{\D_p},
$$
and
$$
Q^\g_{f,\aa}=\Lambda_f\[ \eta^\g_\aa(0) (\jump{p^\g}_{\D_p})^+ - \eta^\g_f(0)(\jump{p^\g}_{\D_p})^-\] = \Lambda_f \eta^\g_f(0) \jump{p^\g}_{\D_p},
$$
these equations form a linear square system in $(p^\l_m,p^\l_f,p^\g_m,p^\g_f)$. Since we established that
any solution to this system satisfies an \emph{a priori} estimates, this proves that the underlying function defining this system has a nonzero degree on a ball of radius larger than the bounds provided by these estimates.

\medskip

\emph{Step 3: topological degree of the transformed mechanical equations}.

The mechanical equations are a bit more challenging due to the presence of the variational inequality. We first perform the homotopy $\bm{f}\leadsto \rho\bm{f}$ in the source term and $F\leadsto \rho F$ in the cone $\bm{C}_{\D^f_\bu}(\lambda_n)$. For $\rho$ going from $1$ to $0$, this transforms this cone into 
$$
\bm{K}=\{\bm\mu\in X_{\D_\bu^f}\,:\,\mu_{n}\ge 0\,,\; \bm{\mu}_\tau=0\},
$$
which no longer depends on the solution $\bm\lambda$. Lemma \ref{lem:local.coulomb} remains valid with this transformed $F$, which shows in particular that \eqref{eq:energy:coulomb.2} (in which we remember that, in our current context, $b=0$ and $\Pi_{\D_p}^f p^E_f$ has been removed) still holds with $\rho\bm{f}$ instead of $\bm{f}$. This estimate leads to an upper bound on $\| \bbeps(\bu)\|_{L^\infty(0,T;L^2(\Omega,\S_d(\R)))}$ that does not depend on $\rho$; using then the inf-sup condition \eqref{infsup} as in the proof of Theorem \ref{th:energy_estimates} gives a bound on $\bm{\lambda}$ that does not depend of $\rho$.

For $\rho=0$, we obtain the following equations: find $\bu \in (X^0_{\D_\bu^m})^{N+1}$ and $(\boldsymbol{\lambda}^k)_{k=0,\ldots,N}\in \bm{K}^{N+1}$ such that
\begin{subequations}\label{GD_dec}
\begin{align}
\label{GD_meca_dec}
   \int_0^T \int_\Omega \bbsig(\bu) : \bbeps(\bv)  \d\x \d t
   +  \int_0^T \int_\Gamma  \boldsymbol{\lambda} \cdot  \jump{\bv} \d\sigma\d t
    ={}& 0\quad\forall \bv \in (X^0_{\D_\bu^m})^{N+1},\\
  \label{GD_meca_var_dec}
   \int_0^T \int_\Gamma (\mu_n - \lambda_n) \jump{\bu}_{n} \d\sigma\d t \leq{}& 0\quad\forall(\boldsymbol{\mu}^k)_{k=1,\ldots,N} \in \bm{K}^{N+1}
\end{align}
\end{subequations}
(and similar equations, omitted here, for the initial displacement field $\boldsymbol{u}^0$ and Lagrange multiplier $\boldsymbol{\lambda}^0$).

To describe the final homotopy and properly ensure that it is a continuous one, we now recast these equations in the form of the zero of a function on a vector space. Let us first notice that $\bm{\lambda}\in \bm{K}^{N+1}$ satisfies \eqref{GD_meca_var_dec} if and only if, for all $\bm{\mu}\in\bm{K}^{N+1}$,
\begin{equation}\label{meca:proj}
   \int_0^T \int_\Gamma \left[\bm{\mu} - \bm{\lambda}\right] \cdot \left[(\jump{\bu}+\bm{\lambda})-\bm{\lambda}\right] \d\sigma\d t 
   \leq 0.
\end{equation}
This relation is a characterization of the fact that $\bm{\lambda}\in (X_{\D_\bu^f})^{N+1}$ satisfies $\bm{\lambda}=P_{\bm{K}}(\jump{\bu}+\bm{\lambda})$, where $P_{\bm{K}}$ is the projection on the closed convex cone $\bm{K}^{N+1}$ (interpreted as usual as a space of piecewise-constant functions in time with values in $\bm{K}$) for the $L^2((0,T)\times\Gamma)^d$-inner product. Recasting \eqref{GD_meca_dec} as a linear relation $\mathcal G(\bu,\bm{\lambda})=0$ and defining the vector space $\mathcal X_{\rm mech}= (X_{\D_\bu^m}^0)^{N+1}\times (X_{\D_\bu^f})^{N+1}$, we therefore have the equivalence
\begin{equation}\label{GD_meca_equiv}
\begin{array}{l}
(\bu,\bm{\lambda})\in (X_{\D_\bu^m}^0)^{N+1}\times\bm{K}^{N+1}\\
\mbox{solves \eqref{GD_dec} for the cone $\bm{K}$ in \eqref{GD_meca_var_dec}}
\end{array}
\Longleftrightarrow
\begin{array}{l}(\bu,\bm{\lambda})\in \mathcal X_{\rm mech}\mbox{ solves }\\
\mathcal Z(\bu,\bm{\lambda})=\left(\mathcal G(\bu,\bm{\lambda}),\bm{\lambda}-P_{\bm{K}}(\jump{\bu}+\bm{\lambda})\right)=0.
\end{array}
\end{equation}
The proof now consists in showing that the topological degree, on a large enough ball, of $\mathcal Z$ is nonzero. To do so, we will perform a continuous homotopy $(\mathcal Z_\varpi)_{\varpi\in [0,1]}$ from $\mathcal Z=\mathcal Z_1$ to a function $\mathcal Z_0$ which is linear, in a such a way that the only solution to $\mathcal Z_{\varpi}(\bu,\bm{\lambda})=0$ on $\mathcal X_{\rm mech}$ is the zero element of that space. This will show that $\mathcal Z_0$ is invertible and thus has a nonzero degree.

Let $\bm{K}(\varpi)=\{\bm{\mu}\in X_{\D_\bu^f}\,:\,\varpi\mu_n\ge -(1-\varpi)\,,\;\bm{\mu}_\tau=0\}$. Then $\bm{K}(1)=\bm{K}$ and $\bm{K}(0)=\{\bm{\mu}\in X_{\D_\bu^f}\,:\,\bm{\mu}_\tau=0\}$. Replacing $\bm{K}$ with $\bm{K}(\varpi)$ in the rightmost statement of \eqref{GD_meca_equiv} defines the mapping $\mathcal Z_\varpi$. For $\bm{g}=(g_n,\bm{g}_\tau)\in L^2((0,T)\times\Gamma)^d$ it can easily be checked that, if $\varpi>0$, $P_{\bm{K}(\varpi)}(\bm{g})=(\max(g_n,-\frac{1-\varpi}{\varpi}),0)$, while $P_{\bm{K}(0)}(\bm{g})=(g_n,0)$; this shows, using dominated convergence theorem, that for a fixed $\bm{g}$ the mapping $\varpi\in [0,1]\mapsto P_{\bm{K}(\varpi)}(\bm{g})\in L^2((0,T)\times\Gamma)^d$ is continuous and, using the fact that $P_{\bm{K}(\varpi)}$ is $1$-Lipschitz continuous, that $(\varpi,\bm{g})\in[0,1]\times L^2((0,T)\times\Gamma)^d\mapsto P_{\bm{K}(\varpi)}(\bm{g})$ is continuous. Hence $(\mathcal Z_\varpi)_{\varpi\in[0,1]}$ is a continuous homotopy. Moreover, the expression of $P_{\bm{K}(0)}$ above shows that it is linear, and thus that $\mathcal Z_0$ is also linear.

If $(\bu,\bm{\lambda})\in \mathcal X_{\rm mech}$ is a zero of $\mathcal Z_\varpi$ for some $\varpi\in[0,1]$ then, using the equivalence \eqref{GD_meca_equiv} for this $\varpi$ and setting
$(\bv,\bm{\mu})=(\bu,\bm{0})$ (which is a valid element of $(X_{\D_\bu^m}^0)^{N+1}\times \bm{K}(\varpi)^{N+1}$) in \eqref{GD_dec}, adding together the two relations and recalling that $\bm{\lambda}_\tau=\bm0$ (since $\bm{\lambda}\in \bm{K}(\varpi)^{N+1}$) we obtain $\int_0^T \int_\Omega \bbsig(\bu) : \bbeps(\bu)  \d\x \d t\le 0$ and thus $\bu=\bm0$ by definition of $\bbsig$ and the Korn inequality. Coming back to \eqref{GD_meca_dec} with a generic $\bv$ and using the inf-sup condition \eqref{infsup}, we infer $\bm{\lambda}=\bm0$, which concludes the proof.
\end{proof}

\begin{remark}[Expressing solutions to the decoupled mechanical equations as zeros of a function]\label{rem:meca.zero}
The arguments used in the proof above can also be applied to the decoupled mechanical equations, before the homotopy driven by $\rho$, to express these equations as the zero of a function. Specifically, the recasting of \eqref{GD_meca_var} leads, instead of \eqref{meca:proj}, to: for all $\bm{\mu}\in \bm{C}_{\D_\bu^f}(\lambda_n)^{N+1}$,
$$
   \int_0^T \int_\Gamma \left[\bm{\mu} - \bm{\lambda}\right] \cdot \Bigg[\underbrace{\left(\begin{array}{c}\jump{\bu}_{n}+\lambda_n\\
   \delta_t\jump{\bu}_\tau+\bm{\lambda}_\tau\end{array}\right)}_{=W(\bu,\bm{\lambda})}-\bm{\lambda}\Bigg] \d\sigma\d t 
   \leq 0,
$$
which is equivalent to the equation $\bm{\lambda}-P_{\bm{C}_{\D_\bu^f}(\lambda_n)}(W(\bu,\bm{\lambda}))=0$.
\end{remark}

\begin{remark}[Convergence of the scheme]\label{rem:conv.scheme}
Having established the existence of a solution to the scheme and estimates on this solution, the natural question would be to analyze its convergence as the mesh size and time steps tend to zero. This convergence requires to establish complex compactness results on sequences of approximate solutions; see, e.g., the convergence analysis for the model without contact in~\cite{GDM-poromeca-disc}, which assumes that the fracture width and porosity remain suitably bounded from below. For the models with Coulomb friction considered here (either with or without the small porosity assumption of Remark \ref{rem:modified2phase}), a very challenging element is the variational inequality \eqref{GD_meca_var}: either the Lagrange multiplier or the jump of displacement would need to converge strongly in appropriate spaces. Obtaining such a compactness results probably requires to consider mechanical equations with the usually neglected inertial term $\partial_t^2\bar\bu$ and, even so, getting the compactness of traces of the displacement is not a small affair. We are actually not aware of any work that analyzes the existence of a solution and/or convergence of numerical schemes for the model including Coulomb friction.
\end{remark}
{%
\begin{remark}[Uniqueness of the solution]\label{rem.uniqueness}
Determining the uniqueness of the solution even for the frictional contact problem alone is, in general, still an open question. It can be achieved under some very special circumstances (see e.g. the discussion in the monograph~\cite[Section 1.3.2]{book-contact-2005}, as well as the very recent contribution~\cite{iurlano}), but not under a general framework. Moreover, to the best of our knowledge, uniqueness of the solution has not so far been addressed for the two-phase flow problem alone, even in a discrete setting. Therefore, well-posedness of the coupled problem remains \emph{a fortiori} an open question as well.
\end{remark}
}
\section{Numerical experiments}
\label{num.experiments}
In this section, we present three numerical experiments to evaluate the computational performance of our discretization, as well as its convergence properties. The first two examples validate the discretization of pure contact mechanics, without Darcy flow. In the third example we consider the coupling with a two-phase flow, and we present the simulation of a drying model in a radioactive waste geological storage structure.

For the complete coupled problem, the flow part~\eqref{eq_edp_hydro} is discretized in space by a TPFA cell-centered finite volume scheme with additional face unknowns at matrix--fracture interfaces~\cite{gem.aghili}.
This implementation is based on an upwind approximation of the mobilities, and therefore does not fit into the gradient scheme form \eqref{GD_hydro} (which contains TPFA in case of centered approximations of mobilities). It is however not difficult -- albeit heavier in terms of notations -- to check that the energy estimates and the existence result stated in Theorems \ref{th:energy_estimates} and \ref{th:existence} extend to the scheme based on this upwind approximation of the mobilities. We also note that, although such upwinding is known to be more robust than the centered approximation (especially on coarse meshes), the centered approximation has been shown in~\cite{GDM-poromeca-cont} to be more accurate on finer meshes.

The mechanical part~\eqref{eq_edp_meca} is discretized by second-order finite elements ($\P_2$) for the matrix displacement field~\cite{daim.et.al,jeannin.et.al}, with supplementary unknowns on the fracture faces to account for the discontinuities, coupled with face-wise constant ($\bb P_0$) Lagrange multipliers on fractures, representing normal and tangential stresses, to discretize the frictional contact conditions (more details are given in Section~\ref{sec:NS-Newton} below). When studying the convergence rate in a pure contact mechanics framework, we compute the
$L^2$-norm of the error along the fracture network. To do so, we use the natural $\bb P_2$ reconstruction for the jump of the displacement field $\jump{\bu}$ and a node-based $\bb P_2$ reconstruction for the vector Lagrange multiplier $\bm\lambda$ using the definition of the tractions stemming from the weak form of the problem.

Triangular grids are employed to decompose $\Omega$ (cf.~Figure \ref{mesh_tpfa_p2}). 
\begin{figure}
\centering
\subfloat[Two-phase flow unknowns, $\alpha\in\{\l,\g\}$]{
\includegraphics[scale=1.2]{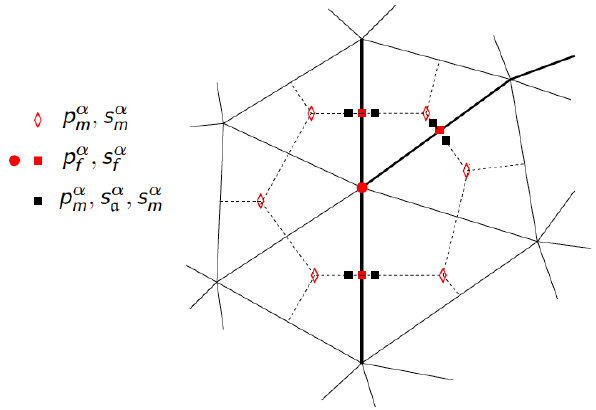}} \qquad\qquad
\subfloat[Mechanics unknowns]{
\includegraphics[scale=1.2]{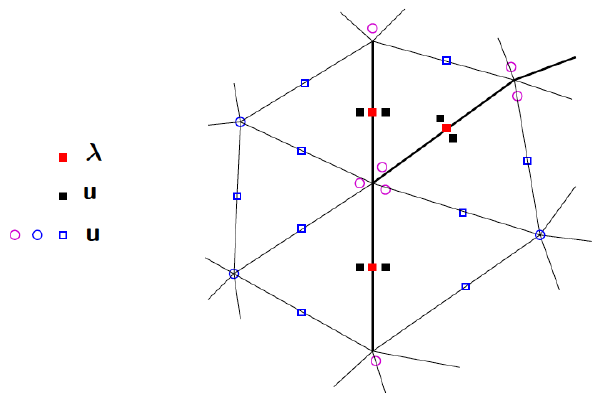}}
\caption{{Example of triangular mesh with three fracture edges in bold. The dot line joining the centers of two cells sharing an edge is orthogonal to the edge. The discrete unknowns of the discontinuous pressure model are presented for the two-phase flow (a) and the mechanics (b). The discontinuities of the pressures, of the saturations and of the displacement are captured at matrix--fracture interfaces. Additional nodal unknowns are defined at the intersections of three or more fractures}.}
\label{mesh_tpfa_p2}
\end{figure}
Let $k\in\N^\star$ denote the time step index. The time step is adaptive, and defined as
$$\dtk = \min \{ {\varrho}\dtkmun , \Delta t^{\max} \},$$
where $\dtzero = 0.001$ days is the initial time step, $\Delta t^{\max} = 10$ years, and ${\varrho} = 1.1$. At each time step, a Newton--Raphson algorithm is used to compute the flow unknowns. At each iteration, the Jacobian matrix is computed analytically and the linear system is solved using a GMRes iterative solver. As in~\cite{gem.aghili}, the matrix--fracture interface pressures are eliminated using a local nonlinear solver. In the case where the Newton--Raphson algorithm does not converge within 50 iterations, the time step is reduced by a factor 2. The stopping criterion is either that the relative residual norm be lower than $10^{-5}$, or that a maximum normalized variation of the primary unknowns be lower than
$10^{-4}$. 

Following~\cite{ecmor}, which considers the case of open fractures, the coupled nonlinear system is solved at each time step using a Newton--Krylov algorithm~\cite{nitsol}. Let us set $p^E = (p^E_m,p^E_f)$ and define the 
functions 
$$
{\bf g}_p: \,
p^E
\underset{\substack{\text{Contact Mechanics} \\[1ex]\rm Solve}}{\rightarrow} \, \bu\, 
\underset{\substack{\text{Darcy} \\[1ex]\rm Solve}}{\rightarrow} \, 
\widetilde  p^E 
$$
and 
$$
{\bf g}_\bu: \, \bu \,\underset{\substack{\rm Darcy \\[1ex]\rm Solve}}{\rightarrow} \, 
p^E
\underset{\substack{\text{Contact Mechanics} \\[1ex]\rm Solve}}{\rightarrow} \, \wt \bu. 
$$
The Newton--Krylov algorithm is either applied to the fixed point $\bu = {\bf g}_\bu(\bu)$
or to $p^E = {\bf g}_p(p^E)$. The second choice is expected to provide a better convergence than its displacement-based counterpart, because of the frictional contact conditions (cf.~the test case in Section~\ref{cas_andra}).
The stopping criterion is fixed at $10^{-6}$ on the relative increment or on the relative residual.
{In the case of open fractures, these Newton--Krylov algorithms are compared in~\cite{ecmor} to the fixed-stress algorithm~\cite{KTJ11} extended to discrete fracture-matrix models in~\cite{GKW16}. They are shown to address the robustness issue of fixed-stress algorithms with respect to small initial time steps in the case of incompressible fluids.}

\subsection{Complementarity functions and non-smooth Newton method}
\label{sec:NS-Newton}
To take account of the local Coulomb contact conditions~\eqref{fracture_frictionconditions}, we recast them in the form of zeros of given complementarity functions, and employ a non-smooth Newton method to compute such zeros. We follow the same arguments as in~\cite[Section~3.2]{contact-norvegiens} (see also~\cite{Wohlmuth08}). Let a face $\sigma\in \faces_{\D_\bu}$ and a time index $k\in\bb N^\star$ be given, and let $\delta\bu = \bu - \bu^{k-1}$ (we write $\bu$ in lieu of $\bu^k$ for simplicity). For a given $c>0$, we define the scalar and vector quantities 
\begin{equation}
\label{friction_bound}
b_{\sigma} = \lambda_{n,\sigma} + c\jump{\bu}_{n,\sigma},\quad \bm a_\sigma = \bm\lambda_{\tau,\sigma} +c\jump{\delta\bu}_{\tau,\sigma},
\end{equation}
where the first one is the so-called \emph{friction bound}. Upon introducing the nonlinear \emph{complementarity functions}
$$\begin{aligned}
\mathcal C_n(\jump{\bu}_{n,\sigma},\lambda_{n,\sigma}) & = \lambda_{n,\sigma} - \max\{0,b_\sigma\}, \\
\bm{\mathcal{C}}_{\bm \tau}({\jump{\bu}}_\sigma, \bm\lambda_{\sigma}) &  =
 \bm\lambda_{\tau,\sigma}\max\{Fb_\sigma,|\bm a_{\sigma}|\} - \bm a_\sigma \max\{0,Fb_\sigma\},
 \end{aligned}$$
it can be shown that~\eqref{fracture_frictionconditions.1} and~\eqref{fracture_frictionconditions.2}--\eqref{fracture_frictionconditions.3} can be rewritten, respectively, as
\begin{equation}
\mathcal C_n(\jump{\bu}_{n,\sigma},\lambda_{n,\sigma}) =0\quad\mbox{ and }\quad
\bm{\mathcal C}_{\bm \tau}({\jump{ \bu}}_\sigma, \bm\lambda_{\sigma})  = \mathbf 0 
\label{friction_newton_system}
\end{equation}
(note that~\eqref{fracture_frictionconditions.3} does not change upon multiplication by the time step $\delta t^k > 0$). 
The nonlinear system of equations resulting from the mechanics contribution is therefore
\begin{equation}
\label{mech_nonlinear_sys}
\rm G(\bu,\bm\lambda) = \bf 0,\quad\text{with}\quad
\rm G(\bu,\bm\lambda) =
\begin{pmatrix}
\underline{\mathbf A}\bu + \underline{\bm \ell}(\bm \lambda) - \underline{\mathbf b} \\
\left[\mathcal C_n(\jump{\bu}_{n,\sigma},\lambda_{n,\sigma})\right]_{\sigma\in\faces_{\D_\bu}}\\
\left[\bm{\mathcal C}_{\bm \tau}({\jump{ \bu}}_\sigma, \bm\lambda_{\sigma})\right]_{\sigma\in\faces_{\D_\bu}}
\end{pmatrix},
\end{equation}
where the first vector equation represents the finite-element version of~\eqref{GD_meca}, in the sense that we have the following correspondence between matrix-/vector-like objects and bilinear/linear forms:
$$
\begin{gathered}
\ul{\bf A} \sim \int_\Omega \bbsig(\bu) : \bbeps(\bv)\d\x, \qquad
\ul{\bm \ell}(\bm\lambda) \sim \int_\Gamma  \boldsymbol{\lambda} \cdot  \jump{\bv} \d\sigma, \\
\ul{\mathbf b} \sim  \int_\Omega \( \mathbf{f} \cdot  \bv +  b ~\Pi_{\D_p}^m p_m^E~  \div\,\bv \)  \d\x 
     - \int_\Gamma \Pi_{\D_p}^f p_f^E~  \jump{\bv}_n \d\sigma.
\end{gathered}
$$
The non-smooth Newton method used to solve~\eqref{mech_nonlinear_sys} is the following. Let $q\in\bb N$ be the iteration index. We split the fracture faces into the following three sets:
\begin{equation}
\label{iteration_sets}
\begin{aligned}
\cc I_n^{q+1} & = \{\sigma\in \faces_{\D_\bu} : b_\sigma^q \le 0\},\\
\cc I_\tau^{q+1} & =  \{\sigma\in \faces_{\D_\bu} : |\bm a_\sigma^q| < Fb_\sigma^q\},\\
\cc A^{q+1} & = \{\sigma\in \faces_{\D_\bu} : |\bm a_\sigma^q| \ge Fb_\sigma^q > 0\}.
\end{aligned}
\end{equation}
Here, $\cc I_n^{q+1}$ contains the faces not in contact, $\cc I_\tau^{q+1}$ the faces in contact but sticking, i.e.~whose friction bound is not reached, and $\cc A^{q+1}$ contains the faces in contact and slipping, i.e.~for which the friction bound is reached. Given the solution $(\bu^q,\bm\lambda^q)$ of~\eqref{mech_nonlinear_sys} at iteration $q$, the new iterates $(\bu^{q+1},\bm\lambda^{q+1})$ are obtained by computing the derivatives of $\mathrm G$ and then updating the solution. We refer to~\cite[Section~3.2.1]{contact-norvegiens} and to~\cite[Section~3]{Wohlmuth08}, where this approach is applied to the Tresca model, for a detailed discussion including a regularization technique to stabilize and improve the convergence of the method which is also applied in all following numerical experiments. 

In all the examples we present, we also compare the computational performance of this method with that of its \emph{active set} counterpart, where equations~\eqref{friction_newton_system} are replaced by simplified equations, which are linear in the open and stick cases (first two sets in~\eqref{iteration_sets}), and piecewise linear in the slip case (third set in~\eqref{iteration_sets}) in a two-dimensional framework.
The zeros of such simplified equations are the same as the zeros of the corresponding original non-smooth Newton equations. As a matter of fact, in the general case of three space dimensions, upon introducing the unit vector $$\mathbf w_{\tau,\sigma} = \frac{\bm\lambda_{\tau,\sigma}
+c\jump{\delta\bu}_{\tau,\sigma}}{| \bm\lambda_{\tau,\sigma}+c\jump{\delta\bu}_{\tau,\sigma} |},$$
equations~\eqref{friction_newton_system} can be simplified in the following way (we drop the iteration index $q$ here):
$$
\begin{gathered}
\cc C_n(\jump{\bu}_{n,\sigma},\lambda_{n,\sigma}) = 0\quad \Longleftrightarrow \quad
\begin{aligned}
\lambda_{n,\sigma} & = 0\ \ \ \hbox{if }\sigma\in \cc I_n, \\
\jump{\bu}_{n,\sigma} & = 0\ \ \ \hbox{if }\sigma\in \cc I_\tau \cup \cc A,
\end{aligned}
\\[2ex]
\bm{\cc C}_\tau(\jump{\bu}_{\sigma},\bm\lambda_{\sigma}) = \bm 0\quad \Longleftrightarrow \quad
\begin{alignedat}{2}
\bm\lambda_{\tau,\sigma} & = \bm 0&\ \ &\ \ \hbox{if }\sigma\in \cc I_n, \\
\jump{\delta\bu}_{\tau,\sigma} & = \bm 0&\ \ &\ \ \hbox{if }\sigma\in \cc I_\tau, \\
\bm\lambda_{\tau,\sigma} - F\lambda_{n,\sigma} \mathbf w_{\tau,\sigma} & = \bm 0
&\ \ &\ \ \hbox{if }\sigma\in \cc A.
\end{alignedat}
\end{gathered}
$$
Notice that, in two dimensions, the equation $\bm\lambda_{\tau,\sigma} - F\lambda_{n,\sigma} \mathbf w_{\tau,\sigma} = \bm 0$ for $\sigma\in \cc A$ turns out to be piecewise linear.
For both algorithms, the iterations are stopped when the relative residual norm is lower than or equal to $10^{-6}$.

Since all the examples discussed in this section are set in a two-dimensional framework, quantities such as $\jump{\bu}_{\tau,\sigma}$ and $\bm\lambda_{\tau,\sigma}$ for a given $\sigma\in\faces_{\D_\bu}$ can be considered scalar, and we remove the bold face from the latter. Let us point out that, although the state given by $\lambda_{n,\sigma} = \lambda_{\tau,\sigma} = 0$ for all $\sigma\in\faces_{\D_\bu}$ and $\bu=\bm 0$ is physically undetermined, in our implementation of the non-smooth Newton and active set methods we assume that this is a \emph{no-contact} state, corresponding to open fractures.
Note that the contact mechanics Jacobian system is computed at each non-smooth Newton or active set iteration using the sparse direct solver UMFPACK~\cite{umd}. Let us refer to~\cite{FRANCESCHINI2019376} for an alternative iterative solver based on a block preconditioner of the saddle point problem. 
Note also that the Lagrange multipliers could be locally eliminated from this system, which is likely to be a key feature for preconditioned iterative solvers as this avoids having to cope with saddle point problems. This property generalizes to any conforming discretization of the displacement field containing a bubble function associated to each fracture face.
{%
\begin{remark}[Nonsingularity of the Jacobian matrix]\label{rem:nonsingularity}
The regularized non-smooth Newton and the active set method both yield a nonsingular Jacobian matrix in the case of Tresca's friction model. In general, one cannot prove nonsingularity when considering Coulomb's model -- the proof can possibly be achieved assuming a sufficiently small friction coefficient, but we didn't come across any issues in our numerical implementation.
A possible workaround could be a Tresca-like fixed-point iterative method (see e.g.~\cite{Wohlmuth08}). Notice also that singularity issues are not raised in~\cite{contact-norvegiens} nor in~\cite{Wohlmuth08}.
\end{remark}
}
\subsection{Contact mechanics in the absence of Darcy flow}
\subsubsection{Unbounded domain with single fracture under compression}
\label{tchelepi_meca}
This example was presented in~\cite{contact-BEM,tchelepi-castelletto-2020,GKT16}. It consists of a 2D unbounded domain containing a single fracture and subject to a compressive remote stress $\bar\sigma = 100\,\text{MPa}$ (cf.~Figure~\ref{test_compression}).
The fracture inclination with respect to the horizontal direction is $\psi = 20^{\circ}$, its length is $2\ell = 2$ m, and the friction coefficient is $F=1/\sqrt{3}$. The same values of Young's modulus and Poisson's ratio as in~\cite{tchelepi-castelletto-2020} are used here, i.e.~$E=25$ GPa and $\nu=0.25$. The analytical solution in terms of the Lagrange multiplier $\bar\lambda_n$, representing the normal stress on the fracture up to the sign, and of the jump of the tangential displacement field, is given by
\begin{equation}\label{sol.compression}
\bar\lambda_n = \bar\sigma\sin^2\psi,\quad
|\jump{\bar\bu}_\tau| = \frac{4(1-\nu)}{E}(\bar\sigma\sin\psi(\cos\psi-F\sin\psi))\sqrt{\ell^2-(\ell^2-\tau^2)},
\end{equation}
where $0\le\tau\le2\ell$ is a curvilinear abscissa along the fracture. Note that since $\bar\lambda_n>0$, we have $\jump{\bar\bu}_n=0$ on the fracture. Boundary conditions are imposed on $\bu$ at specific nodes of the mesh, as shown in Figure~\ref{test_compression}, to respect the symmetry of the expected solution. For this simulation, we sample a $320\times320\,\rm m$ square, and carry out uniform refinements at each step in such a way to compute the solution on meshes containing 100, 200, 400, and 800 faces on the fracture (corresponding, respectively, to 12\,468, 49\,872, 199\,488, and 797\,952 triangular elements). The initial mesh is refined in a neighborhood of the fracture; starting from this mesh, we perform global uniform refinements at each step, i.e.~we do not refine further near the fracture.

In Tables~\ref{perfs_tchelepi_1} and~\ref{perfs_tchelepi_2}, we give an insight into the computational performance of the active set and non-smooth Newton algorithms, depending on the initial guess, as well as on the value of the parameter $c$ introduced in~\eqref{friction_bound}. Figure~\ref{compression_comparison} shows the comparison between the analytical and numerical Lagrange multipliers $\lambda_n$ and jumps of the tangential displacement on the fracture $\jump{\bu}_\tau$, computed on the finest mesh. As in~\cite{tchelepi-castelletto-2020}, the difference between the computed displacement and the analytical one is un\-de\-tectable, and the Lagrange multiplier $\lambda_n$ presents some oscillations in a neighborhood of the fracture tips. As already explained in~\cite{tchelepi-castelletto-2020}, this is due to the sliding of faces close to the fracture tips (notice that all fracture faces are in a contact-slip state). Finally, Figure~\ref{compression_error} shows the convergence properties of this discretization, yielding a first-order convergence for the jumps of the displacement field, and a convergence order slightly greater than 2 for the Lagrange multiplier $\lambda_n$. The former rate is related to the low regularity of $\jump{\bar\bu}$ close to the tips (cf.~the analytical expression~\eqref{sol.compression}), the latter is likely related to the fact that $\bar\lambda_n$ is constant. Because of the oscillations of the approximation $\lambda_n$ close to the fracture tips, as in~\cite{tchelepi-castelletto-2020}, we consider the central $90\%$ of the fracture size to compute the norm of the error. Notice also that,
since $\jump{\bar\bu}_n=0$, the relative error on the normal jump on the fracture is not defined. The absolute error of the normal jump converges at order 1, and the face mean values of the normal jump are actually zero up to solver accuracy, as expected. 
\begin{table}
\centering
{

    \begin{tabular}{|c|c|c|c|c|c|c|}
\hline  \multicolumn{7}{|c|}{Initial guess:  $\lambda_{n} = 100\,\rm Pa$, $\lambda_{\tau} = 0$, $\bu = \mathbf 0$ (stick) } \\  \hline
   &  \multicolumn{3}{c|}{Active set} & \multicolumn{3}{c|}{Regularized NS Newton} \\ \hline
$c$ (N/m$^3$) & \multicolumn{3}{c|}{$10^6$, $10^9$, $10^{11}$} & \multicolumn{3}{c|}{$10^6$, $10^9$, $10^{11}$} \\ \hline
NbFracFaces& 100 & 200 &  400 & 100 & 200 &  400 \\ \hline
Iterations & 2 & 2 &  2  &  2 & 2 & 2\\ \hline
    \end{tabular}
}
 \caption{Performance of the active set and regularized non-smooth Newton algorithms for the example of Section~\ref{tchelepi_meca}, for the initial guess $\lambda_{n} = 100\,\rm Pa$, $\lambda_{\tau} = 0$, $\bu = \mathbf 0$. NbFracFaces is the number of faces in the fracture.} 
  \label{perfs_tchelepi_1}
\end{table}

\begin{table}
\centering
{

    \begin{tabular}{|c|c|c|c|c|c|c|}
\hline  \multicolumn{7}{|c|}{Initial guess:  $\lambda_{n} = \lambda_{\tau} = 0$, $\bu = \mathbf 0$ (open)} \\  \hline
   &  \multicolumn{3}{c|}{Active set} & \multicolumn{3}{c|}{Regularized NS Newton} \\ \hline
$c$ (N/m$^3$) & \multicolumn{3}{c|}{$10^{11}$} & \multicolumn{3}{c|}{$10^{11}$} \\ \hline
NbFracFaces& 100 & 200 &  400 & 100 & 200 &  400 \\ \hline
Iterations & 2 & 2 &  2  &  6& 6 & 6\\ \hline
$c$ (N/m$^3$)& \multicolumn{3}{c|}{$10^6$, $10^9$} & \multicolumn{3}{c|}{$10^6$, $10^9$} \\ \hline
NbFracFaces& 100 & 200 &  400 & 100 & 200 &  400 \\ \hline
Iterations & 2 & 2 &  2  &  3& 3 & 3\\ \hline

    \end{tabular}
 }
 \caption{Performance of the active set and regularized non-smooth Newton algorithms for the example of Section~\ref{tchelepi_meca}, for the initial guess $\lambda_{n} = \lambda_{\tau} = 0$, $\bu = \mathbf 0$.  NbFracFaces is the number of faces in the fracture.} 
  \label{perfs_tchelepi_2}
\end{table}

\begin{figure}
\centering

\subfloat[]{
\raisebox{.45cm}{
\includegraphics[keepaspectratio=true,scale=.35]{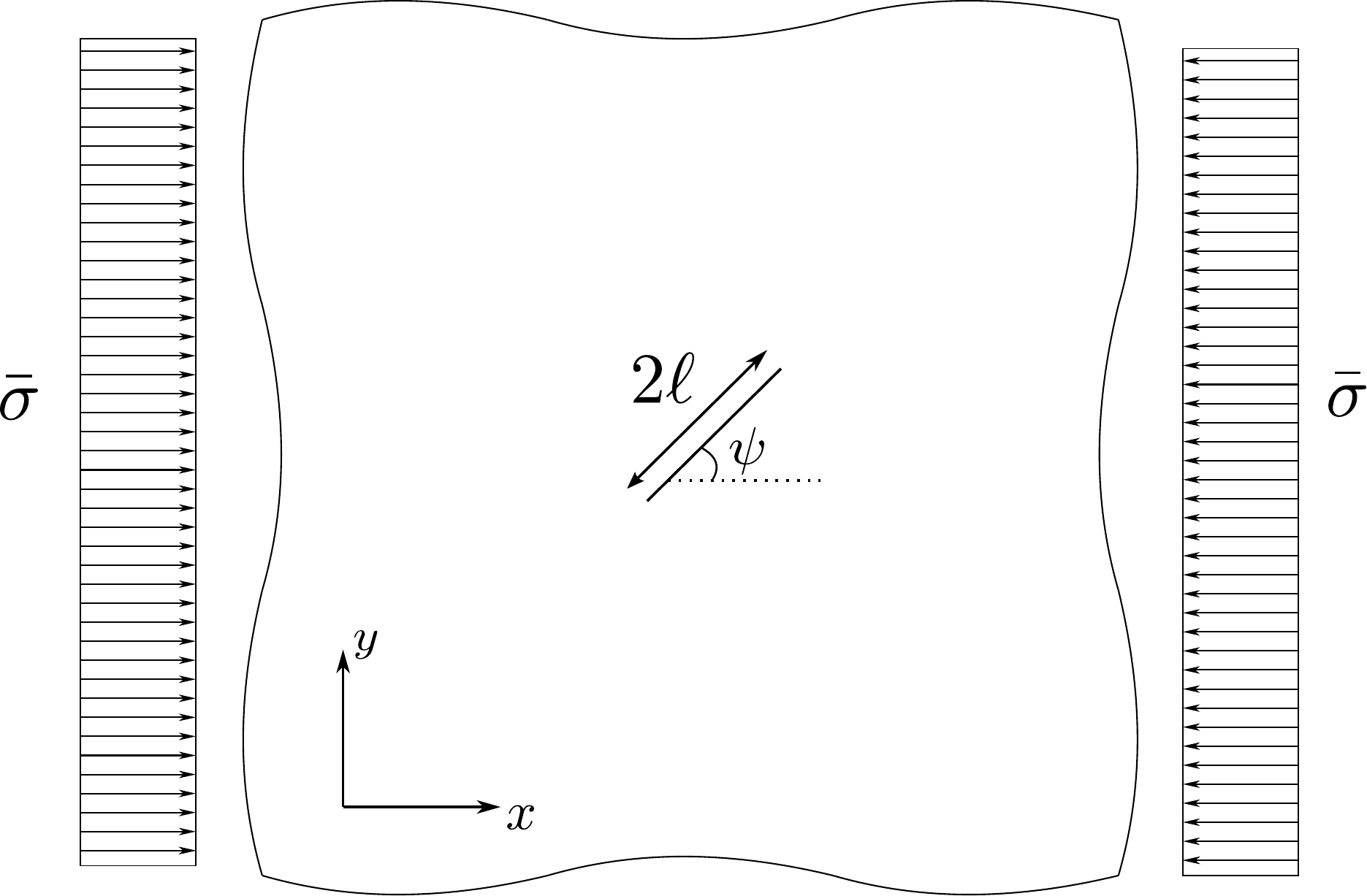}}
}
\hspace{2cm}
\subfloat[]{\includegraphics[keepaspectratio=true,scale=.1]{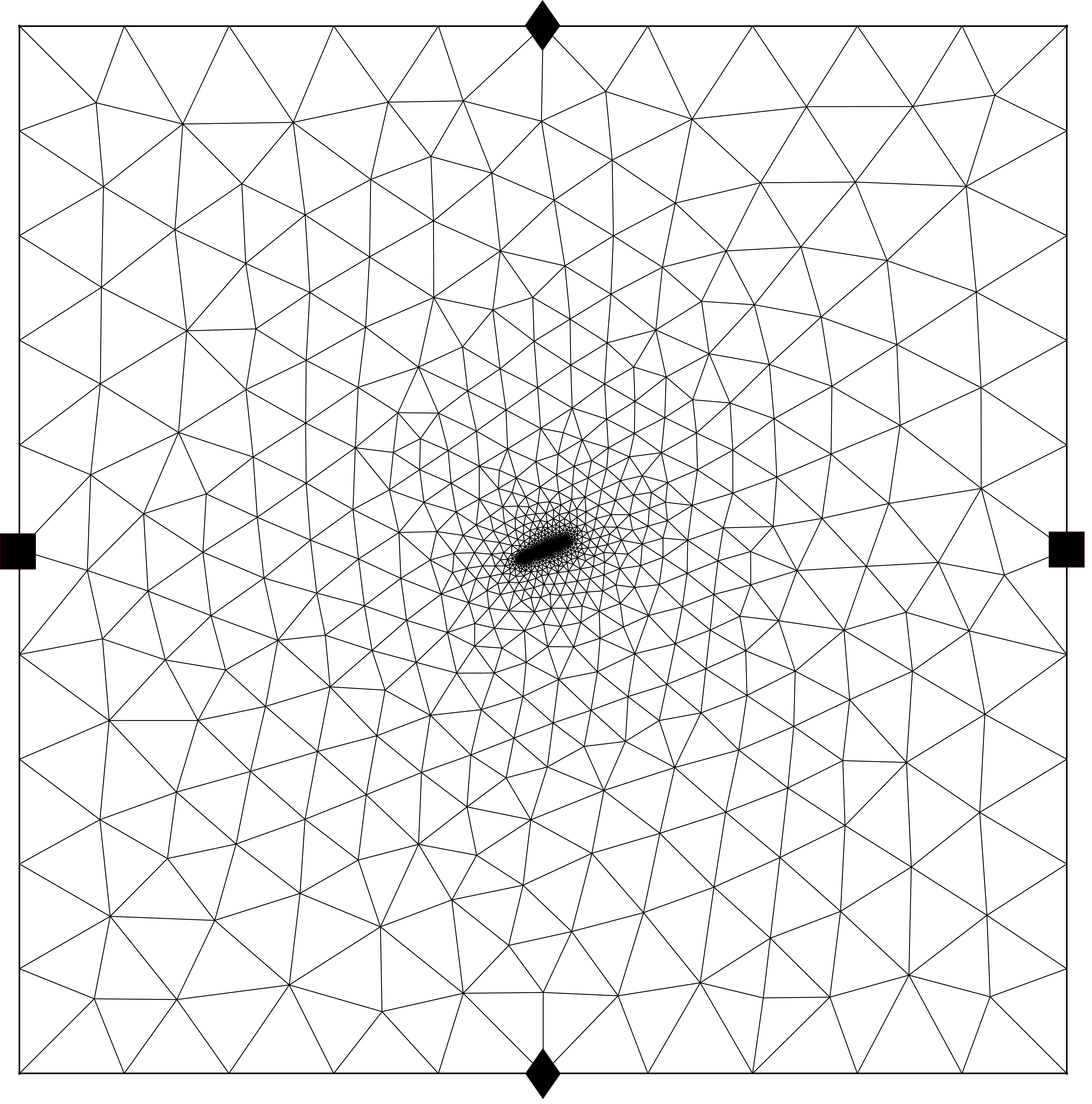}}\\
\caption{Unbounded domain containing a single fracture under uniform compression (a) and mesh including nodes for boundary conditions ($\blacklozenge$: $u_x = 0$, $\blacksquare$: $u_y=0$), for the example of Section~\ref{tchelepi_meca}.}
\label{test_compression}
\end{figure}

\begin{figure}
\centering
\subfloat[]{
\includegraphics[keepaspectratio=true,scale=.635]{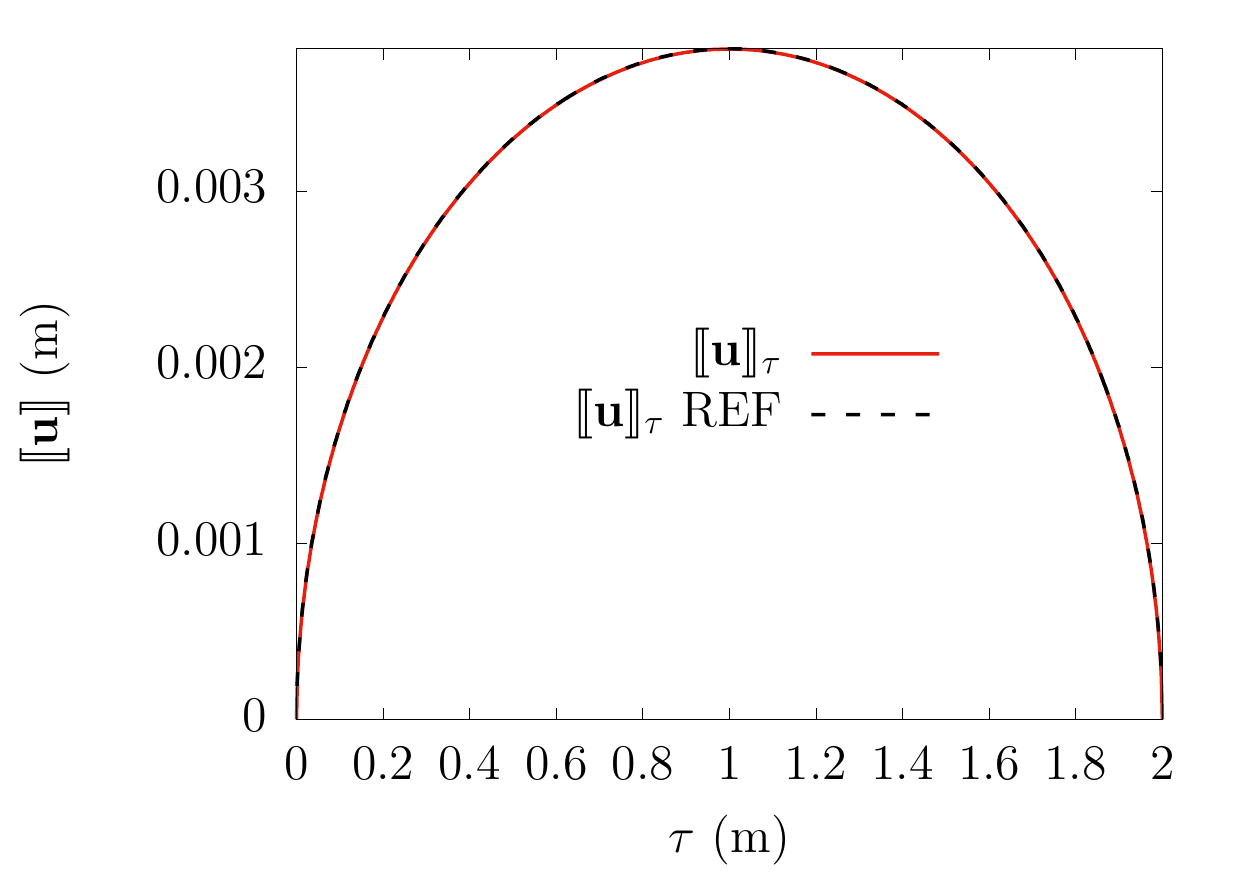}}
\subfloat[]{\includegraphics[keepaspectratio=true,scale=.635]{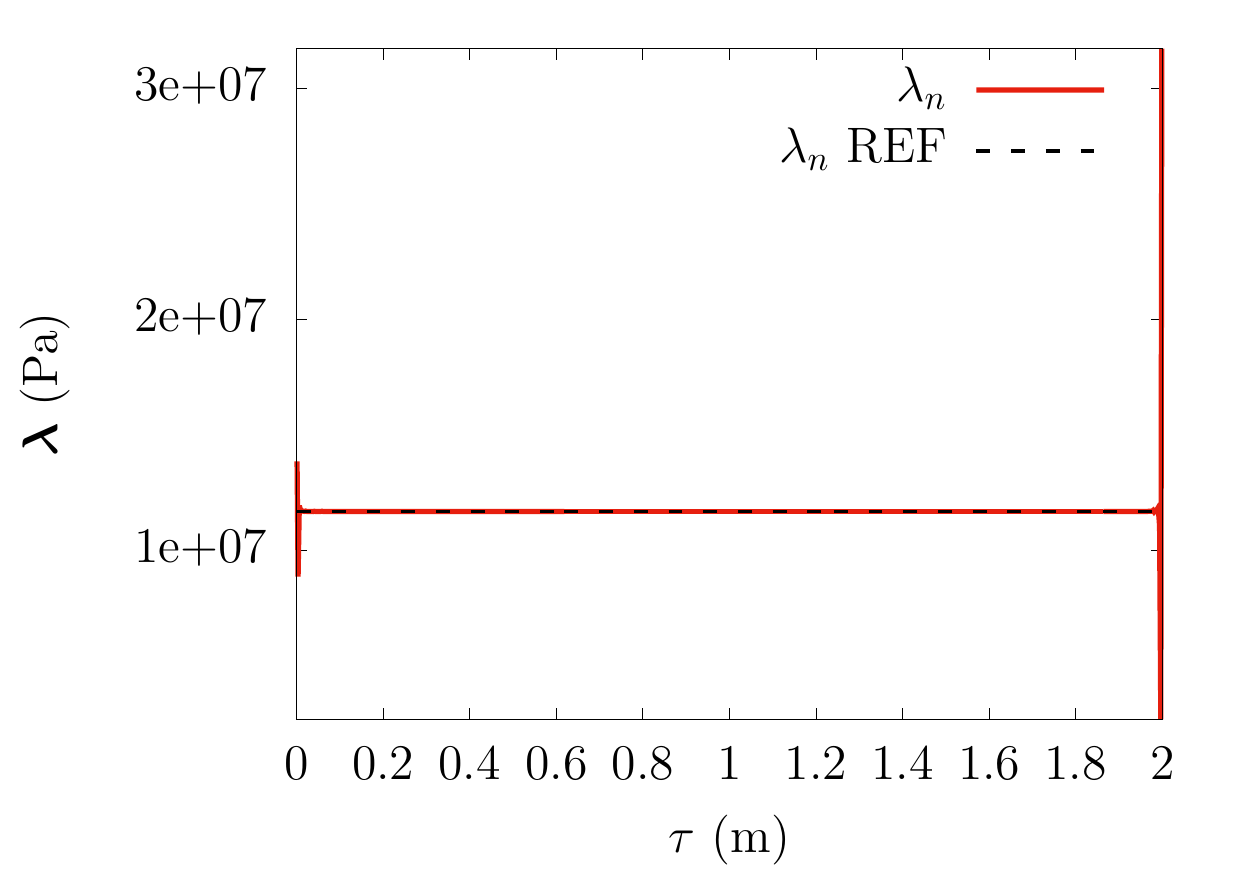}}

\caption{Comparison of the numerical and analytical (labeled as REF) solutions in terms of $\jump{\bu}_\tau$ (a) and $\lambda_n$ (b),
for the example of Section~\ref{tchelepi_meca}.}
\label{compression_comparison}
\end{figure}
\begin{figure}
\centering

\includegraphics[keepaspectratio=true,scale=.65]{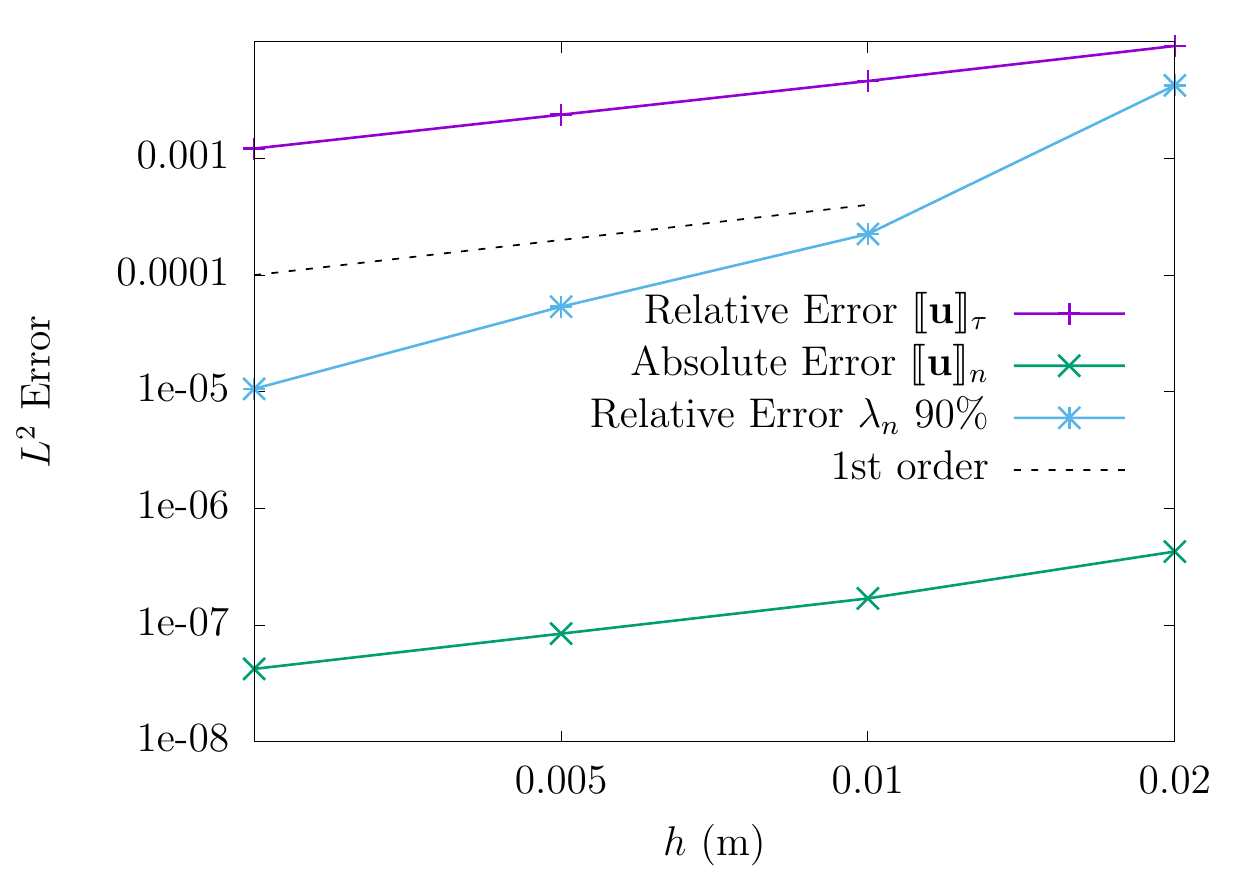}
\caption{Relative and absolute $L^2$ errors between the numerical and analytical solutions in terms of $\jump{\bu}$ and $\bm\lambda$, for the example of Section~\ref{tchelepi_meca}.}
\label{compression_error}
\end{figure}

\subsubsection{Rectangular domain with six fractures}
\label{bergen_meca}
As a second example to illustrate the behavior of our approach, we consider the test case presented in~\cite[Section~4.1]{contact-norvegiens}, where a $2\times1$\,m domain including a network $\Gamma = \bigcup_{i=1}^6 \Gamma_i$ of six fractures is considered, cf.~Figure~\ref{bergen}. Fracture 1 is made up of two sub-fractures forming a corner, whereas one of the tips of fracture 5 lies on the boundary of the domain.

We use the same values of Young's modulus and Poisson's ratio, $E=4$ GPa and $\nu = 0.2$, and the same set of boundary conditions as in~\cite{contact-norvegiens}, that is, the two vertical sides of the domain are free, and we impose $\bu = \bf 0$ on the bottom side and $\bu = [0.005,-0.002]^\top$~m on the top side. The friction coefficient is 
$F_i(\x) = 0.5\left(1+10\exp(-D_i(\x)^2/0.005\,\rm m^2)\right)$, with $i\in\{1,\dots,6\}$ the fracture index, $\x\in\Gamma_i$ a generic point on fracture $i$, and $D_i(\x)$ the minimum distance from $\x$ to the tips of fracture $i$ (the bend in fracture 1 is not considered as a tip). 

Figures~\ref{bergen_u} and~\ref{bergen_lambda} show the fracture aperture $-\jump{\bu}_n$ and sliding $\jump{\bu}_\tau$ against the distance $\tau$ from the tips for each of the six fractures. Except for the sign (recall that the vector Lagrange multiplier is $\bm\lambda = -\bbsig(\bu^+)\n^+$ with $+$ one of the two sides of the domain with respect to a given fracture), our numerical results are in good agreement with the results presented in~\cite[Section~4.1]{contact-norvegiens}, to which we refer for a more detailed discussion about the physical interpretation of these results.

In Tables~\ref{perfs_bergen_1} and~\ref{perfs_bergen_2}, analogously to the previous example, we study the computational performances of the active set and non-smooth Newton methods for this example, depending on the initial guess and letting the parameter $c$ take on three different values.

Since no closed-form solution is available for this test case, to evaluate the convergence of our method we compute a reference solution on a fine mesh made of 730\,880 triangular elements. Figure~\ref{bergen_errors} shows the convergence rates obtained for both $\jump{\bu}$ and $\bm\lambda$. Analogously to the previous example, we perform uniform mesh refinements at each step, and do not refine only in a neighborhood of tips. As in~\cite{contact-norvegiens}, an asymptotic first-order convergence is observed for the vector Lagrange multiplier for all fractures, except fracture 4 which exhibits a convergence rate close to 2 owing to its entire contact-stick state, and fracture 1 which exhibits a lower rate due to the additional singularity induced by the corner. For the jump of the displacement field across fractures, we obtain an asymptotic convergence rate equal to 1.5 for all fractures.
\begin{figure}
\centering
\includegraphics[keepaspectratio=true,scale=.175]{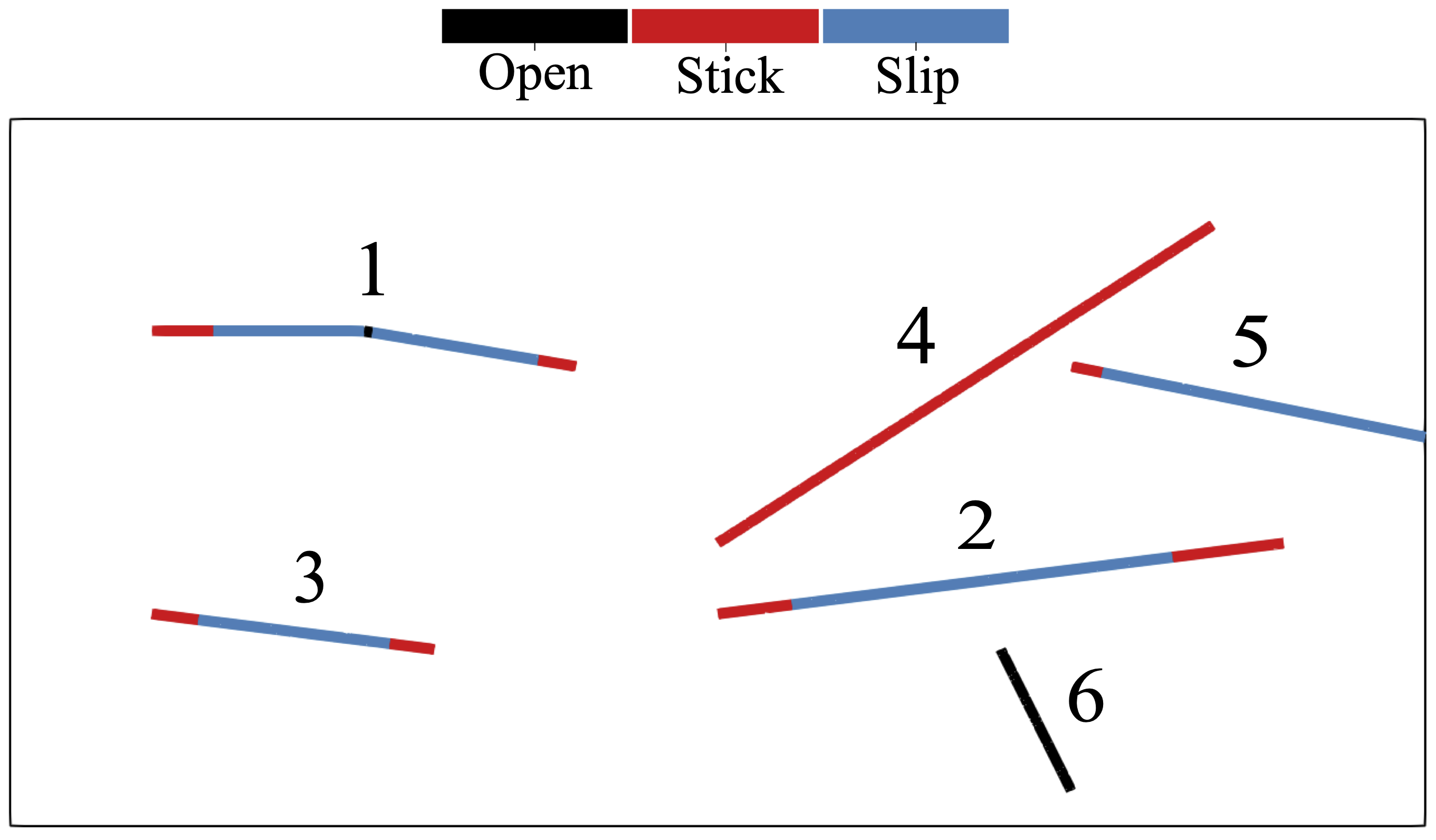}
\caption{Two-dimensional, $2\times1$\,m domain containing six fractures. Fracture 1 comprises two sub-fractures making a corner, and fracture 5 has a tip on the boundary. The contact state of each fracture obtained by the simulation is also shown.}
\label{bergen}
\end{figure}
\begin{table}
\centering
{
    \begin{tabular}{|c|c|c|c|c|c|c|c|c|}
\hline  \multicolumn{9}{|c|}{Initial guess:  $\lambda_{n} = 100$ Pa, $\lambda_{\tau} = 0$, $\bu = \mathbf 0$ (stick) } \\  \hline
   &  \multicolumn{4}{c|}{Active set} & \multicolumn{4}{c|}{Regularized NS Newton} \\ \hline
$c$ (N/m$^3$) & \multicolumn{4}{c|}{$10^6$, $10^9$, $10^{11}$} 
& \multicolumn{4}{c|}{$10^6$, $10^9$, $10^{11}$} \\ \hline
NbCells & 2\,855 & 11\,420 &  45\,680 & 182\,720 & 2\,855 & 11\,420 &  45\,680 & 182\,720 \\ \hline
Iterations & 3 & 6 &  6  & 8 & 3 & 6 & 6 & 8 \\ \hline
    \end{tabular}
  }
 \caption{Performance of the active set and regularized non-smooth Newton algorithms for the example of Section~\ref{bergen_meca}, for the initial guess $\lambda_{n} = 100\,\rm Pa$, $\lambda_{\tau} = 0$, $\bu = \mathbf 0$.  NbCells is the number of cells in the mesh.} 
  \label{perfs_bergen_1}
\end{table}

\begin{table}
\centering
{
  \resizebox{14cm}{!}{
    \begin{tabular}{|c|c|c|c|c|c|c|c|c|}
\hline  \multicolumn{9}{|c|}{Initial guess:  $\lambda_{n} = \lambda_{\tau} = 0$, $\bu = \mathbf 0$ (open)} \\  \hline
   &  \multicolumn{4}{c|}{Active set} & \multicolumn{4}{c|}{Regularized NS Newton} \\ \hline
$c$ (N/m$^3$) & \multicolumn{4}{c|}{$10^{11}$} & \multicolumn{4}{c|}{$10^{11}$} \\ \hline
NbCells & 2\,855 & 11\,420 &  45\,680 & 182\,720  & 2\,855 & 11\,420 &  45\,680 & 182\,720  \\ \hline
Iterations & 8 & no conv. &  no conv. & 10 &  10 & 9 & 10 & 10  \\ \hline
$c$ (N/m$^3$) & \multicolumn{4}{c|}{$10^{9}$} & \multicolumn{4}{c|}{$10^9$} \\ \hline
NbCells & 2\,855 & 11\,420 &  45\,680 & 182\,720 & 2\,855 & 11\,420 &  45\,680 & 182\,720  \\ \hline
Iterations & 7 & 9 &  8 & 9 &  11 & 11 & 12 & 29  \\ \hline
$c$ (N/m$^3$) & \multicolumn{4}{c|}{$10^{6}$} & \multicolumn{4}{c|}{$10^6$} \\ \hline
NbCells & 2\,855 & 11\,420 &  45\,680 & 182\,720  & 2\,855 & 11\,420 &  45\,680 & 182\,720 \\ \hline
Iterations & 7 & 9 &  8 & 9 &  no conv. & 18 & no conv. & no conv.  \\ \hline
    \end{tabular}
    }
  }
 \caption{Performance of the active set and regularized non-smooth Newton algorithms for the example of Section~\ref{bergen_meca}, for the initial guess $\lambda_{n} = \lambda_{\tau} = 0$, $\bu = \mathbf 0$. 
 NbCells is the number of cells in the mesh.} 
  \label{perfs_bergen_2}
\end{table}

\begin{figure}
\captionsetup[subfigure]{labelformat=empty}
\centering

\subfloat[]{
\includegraphics[keepaspectratio=true,scale=.65]{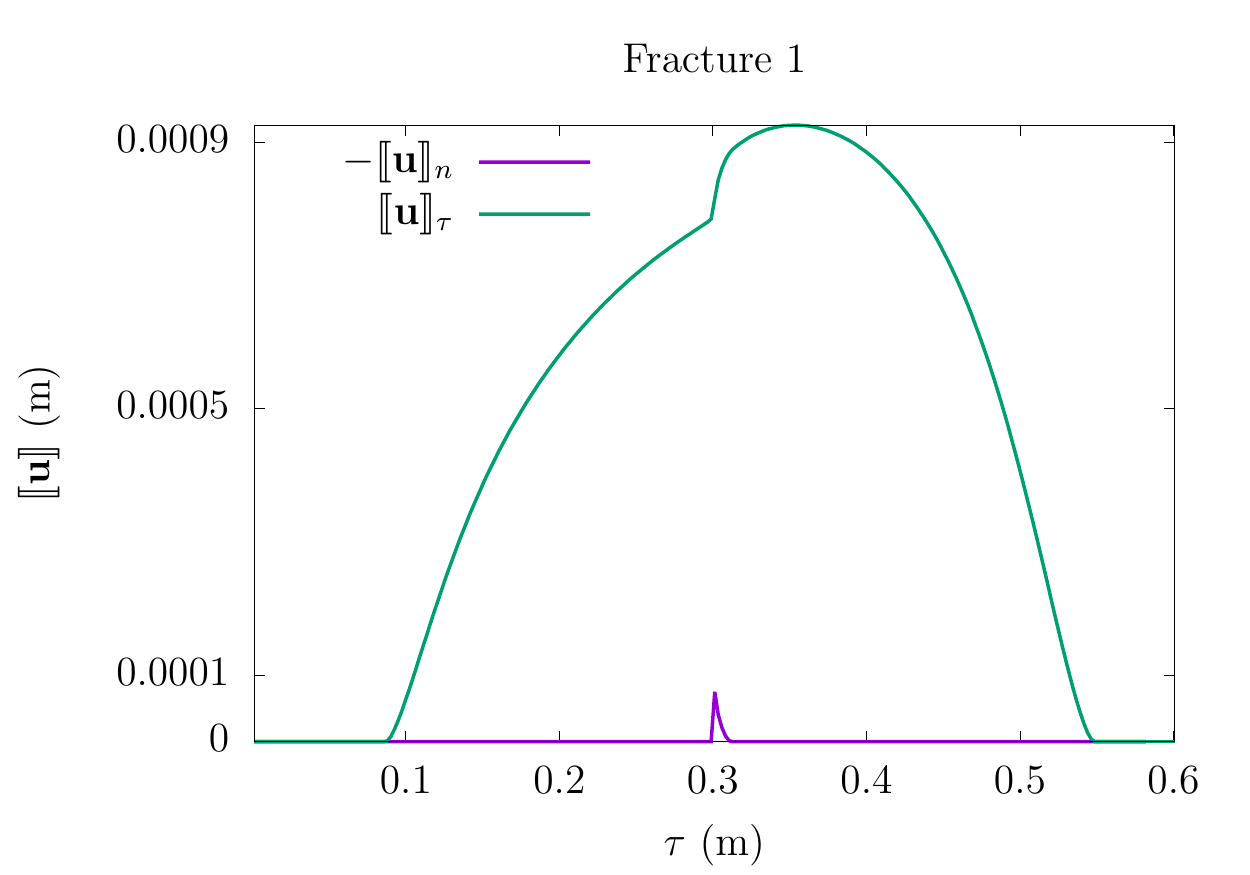}}
\subfloat[]{
\includegraphics[keepaspectratio=true,scale=.65]{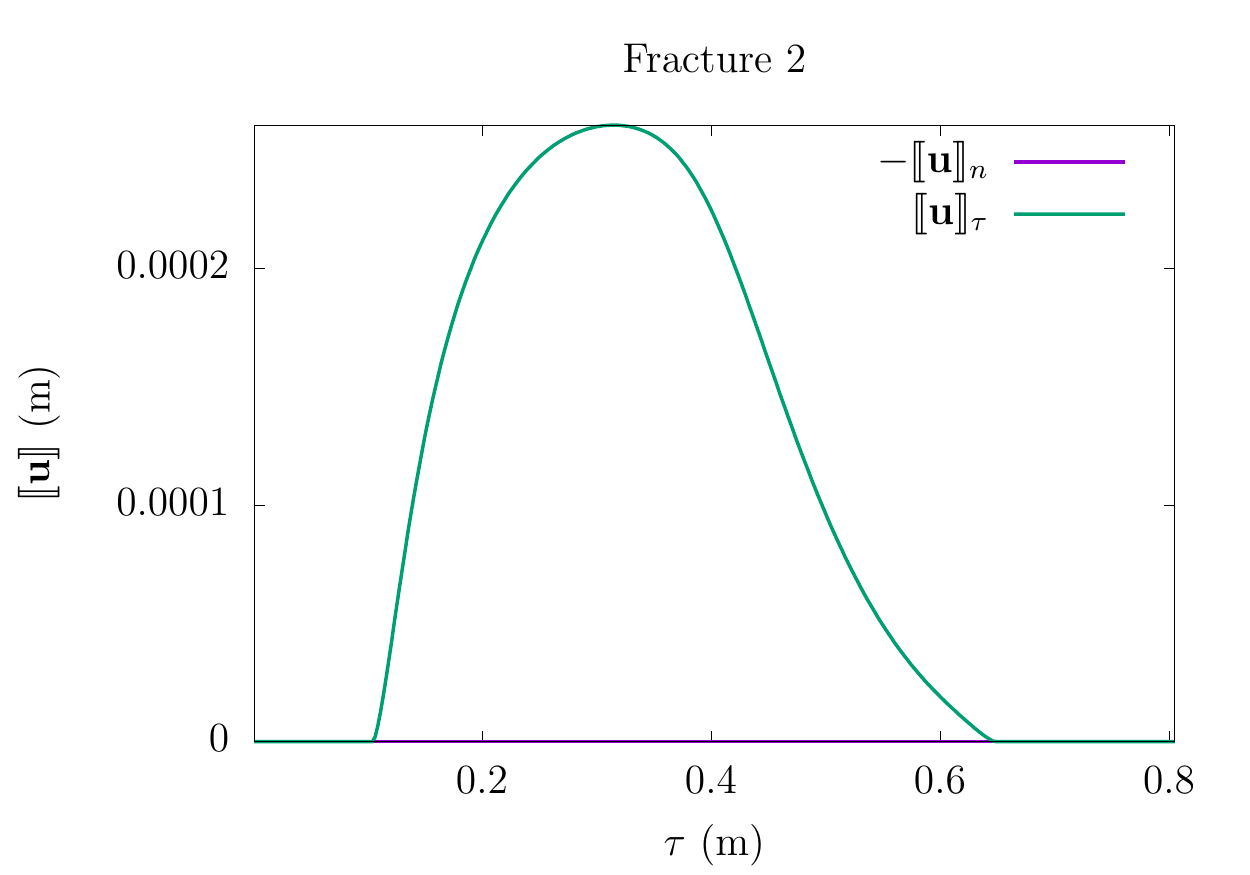}} \\
\subfloat[]{
\includegraphics[keepaspectratio=true,scale=.65]{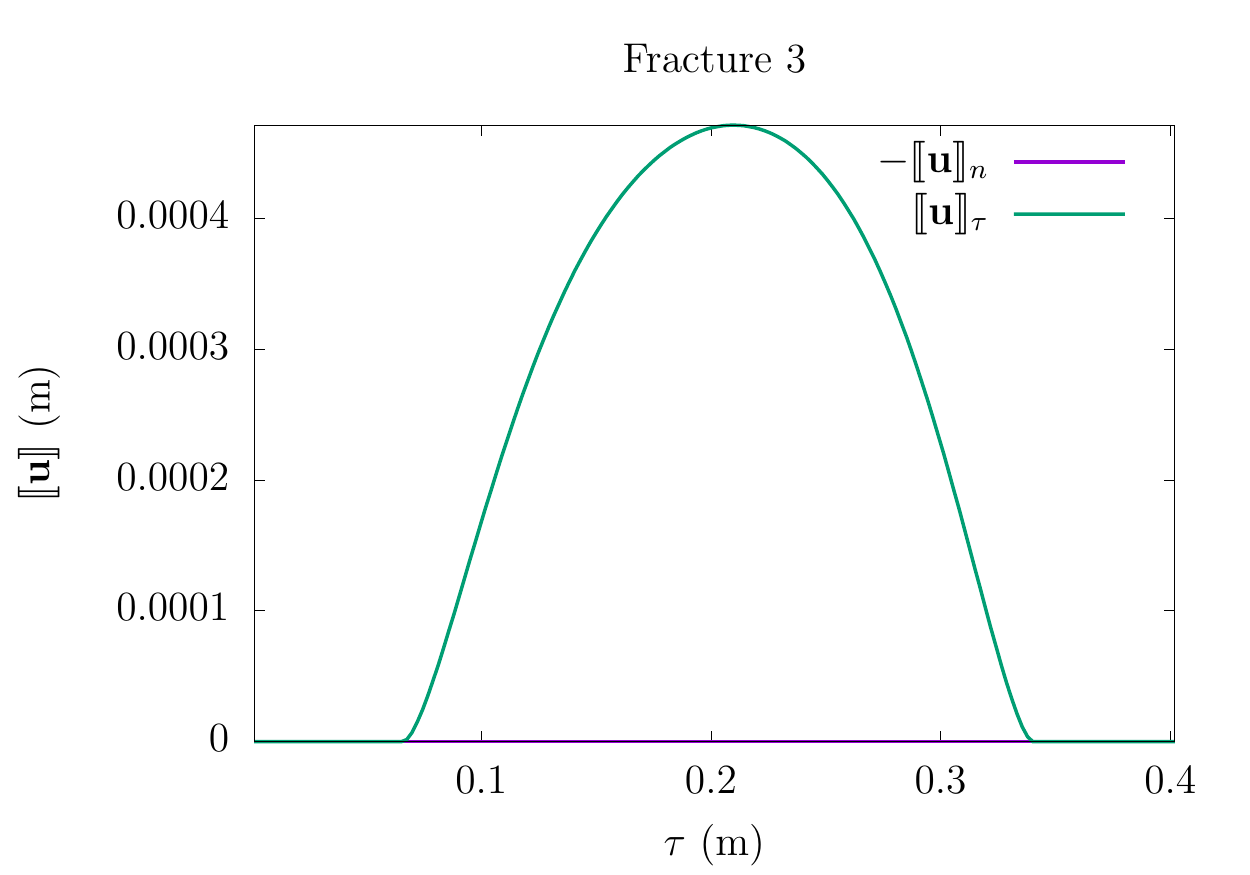}} 
\subfloat[]{
\includegraphics[keepaspectratio=true,scale=.65]{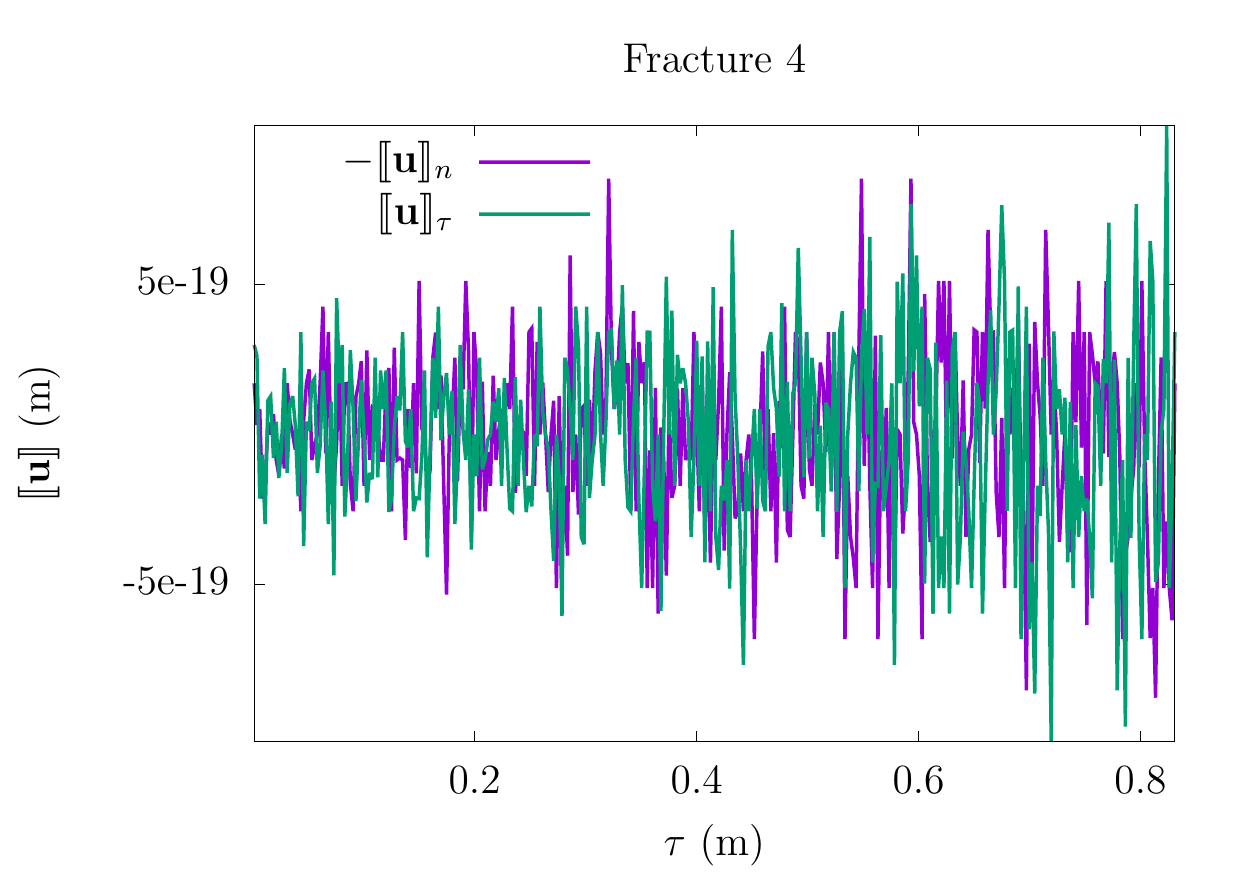}} \\

\subfloat[]{
\includegraphics[keepaspectratio=true,scale=.65]{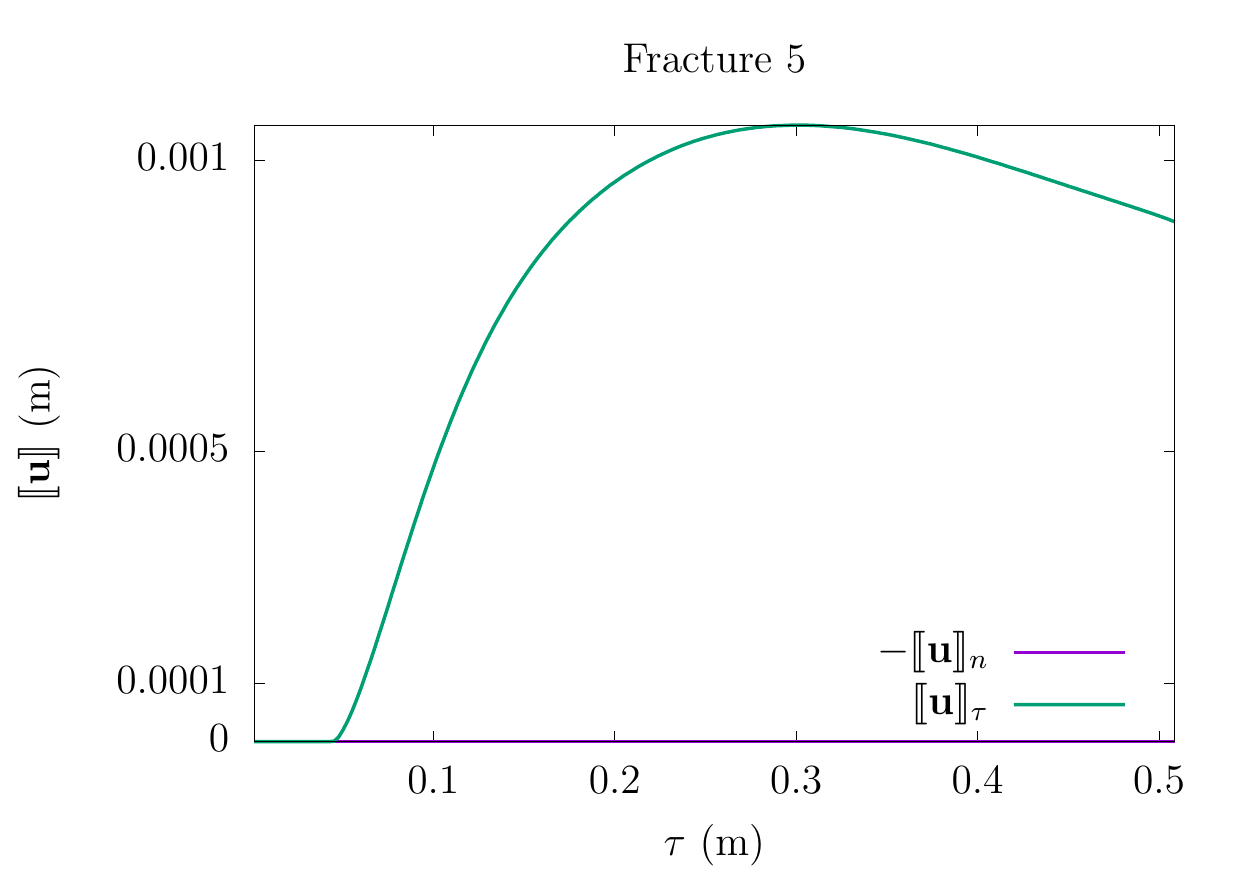}}
\subfloat[]{
\includegraphics[keepaspectratio=true,scale=.65]{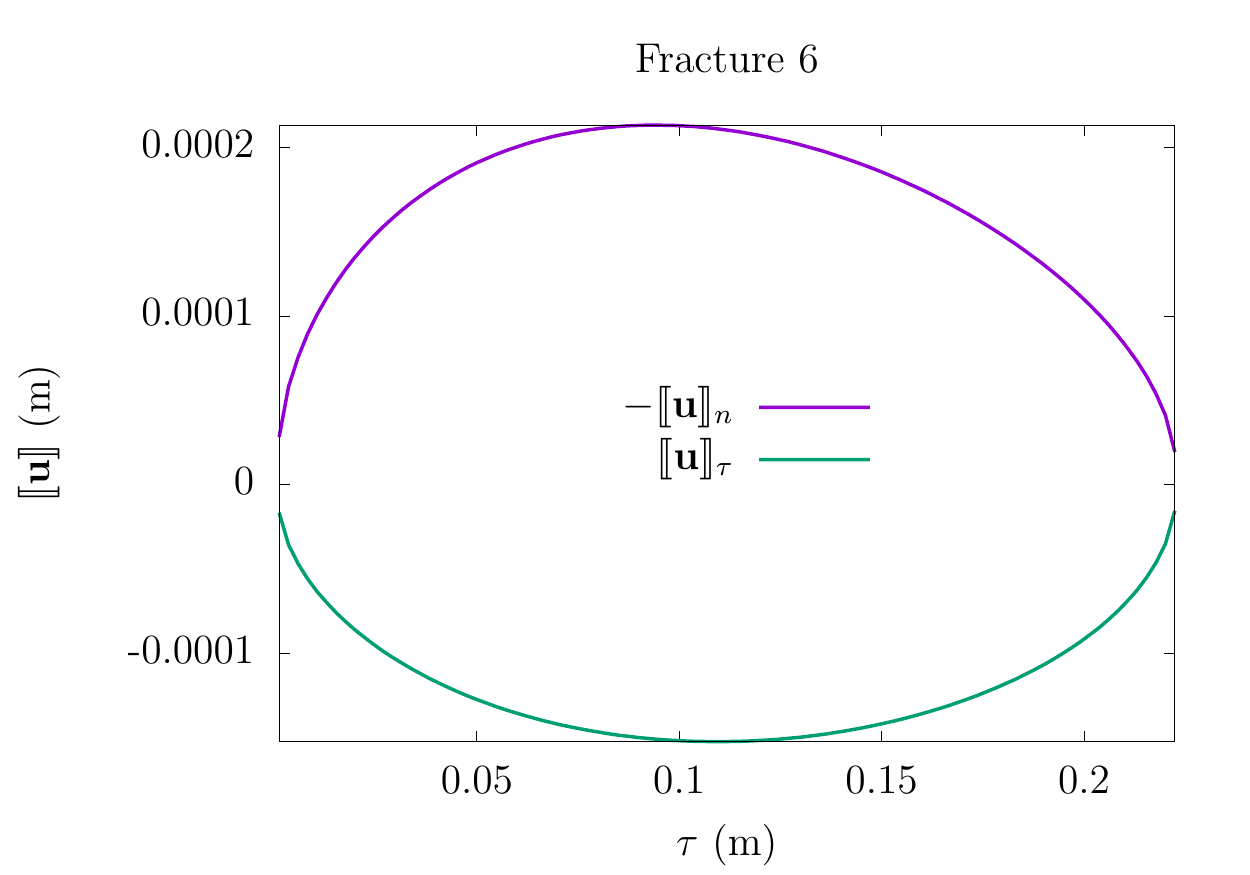}}

\caption{Normal and tangential jumps of the displacement field vs.~distance from the tips (labeled as $\tau$) for each of the six fractures, for the example of Section~\ref{bergen_meca}.}
\label{bergen_u}
\end{figure}

\begin{figure}
\captionsetup[subfigure]{labelformat=empty}
\centering

\subfloat[]{
\includegraphics[keepaspectratio=true,scale=.65]{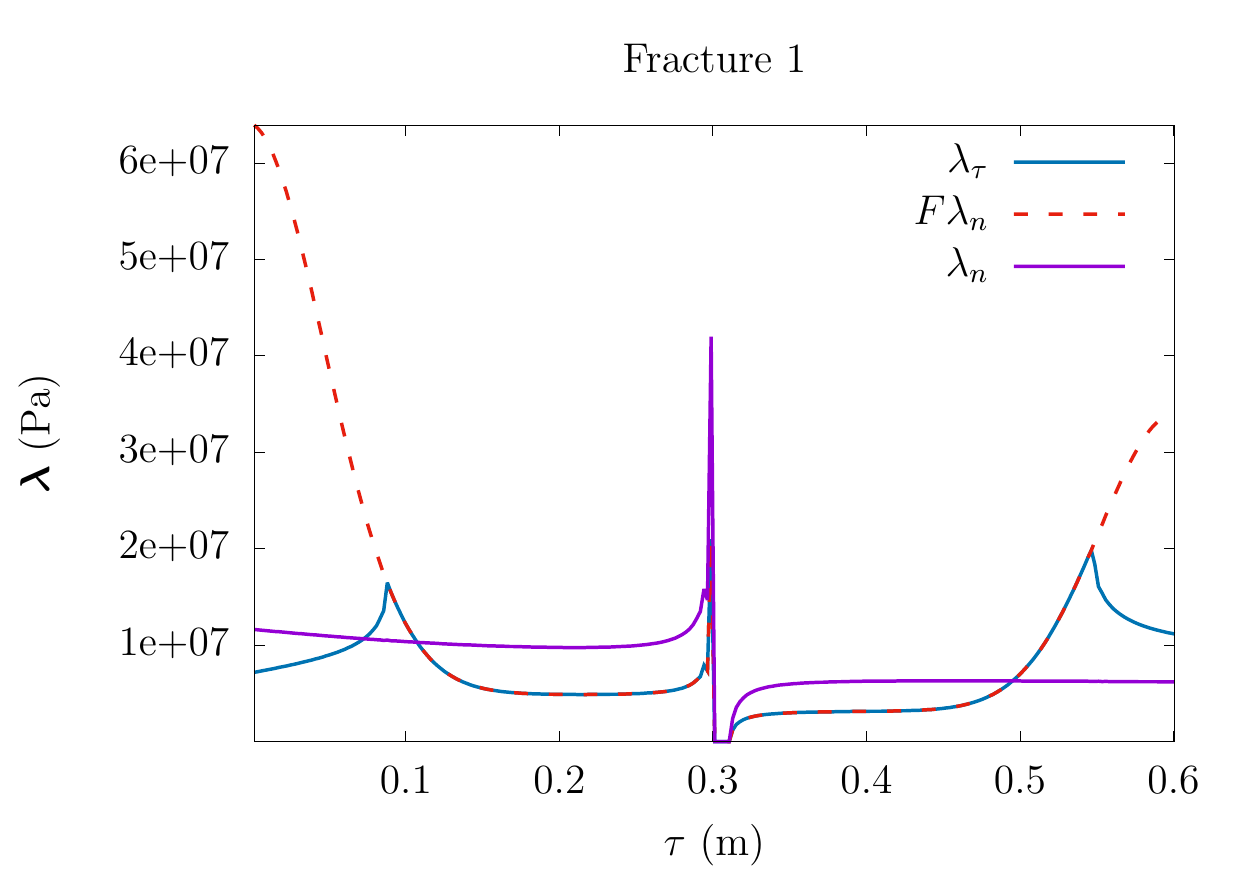}}
\subfloat[]{
\includegraphics[keepaspectratio=true,scale=.65]{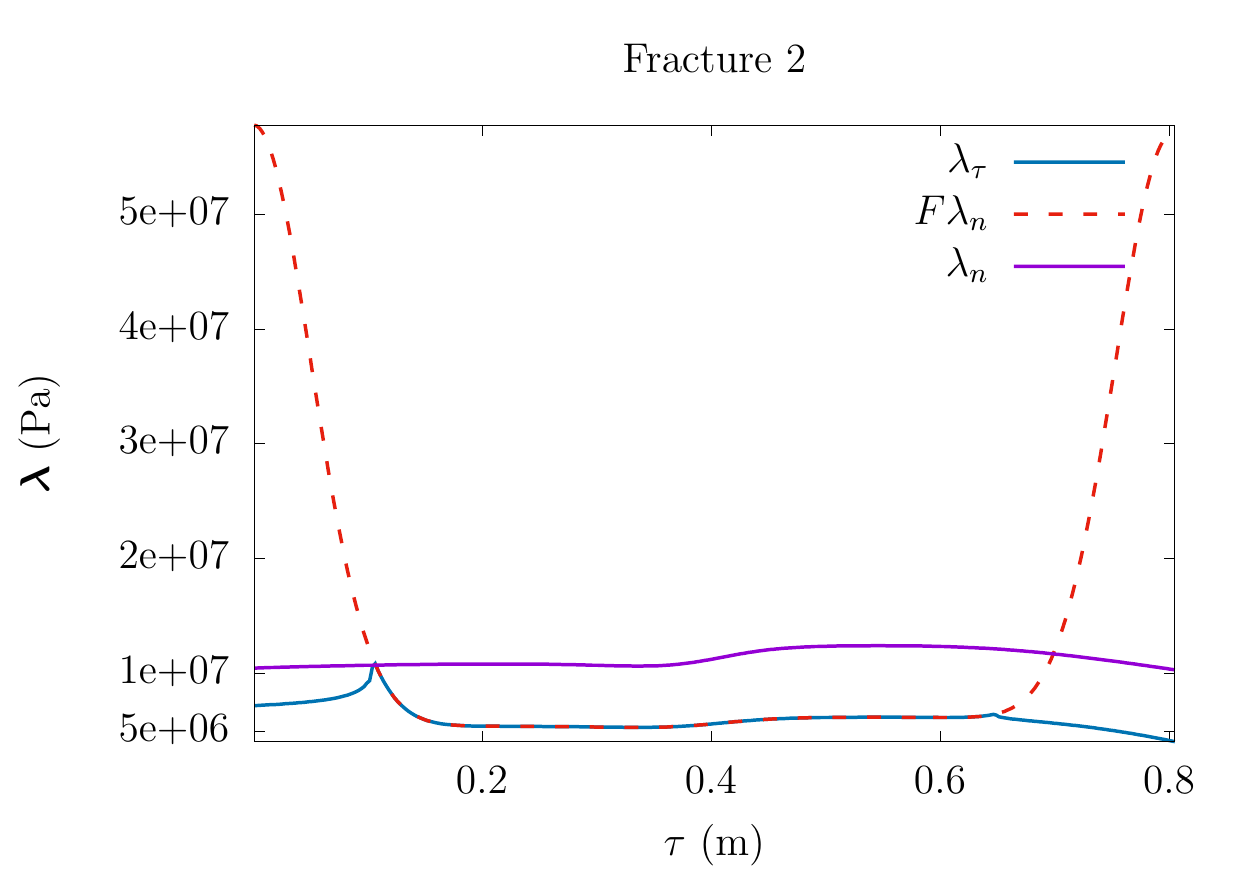}} \\
\subfloat[]{
\includegraphics[keepaspectratio=true,scale=.65]{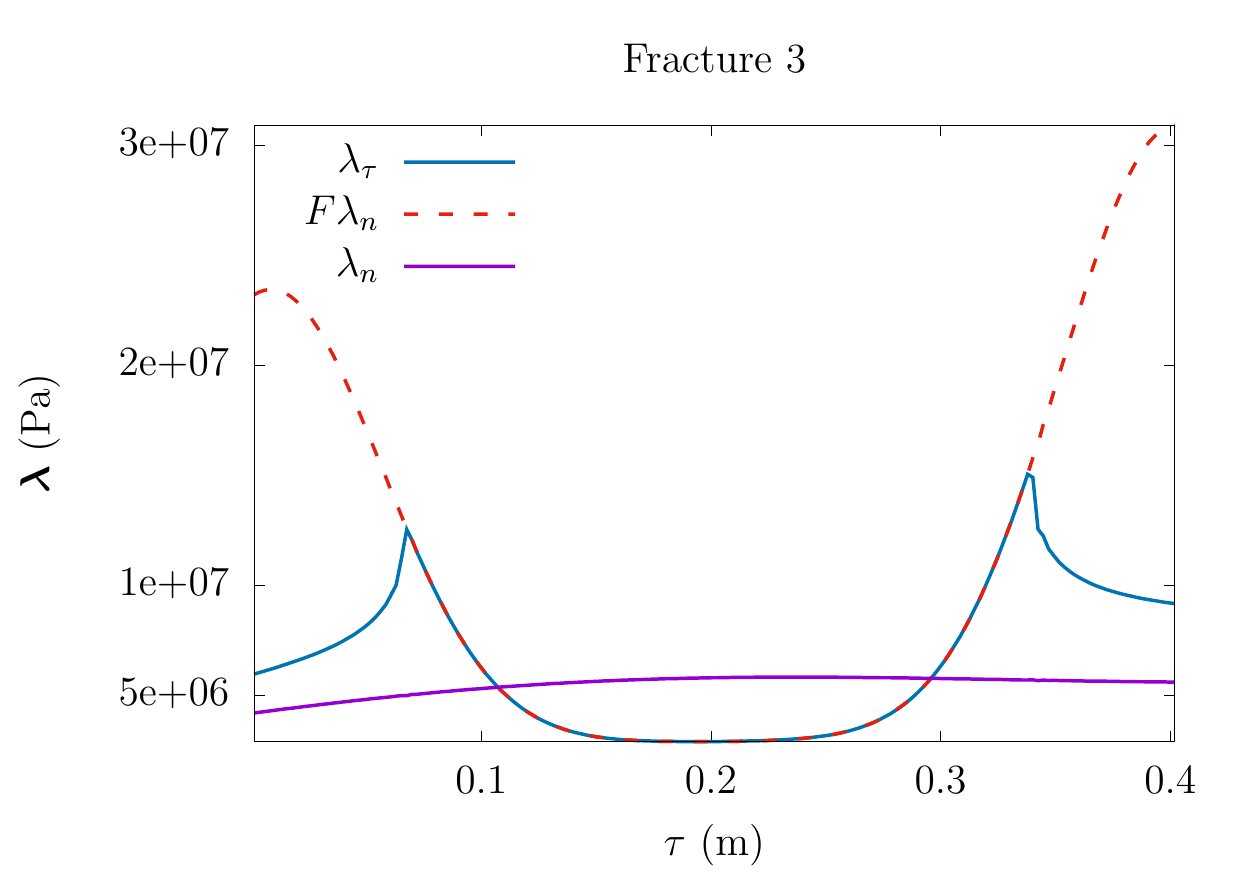}} 
\subfloat[]{
\includegraphics[keepaspectratio=true,scale=.65]{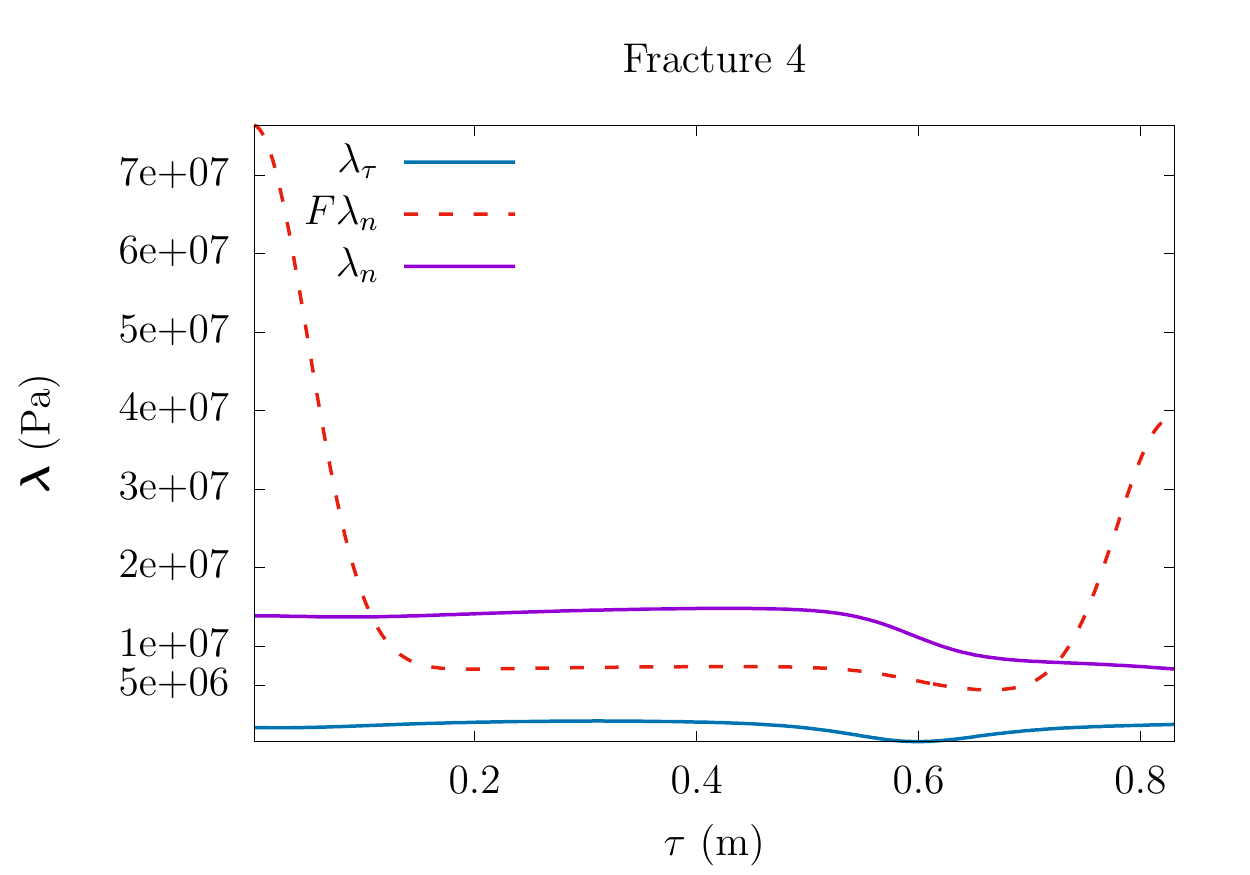}} \\

\subfloat[]{
\includegraphics[keepaspectratio=true,scale=.65]{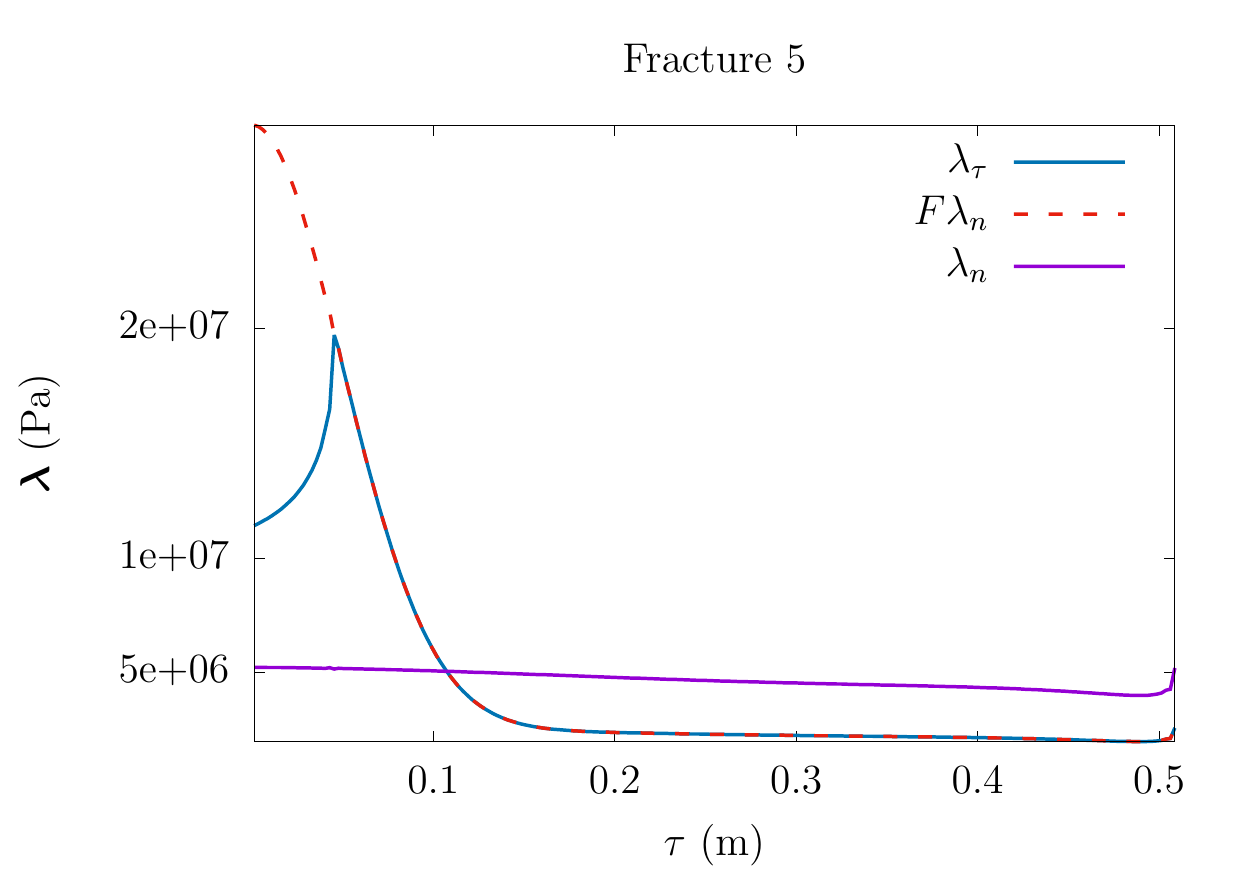}}
\subfloat[]{
\includegraphics[keepaspectratio=true,scale=.65]{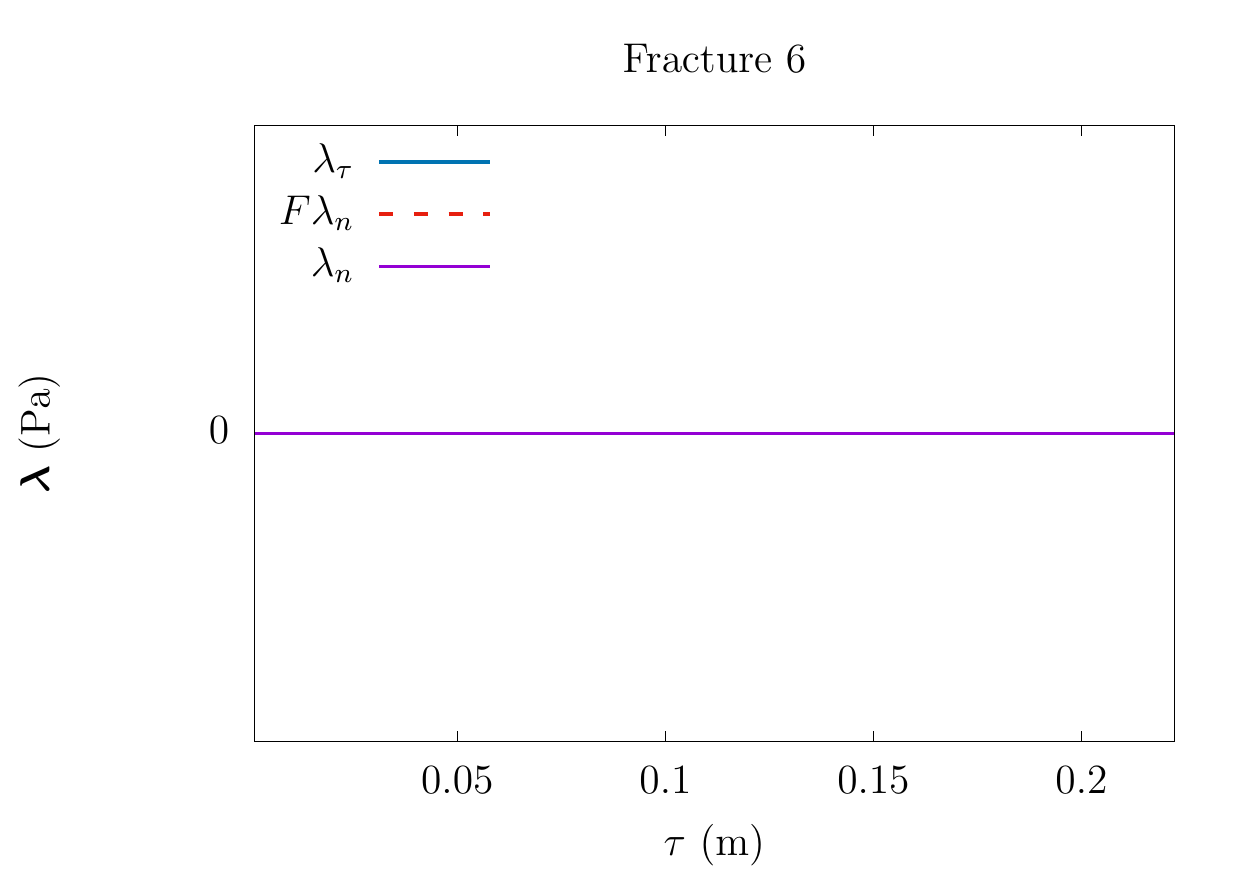}}

\caption{Lagrange multipliers vs.~distance from the tips (labeled as $\tau$) for each of the six fractures, for the example of Section~\ref{bergen_meca}.}
\label{bergen_lambda}
\end{figure}

\begin{figure}
\subfloat[]{
\includegraphics[keepaspectratio=true,scale=.665]{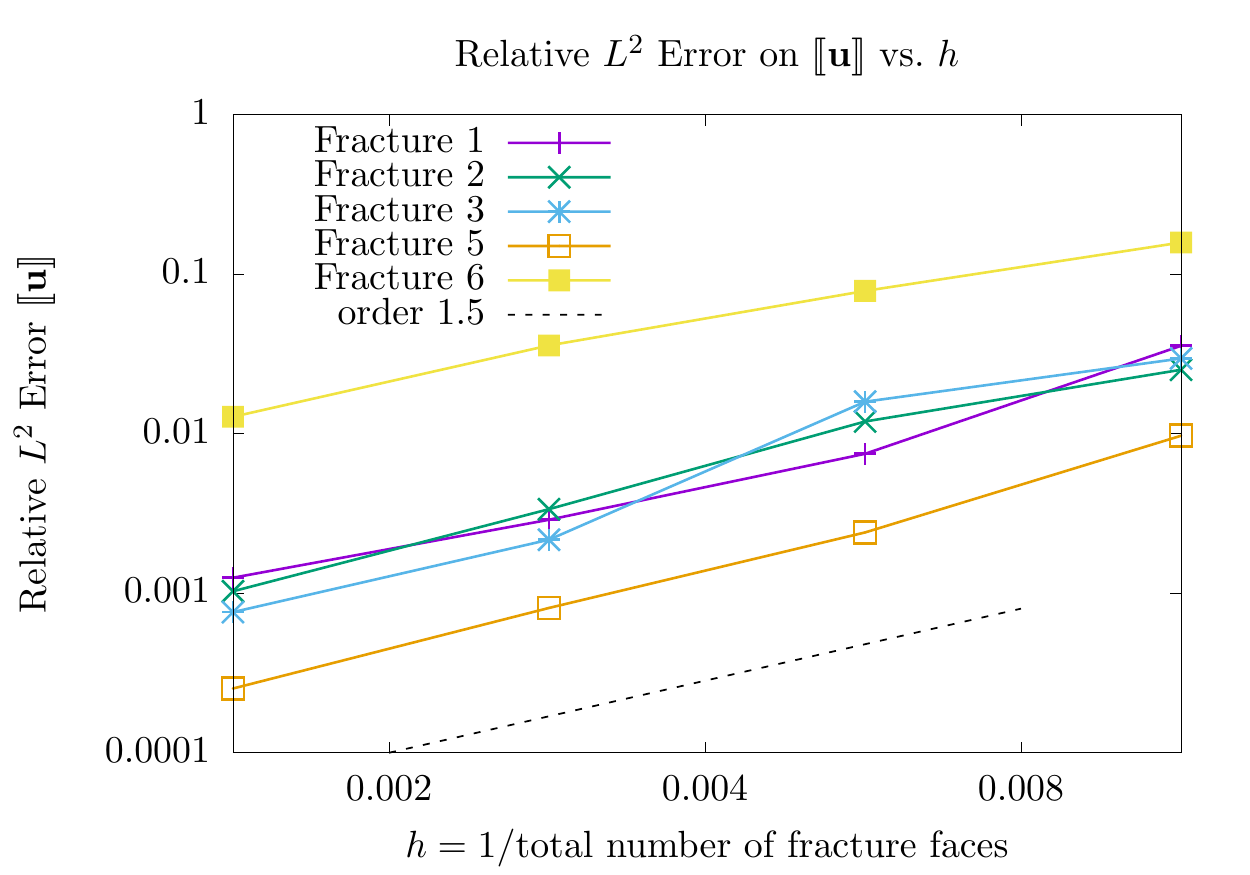}}
\subfloat[]{
\includegraphics[keepaspectratio=true,scale=.665]{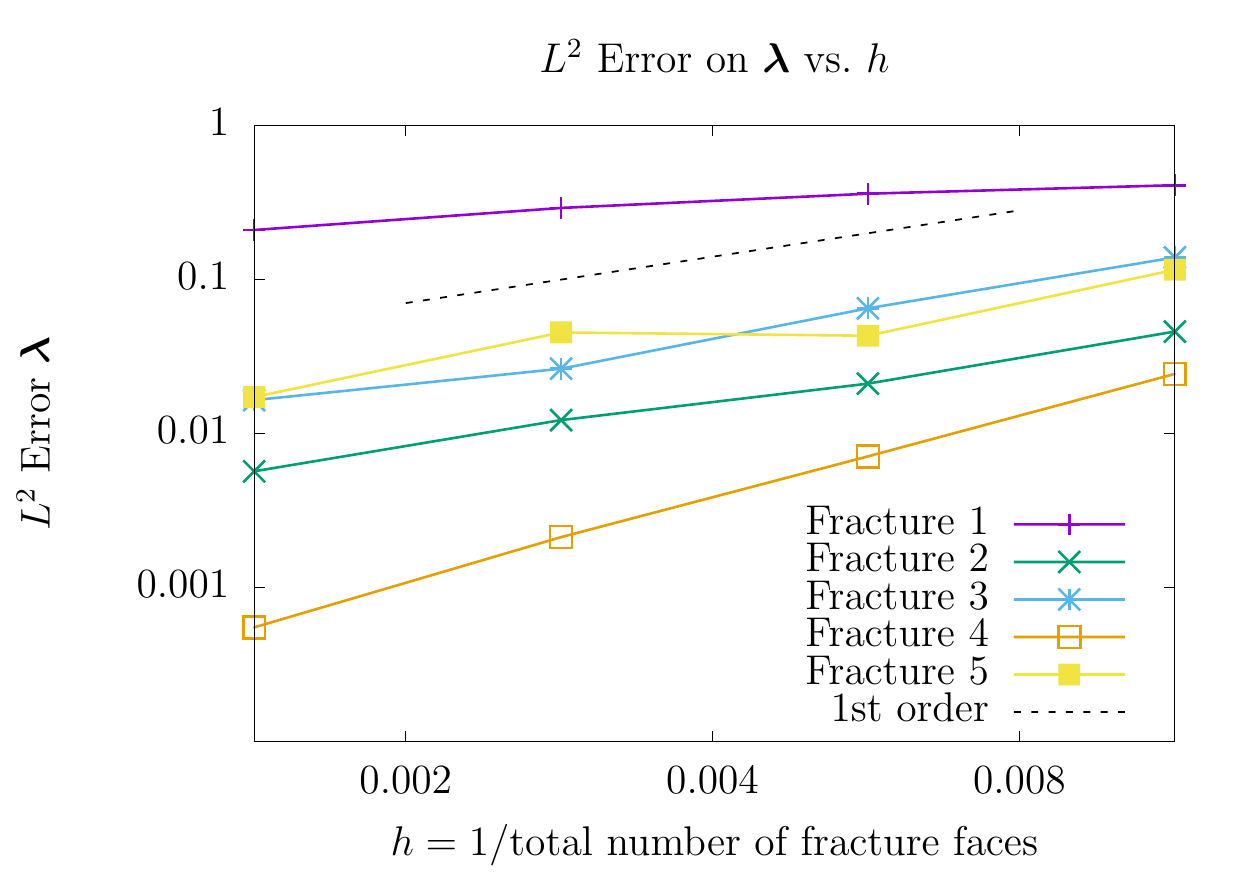}}
\caption{Relative $L^2$ error between the current and reference solutions in terms of $\jump{\bu}$ (a) and $\bm\lambda$ (b), yielding a 1.5-order and a 1st-order convergence, respectively.}
\label{bergen_errors}
\end{figure}

\subsection{Contact mechanics and two-phase Darcy flow: drying model in radioactive waste geological storage}
\label{cas_andra}

This test case mostly presents the same geometry, boundary conditions, and data set as in~\cite{GDM-poromeca-disc}. Unlike here, however, the model and simulations in this reference assume that all fractures remain open during the process. 

We consider a hollow cylinder (Figure~\ref{geometry_andra}) made of a low-permeability porous medium, with an axisymmetric oblique fracture network, subject to axisymmetric loads -- uniform pressures exerted on the internal and external surfaces. Using cylindrical coordinates $(x,r,\theta)$, the problem is reduced to a two-dimensional formulation on the radial section, and the displacement field only consists of its axial and radial components:
$$
\begin{gathered}
\bu(t;x,r,\theta) = u_x(t;x,r) \mathbf e_x + u_r(t;x,r) \mathbf e_r(\theta),\quad
\mathbf e_r(\theta) = (\cos\theta) \mathbf e_y + (\sin\theta) \mathbf e_z,\\
0\le x \le L, \quad R_{\mathrm{int}} \le r \le R_{\mathrm{ext}},\quad 0 \le \theta \le 2\pi,
\end{gathered}
$$
where $t\in[0,T]$
and the vectors of the system of cylindrical coordinates (see Figure~\ref{geometry_andra}) are the axial unit vector $\mathbf e_x$, the radial unit vector $\mathbf e_r = \mathbf e_r(\theta)$, and the orthoradial unit vector $\mathbf e_\theta$. The final time is set to $T=200$ years. The geometry is characterized by the following data set: length $L=10\,\text{m}$, internal and external radii $R_\mathrm{int} = 5\,\mathrm{m}$ and $R_\mathrm{ext} = 35\,\mathrm{m}$; two consecutive fractures are spaced by 1.25\,m. The matrix is characterized by the Young modulus $E = 4.857$~GPa and the Poisson ratio $\nu = 0.21$, Biot's coefficient and modulus $b=1$ and $M=1\,\text{GPa}$ respectively, and by an initial porosity $\phi_m^0 = 0.15$. The matrix permeability tensor is assumed isotropic, i.e.~$\bb K_m = K_m \bb I$, and the permeability $K_m$ is expressed in terms of the current porosity $\phi_m$ by the Kozeny--Carman law:
$$K_m(\phi_m) = K_m^0 \frac{(1-\phi_m^0)^2}{(\phi_m^0)^3} \frac{\phi_m^3}{(1-\phi_m)^2},$$
discretized explicitly in time and with $K_m^0 = 5\cdot10^{-20}\rm\,m^2$. We note that the equation \eqref{GD_hydro} of the gradient scheme can trivially be adapted to account for this dependency of the permeability with respect to the porosity (replace $\bb K_m$ with $\bb K_m(\phi_\D)$), and that under the condition $\phi_{\rm min}\le \phi_\D\le \phi_{\rm max}$ with $0\le\phi_{\rm min}\le\phi_{\rm max}<1$, the estimates \eqref{apriori.est} remain valid.
The normal transmissibility of fractures is  $\Lambda_f = 10^{-9}$ m. The friction coefficient is assumed constant over the whole fracture network and equal to $F = 0.5$. The initial fracture aperture is taken here equal to $1$~mm, instead of $1$~cm as in~\cite{GDM-poromeca-disc} (therein, the larger initial aperture was required to ensure that $d_{f,\D_\bu}$ remain strictly positive throughout the simulation). The initial guess for the non-smooth Newton algorithm for the mechanics is given by the previous time step solution and set to $\bu = \mathbf 0$, $\bm\lambda = \mathbf 0$ (open fractures) at the first time step. Finally, the parameter $c$ appearing in the friction bound~\eqref{friction_bound} is set to $10^6$~N/m$^3$ for the active set algorithm and to $10^9$~N/m$^3$ for the non-smooth Newton algorithm. 
The matrix relative permeabilities of 
the liquid and gas phases are given 
by the Van Genuchten laws:
\begin{eqnarray*}
k_{r,m}^\l(s^\l) = 
\left\{\!\!\!\!\begin{array}{r@{\,\,}c@{\,\,}ll}
&0 &\mbox{if}& s^\l < S_{lr},\\
&1 &\mbox{if}& s^\l > 1-S_{gr},\\
&\sqrt{\bar s^\l} \(1- (1-(\bar s^\l)^{1/q})^q\)^2 &\mbox{if}&  S_{lr}\leq s^\l \leq 1-S_{gr},  
\end{array}\right.
\end{eqnarray*}
\begin{eqnarray*}
k_{r,m}^\g(s^\g) = 
\left\{\!\!\!\!\begin{array}{r@{\,\,}c@{\,\,}ll}
&0 &\mbox{if}& s^\g < S_{gr},\\
&1 &\mbox{if}& s^\g > 1-S_{lr},\\
&\sqrt{1-\bar s^\l} \(1- (\bar s^\l)^{1/q}\)^{2q} &\mbox{if}&  S_{gr}\leq s^\g \leq 1-S_{lr}, 
\end{array}\right.
\end{eqnarray*}
with $$
\bar s^\l = {s^\l-S_{lr} \over 1-S_{lr}-S_{gr} },
$$
and the parameter $q=0.328$, the residual liquid and gas saturations 
$S_{lr}=0.35$ and $S_{gr}=0$; in the fractures, we take $k_{r,f}^\alpha(s) = s$ for both phases. The phase mobilities are then $\eta_m^\a(s^\a) = k_{r,m}^\a(s^\a)/\mu^\a$ and $\eta_f^\a(s^\a) = k_{r,f}^\a(s^\a)/\mu^\a$, $\a\in\{\l,\g\}$ both in the matrix and in the fractures, with the wetting and non-wetting dynamic viscosities $\mu^\l = 10^{-3}\,\rm Pa{\cdot}s$ and $\mu^\g = 1.851{\cdot}10^{-5}\,\rm Pa{\cdot}s$. These functions are not bounded below by a strictly positive number, but this does not affect the numerical results (modifications ensuring such a bound can be implemented, and lead to nearly imperceptible changes in the numerical outputs~\cite{GDM-poromeca-cont}).
The saturation--capillary pressure relation is given by Corey's law:
$$
s_{\omega}^\g = S_{\omega}^\g (p_c) = \max\left(1 - \exp\left(-\frac{p_c}{R_{\omega}}\right),0\right),
\quad \omega\in \{m,f\},
$$
with $R_m = 2{\cdot}10^8$\,Pa and $R_f = 10^2$\,Pa. No damaged layer is included in the model, setting $d_\aa = 0$, 
$\eta^\a_\aa = \eta^\a_f$, and $S^\a_\aa = S^\a_f$.  
Moreover, the medium is supposed to have a pre-stress state described by the following tensor:
\begin{equation*}
\bbsig^0 = \sigma_x^0 \, \mathbf e_x \otimes \mathbf e_x + \sigma^0_r \, {\mathbf e}_r \otimes \mathbf e_r +
\sigma^0_\theta \, \mathbf e_\theta \otimes \mathbf e_\theta,
\qquad \sigma_x^0 = 16\,\text{MPa}, \ \ \sigma_r^0 = \sigma_\theta^0 = 12\,\text{MPa},
\end{equation*}
taken into account as an additional term in the sum of the purely elastic and fluid matrix equivalent pressure contributions.

Initially, the system is assumed to be fully saturated with the liquid phase, both in the matrix and in the fracture network, with uniform pressures $p^{0,\l}_m=p^{0,\g}_m = 4\,\mathrm{MPa}$ in the matrix, and $p^{0,\l}_f=p^{0,\g}_f = 10^5\,\mathrm{Pa}$ in the fracture network.

Concerning flow boundary conditions, the porous medium is assumed impervious (vanishing fluxes) on the lateral boundaries corresponding to $x=0$ and $x=L$.
On the inner surface $r=R_{\rm int}$, a given gas saturation is imposed: $s^\g_m = 0.35$ on the matrix side and $s^\g_f = 1 - 10^{-8}$ at fracture nodes, and atmospheric pressure $p^\g_{f} = 10^5\,\text{Pa}$ everywhere. On the outer surface $r=R_{\rm ext}$, a liquid saturation $s_m^\l = 1$ and pressure $p_m^\l = 4\,\text{MPa}$ are imposed.

As for boundary conditions on the mechanical part of the model, a vanishing axial displacement $u_x$ and a vanishing tangential stress are imposed on the lateral boundaries. 
On the other hand, external surface loads $\mathbf g$ (uniform pressures) are applied on the inner and outer surfaces:
$$\mathbf g =\begin{cases}\begin{alignedat}{2} -\sigma_N^T \mathbf n, &\ \ \sigma_N^T > 0,&\ \ &\text{if }r=R_{\rm ext},\\
-p_{\rm atm}\mathbf n,  &\ \ p_{\rm atm} > 0,& \ \ &\text{if }r=R_{\rm int},\end{alignedat}\end{cases}$$
where $\mathbf n = \mathbf e_r$ for $r=R_{\rm ext}$ and $\mathbf n = -\mathbf e_r$ for $r=R_{\rm int}$.
We consider $\sigma_N^T = 10.95\,\text{MPa}$ as the numerical value for the uniform pressure on the outer surface.

Figures~\ref{cf_andra} and \ref{df_andra} show respectively the contact state and the fracture aperture of the fracture network on each fracture face at three different times. Under the effect of the internal and external pressures $p_{\rm atm}$ and $\sigma_N^T$, the fractures are all in contact at initial time with a slipping state for the chevron-like fractures and a sticking state for the horizontal one. The chevron-like fractures start opening up at later times, as gas starts filling the matrix with strong capillary pressure. At the final time, most fractures are open, except half of the horizontal one which is still sticking. In Figure~\ref{df_vs_t_andra}, we plot the time history of the mean fracture width, and compare the result obtained using the active set method and the regularized non-smooth Newton method for the mechanics. The difference between the two curves is undetectable as expected. In Figure~\ref{sm_pm_phim}, it can be seen that strong capillary forces induce the drying of the matrix in the neighborhood of the inner surface, along with a highly negative liquid pressure. This negative liquid pressure also triggers the contraction of the pores and the spreading of the fracture sides.

Finally, Figure \ref{perfs_andra} presents a comparison of the respective performances of the Newton--Raphson method for the flow and mechanics, obtained using the active set and the regularized non-smooth Newton methods for the mechanics. It can be seen that the active set method provides a slightly better convergence than the non-smooth Newton algorithm. In the same figure, we also compare the performance of the active set method combined with Newton--Krylov accelerations on the fixed points $p^E = {\bf g}_p(p^E)$ or $\bu = {\bf g}_\bu(\bu)$; as already discussed in the beginning of Section~\ref{num.experiments}, the Newton-Krylov acceleration of the fixed point $p^E = {\bf g}_p(p^E)$  is fairly more efficient than the one of the fixed point $\bu = {\bf g}_\bu(\bu)$.

\begin{figure}
\captionsetup[subfigure]{labelformat=empty}
\centering
\subfloat[]{
\includegraphics[keepaspectratio=true,scale=.55]{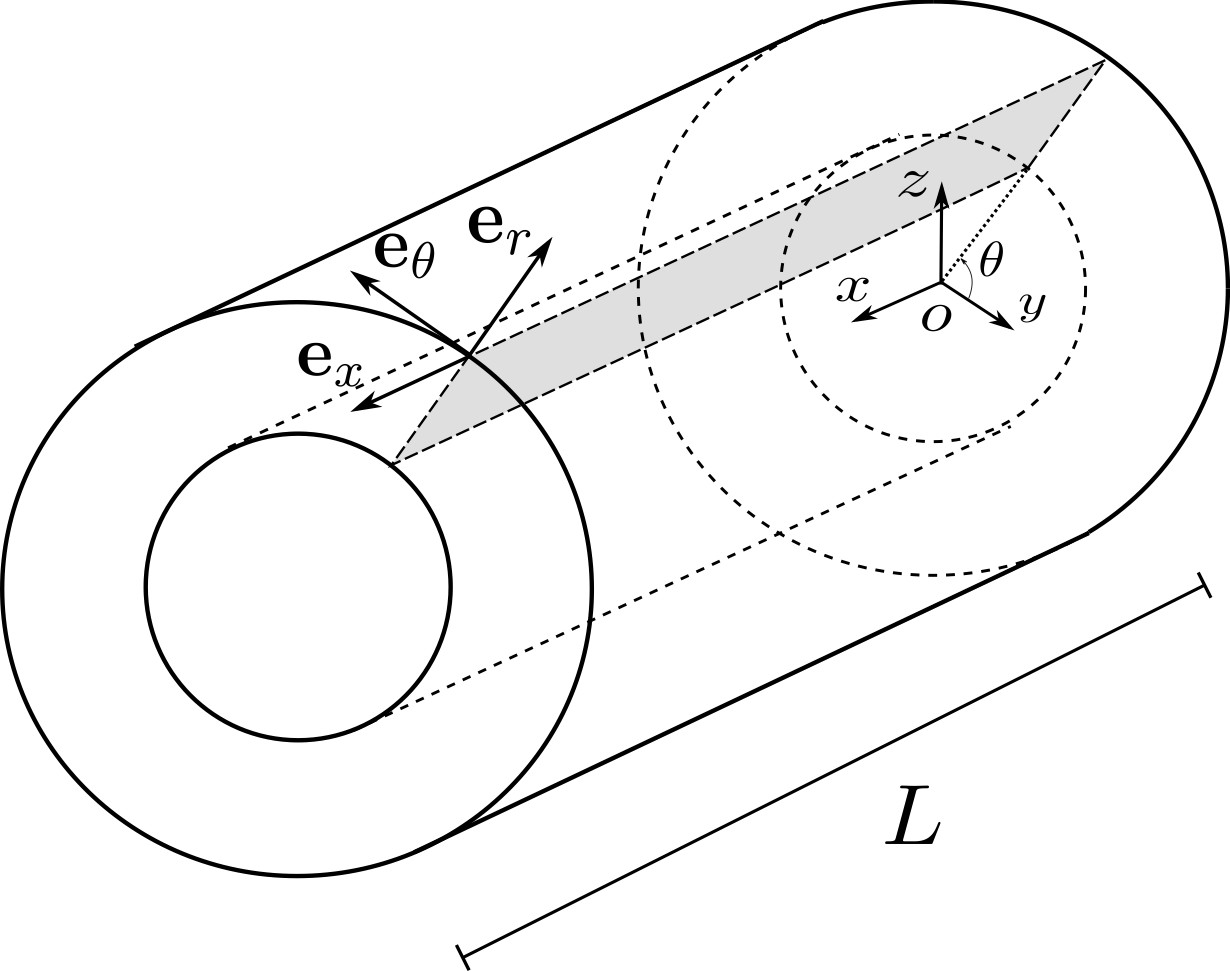}}
\hspace{1cm}
\subfloat[]{
\includegraphics[keepaspectratio=true,scale=.4]{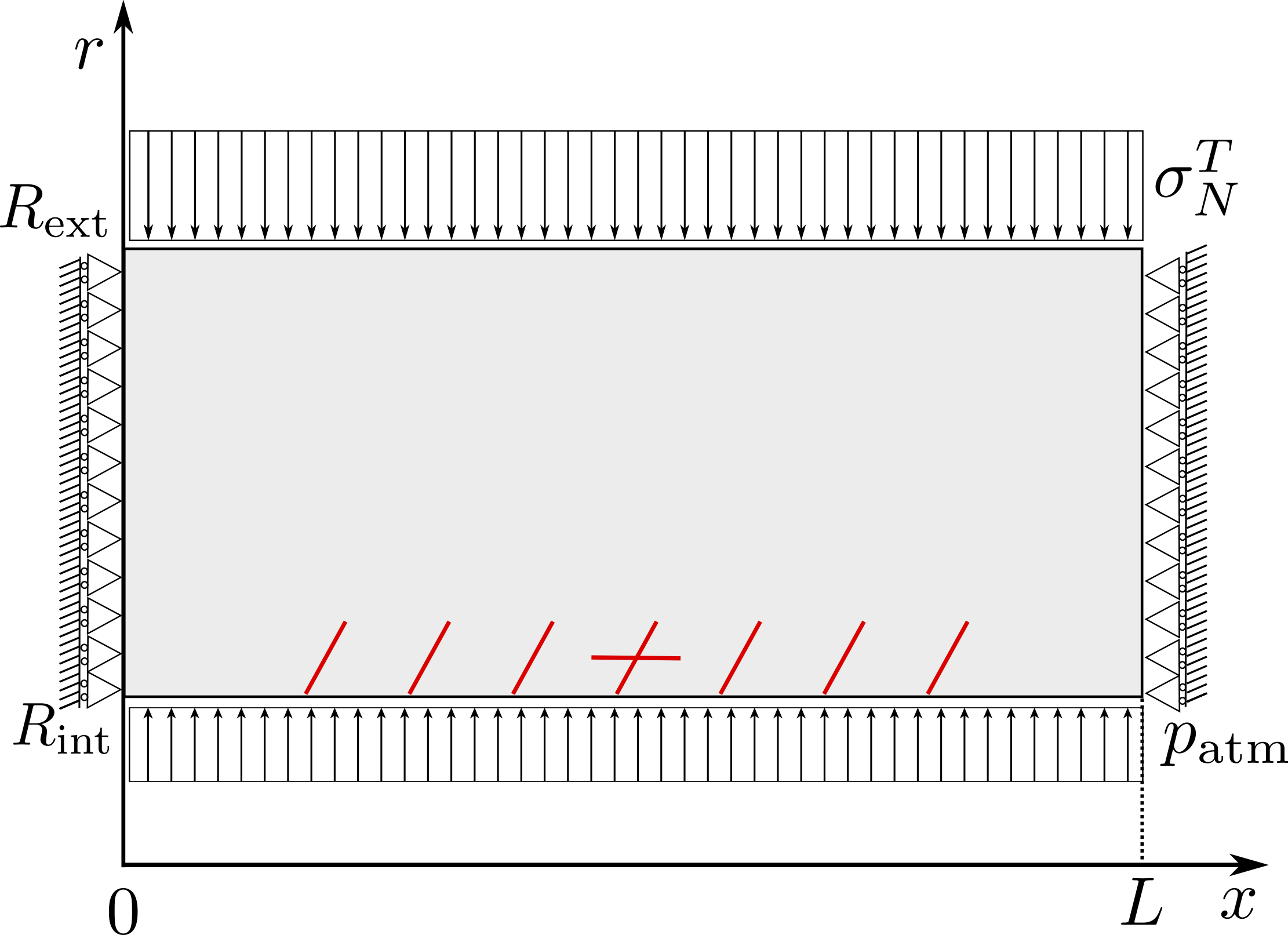}}
\caption{Hollow cylinder of length $L$ and internal and external radii $R_{\rm int}$ and $R_{\rm ext}$, respectively. The radial section (unscaled) is highlighted in gray and the fracture network is shown in red on the right figure, along with mechanical boundary conditions.}
\label{geometry_andra}
\end{figure}

\begin{figure}
\captionsetup[subfigure]{labelformat=empty}
\centering

\subfloat[]{
\includegraphics[keepaspectratio=true,scale=.15]{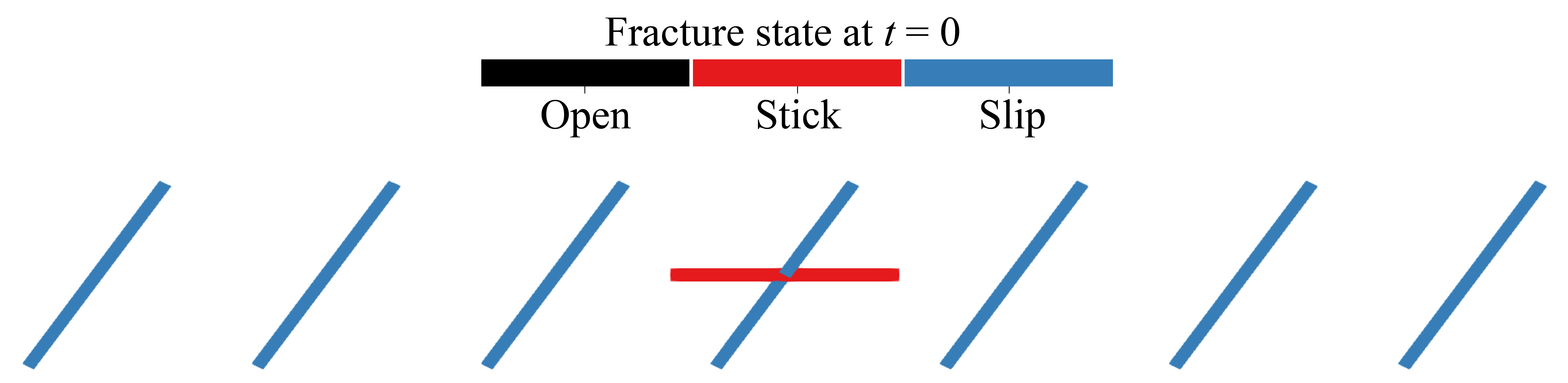}}\\
\subfloat[]{
\includegraphics[keepaspectratio=true,scale=.15]{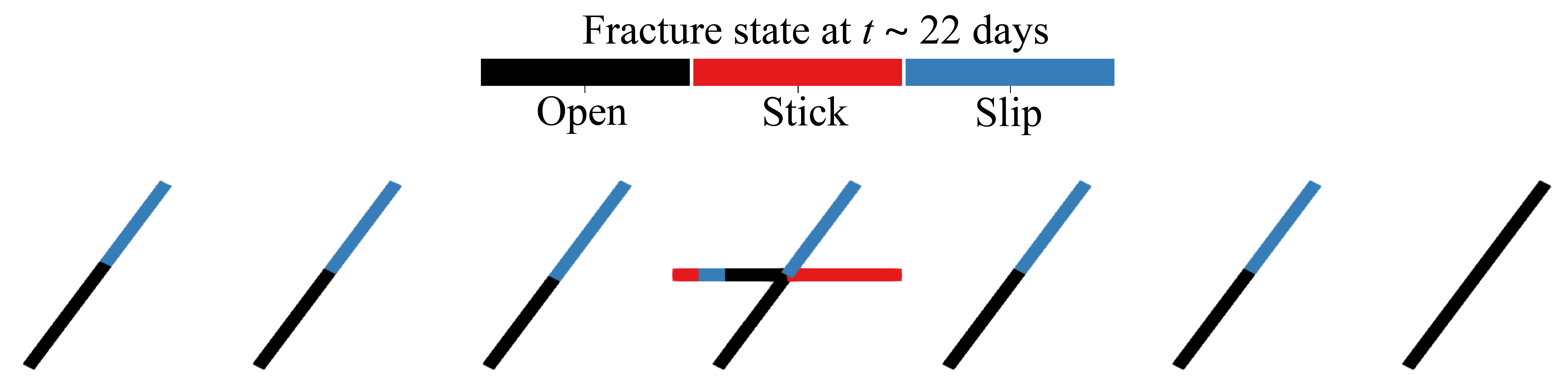}} \\
\subfloat[]{
\includegraphics[keepaspectratio=true,scale=.15]{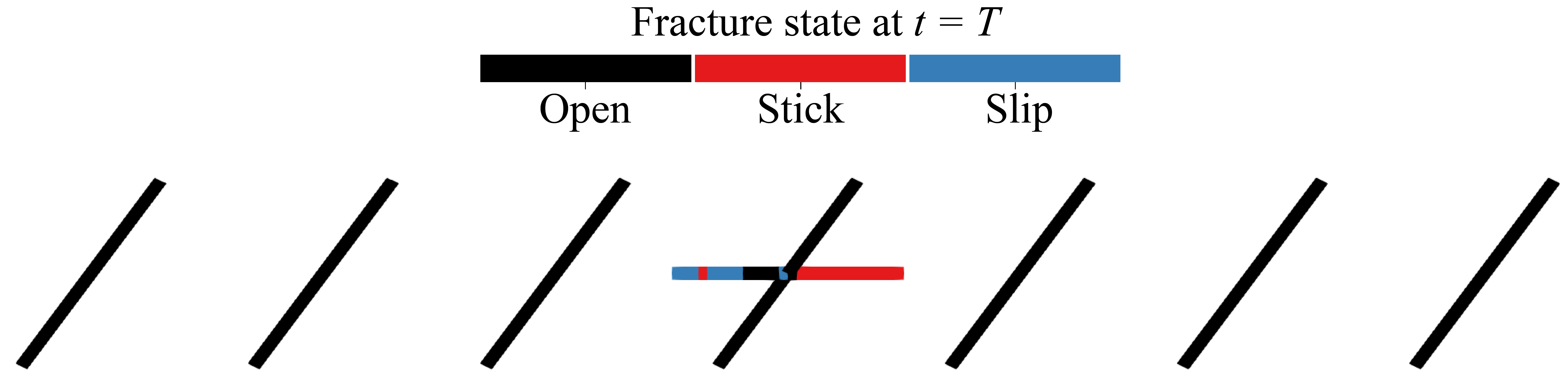}} 

\caption{Fracture state at different times, for the example of Section~\ref{cas_andra}.}
\label{cf_andra}
\end{figure}

\begin{figure}
\centering

\includegraphics[keepaspectratio=true,scale=.8]{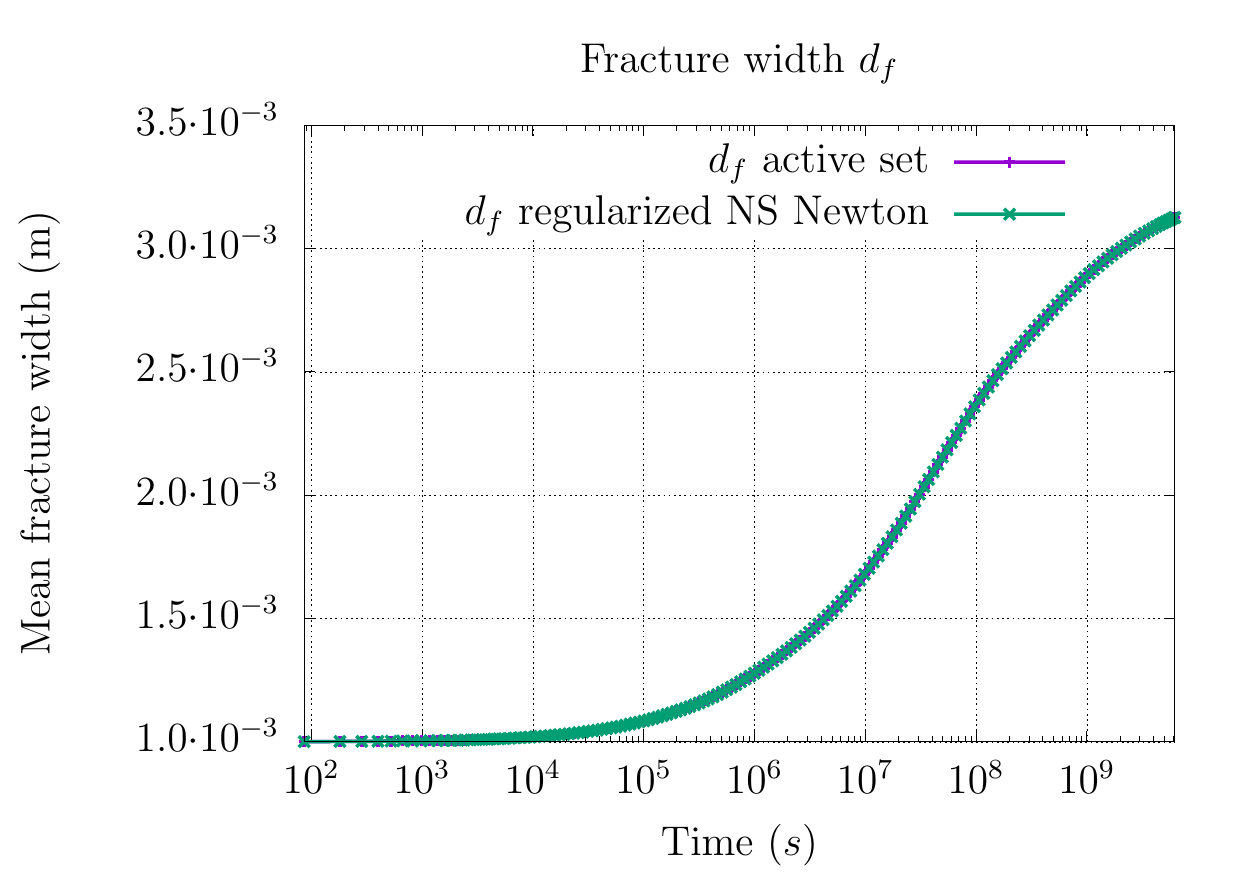}
\caption{Time histories of the mean fracture width given by the active set method and the regularized non-smooth Newton method for the mechanics, for the example of Section~\ref{cas_andra}.}
\label{df_vs_t_andra}
\end{figure}

\begin{figure}
\captionsetup[subfigure]{labelformat=empty}
\centering

\subfloat[]{
\includegraphics[keepaspectratio=true,scale=.15]{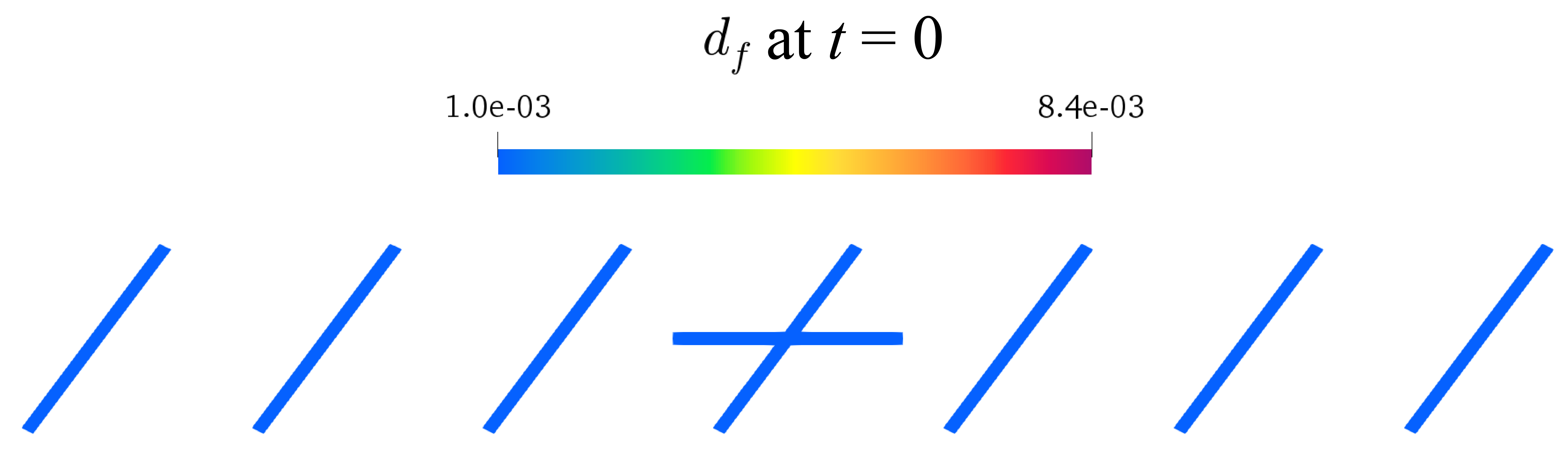}}\\
\subfloat[]{
\includegraphics[keepaspectratio=true,scale=.15]{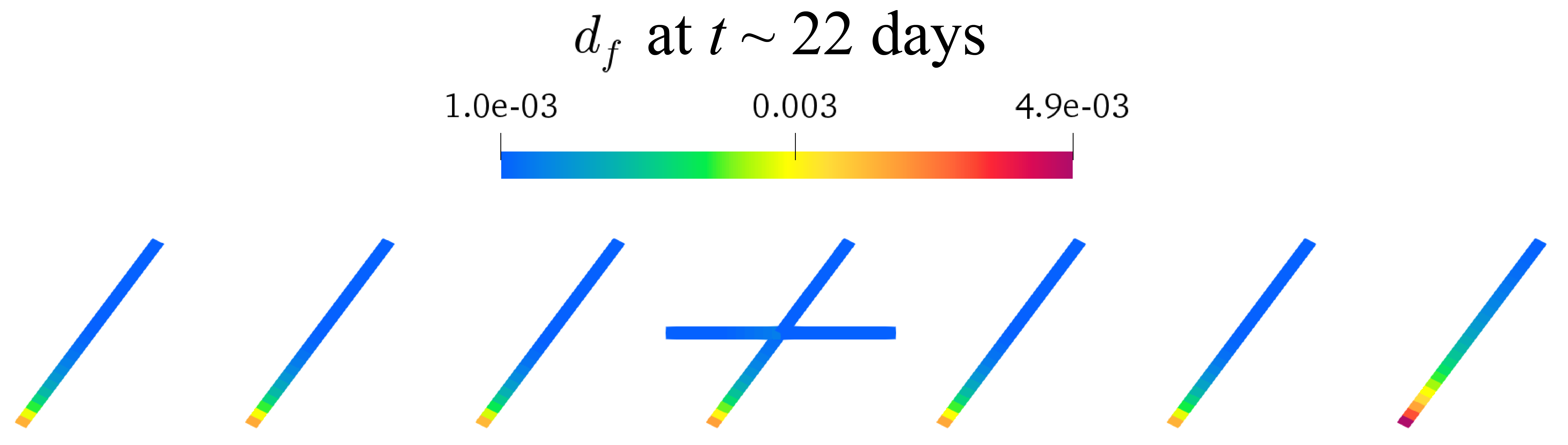}} \\
\subfloat[]{
\includegraphics[keepaspectratio=true,scale=.15]{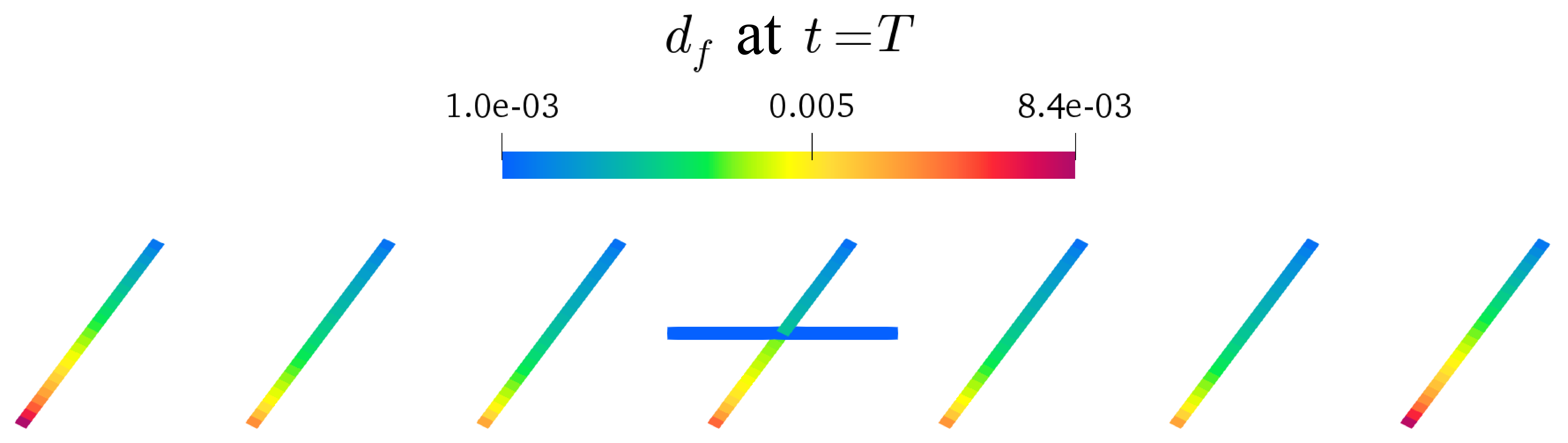}} 

\caption{Fracture aperture (m) at different times, for the example of Section~\ref{cas_andra}.}
\label{df_andra}
\end{figure}

\begin{figure}
\captionsetup[subfigure]{labelformat=empty}
\centering

\subfloat[]{
\includegraphics[keepaspectratio=true,scale=.175]{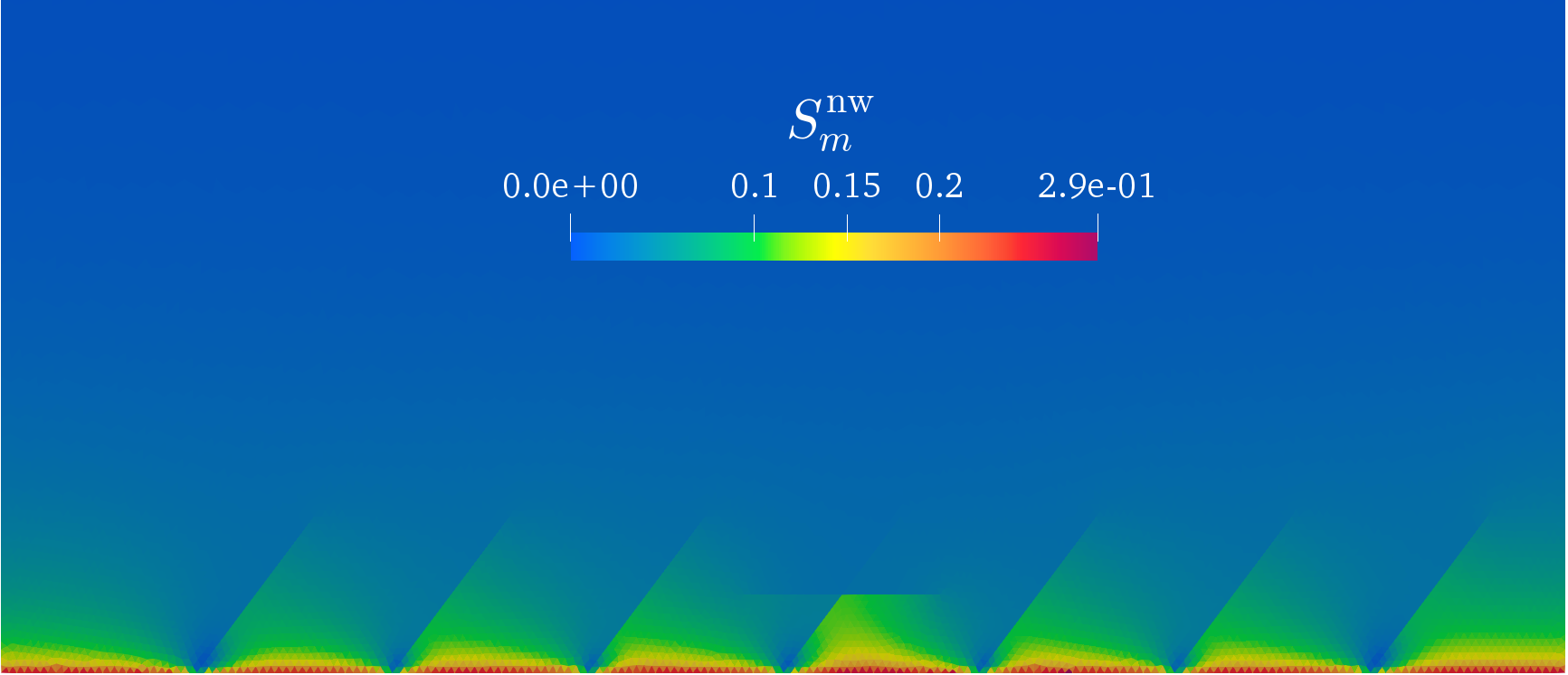}}\\
\subfloat[]{
\includegraphics[keepaspectratio=true,scale=.175]{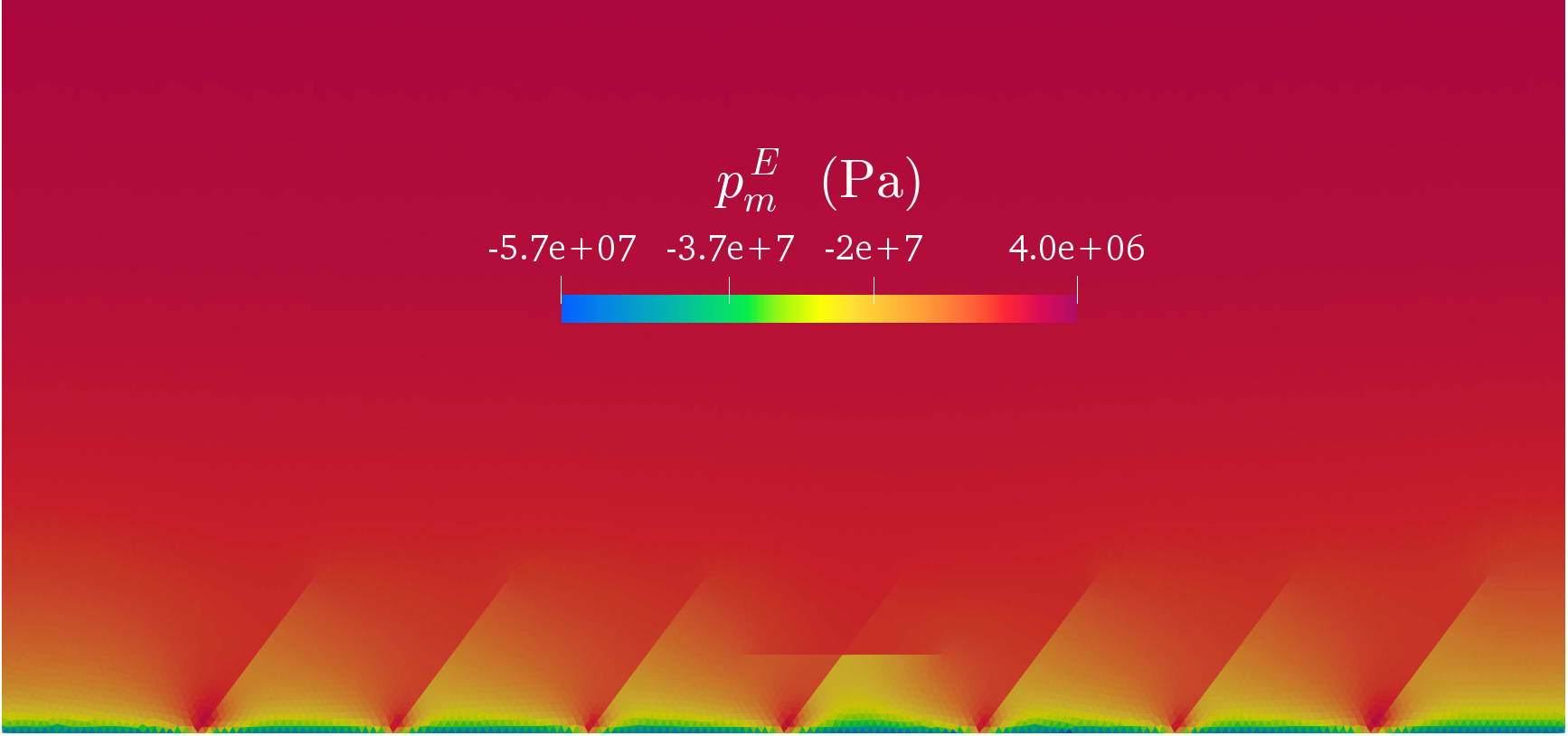}} \\
\subfloat[]{
\includegraphics[keepaspectratio=true,scale=.175]{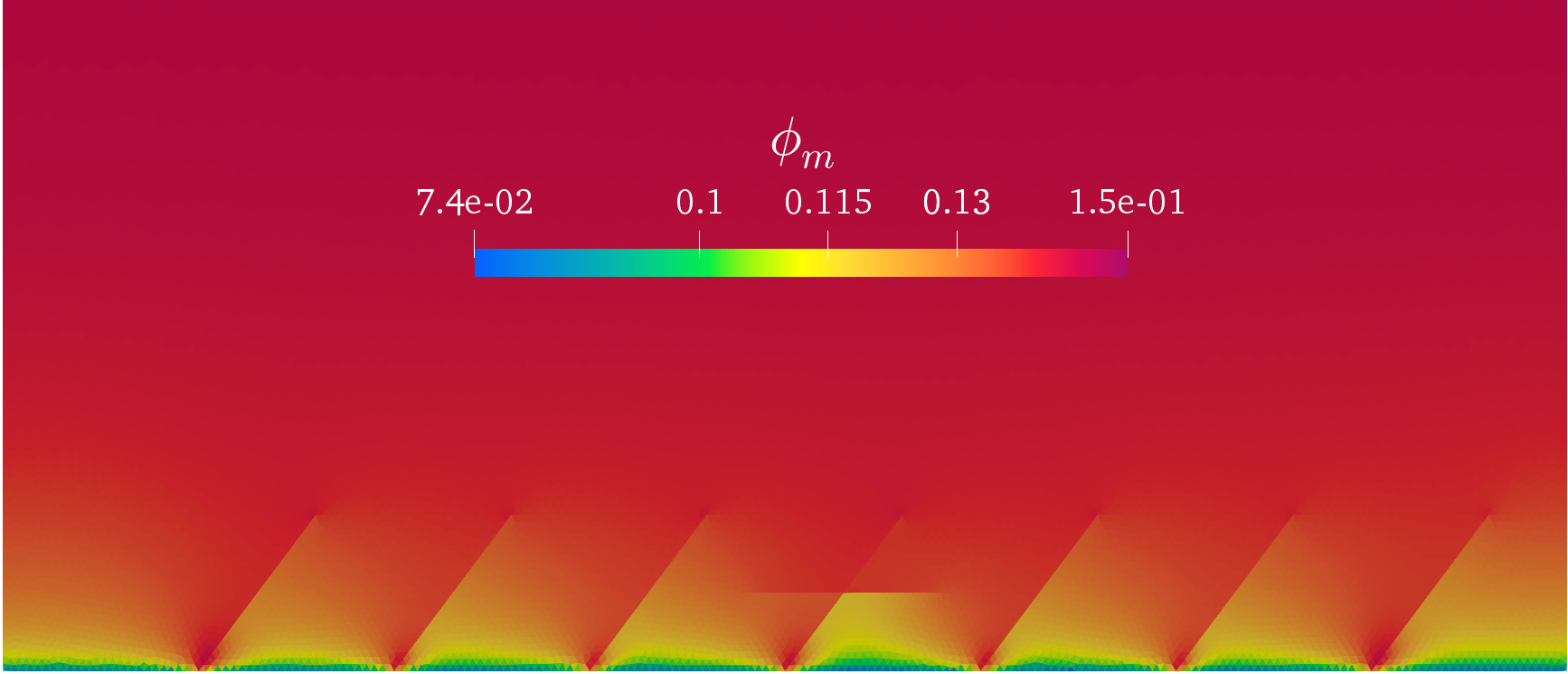}} 

\caption{From top to bottom, zooms on the gas saturation, equivalent pressure, and porosity at final time, for the example of Section~\ref{cas_andra}.}
\label{sm_pm_phim}
\end{figure}

\begin{figure}
\centering
%
\subfloat[Active set and NS Newton methods]{
\includegraphics[keepaspectratio=true,scale=.725]{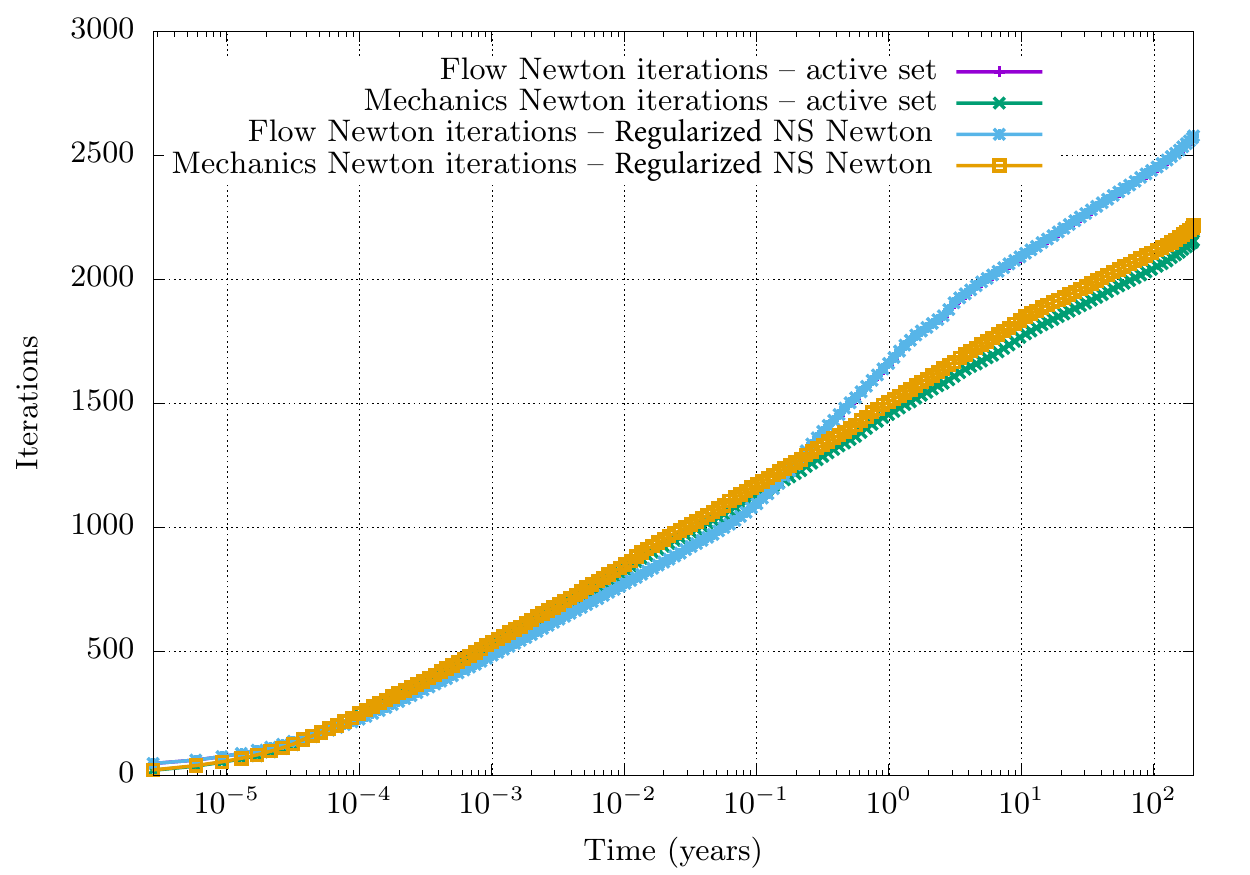}}
\subfloat[Newton--Krylov methods on $p^E$ and $\bu$]{
\includegraphics[keepaspectratio=true,scale=.725]{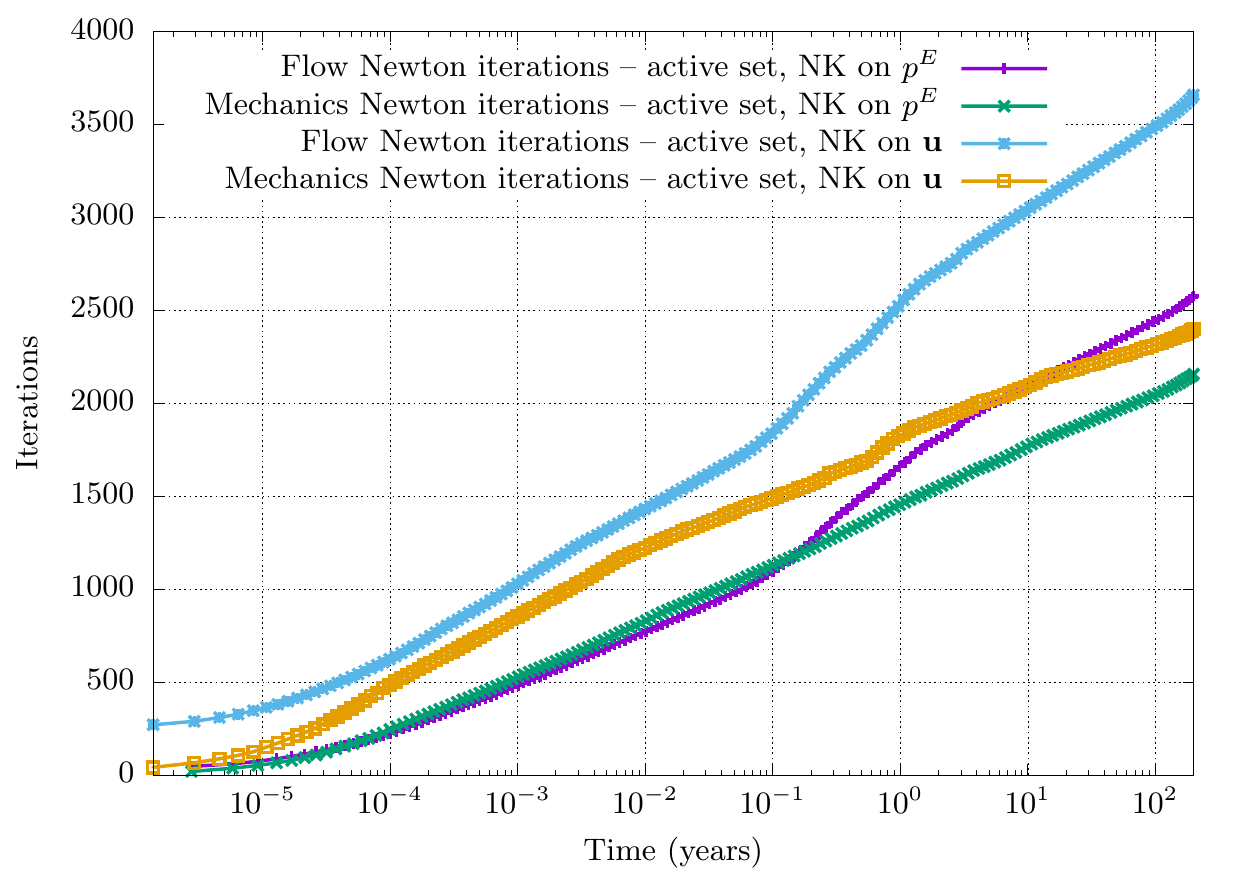}}
\caption{Total number of nonlinear flow and mechanics iterations vs.~time for the example of Section~\ref{cas_andra}. (a) Comparison of the active set and regularized non-smooth Newton algorithms for the mechanics combined with the Newton-Krylov acceleration of the fixed point $p^E = {\bf g}_p(p^E)$. (b) Comparison of the Newton-Krylov accelerations of the fixed points $p^E = {\bf g}_p(p^E)$ and $\bu = {\bf g}_\bu(\bu)$ both combined with the active set algorithm for the mechanics.}
\label{perfs_andra}
\end{figure}
\section{Conclusions}
\label{sec:conclusions}
We presented a model for a two-phase Darcy flow in a linear elastic fractured porous medium, including a Coulomb frictional contact model at matrix--fracture interfaces and assuming that phase pressures are discontinuous at such interfaces; this model generalizes the presentation in~\cite{GDM-poromeca-disc} where fractures are supposed to remain open. We applied the gradient discretization framework to introduce the general discrete counterpart of the problem, and proved its stability through suitable energy estimates, as well as the existence of a solution.

To perform numerical simulations, we discretized the mechanics using a $H^1(\Omega{\setminus}\Gamma)$-conforming scheme ($\bb P_2$ finite elements with discontinuities imposed at matrix--fracture interfaces) for the displacement coupled with fracture face-wise $\bb P_0$ Lagrange multipliers representing normal and tangential stresses for the frictional contact conditions. 
The major advantages of this mixed formulation consist in circumventing difficulties related to intersections between fractures, corners, and tips, as well as yielding a local (fracture face-wise) expression for the contact conditions.
The two-phase flow is, on the other hand, discretized via the TPFA scheme~\cite{gem.aghili}. 

Three test cases were presented, the first two to validate the pure contact mechanics model, the last one to present a realistic simulation of an axisymmetric problem involving the coupling with a two-phase flow in a radioactive waste geological storage structure, for which the data set was provided by Andra.
We compared the performances of two algorithms -- active set and regularized non-smooth Newton -- to solve the nonlinear system stemming from the contact mechanics. We also employed two versions of a Newton--Krylov method to accelerate the fixed-point algorithm to solve the coupled problem, based on the equivalent pressure and on the displacement field. It turns out that the active set method is, in general, slightly more efficient than the regularized non-smooth Newton method in terms of number of iterations.  As expected, the most efficient Newton--Krylov coupling algorithm is the one based on the equivalent pressure, compared with the displacement-based one.

Perspectives for future work include: (i)~the convergence analysis for the gradient scheme presented here, starting from the energy estimates of the discrete problem; (ii)~the extension of the implemented discretizations to more general schemes in three space dimensions; (iii)~the usage of polyhedral grids, to enable the simulation of problems set on more complex geometries (as in most real-life scenarios) while maintaining a reasonable computational cost.
%
%
\begin{acknowledgements}
We are grateful to Andra for partially supporting this work. We also thank Laurent Monasse (Inria COFFEE \& Universit\'e C\^ote d'Azur) for fruitful discussions on this work, mainly related to the discretization of contact mechanics. Finally, we thank Eirik Keilegavlen (University of Bergen) for providing us with the mesh used in the example of Section~\ref{bergen_meca}.
\end{acknowledgements}
%
%
\bibliographystyle{plain}
\bibliography{Poromeca_contact}
\end{document}